\documentclass[aap,preprint]{imsart}

\RequirePackage{amsthm,amsmath,amsfonts,amssymb}
\RequirePackage[numbers,sort&compress]{natbib}
\RequirePackage[colorlinks,citecolor=blue,urlcolor=blue]{hyperref}
\RequirePackage{graphicx}

\usepackage{tikz}
\usepackage{tkz-graph}

\usepackage{graphicx}
\usetikzlibrary{arrows.meta,positioning,fit,decorations.markings}
\usepackage{hyperref}
\usepackage{caption}
\usepackage{graphicx}
\usepackage{booktabs}
\usepackage{multirow}
\newcommand{\cl}{\mathcal}

\newcommand{\Prt}{\mathsf{P}} 
\newcommand{\wh}{\widehat}
\newcommand{\wt}{\widetilde}

\newcommand{\sbf}{\boldsymbol}
\newcommand{\mbf}{\mathbf}
\newcommand{\bb}{\mathbb}
\newcommand{\mrm}{\mathrm}
\newcommand{\E}{\mathsf{E}}

\newcommand{\EqD}{\overset{d}{=}}
\newcommand{\ConvD}{\overset{d}{\rightarrow}}
\newcommand{\ConvP}{\overset{P}{\rightarrow}}

\newcommand{\pa}{\mrm{pa}}
\newcommand{\An}{\mrm{An}}
\newcommand{\an}{\mrm{an}}
\newcommand{\de}{\mrm{de}}
\newcommand{\nd}{\mrm{nd}}
\newcommand{\De}{\mrm{De}}

\def\pp#1{ \left(#1\right) }
\def\pb#1{ \left[#1\right] }
\def\pc#1{ \left\{#1\right\} }
\DeclareMathOperator*{\argmin}{arg\,min}

\usepackage{algorithm}
\usepackage{algpseudocode}
\algrenewcommand\algorithmicrequire{\textbf{Input:}}
\algrenewcommand\algorithmicensure{\textbf{Output:}}

\usepackage{mathtools}

\startlocaldefs
\theoremstyle{plain}

\newtheorem{theorem}{Theorem}[section]
\newtheorem{lemma}[theorem]{Lemma}
\newtheorem{proposition}[theorem]{Proposition}
\newtheorem{corollary}[theorem]{Corollary}
\theoremstyle{definition}
\newtheorem{definition}[theorem]{Definition}
\newtheorem{example}{Example}

\newtheorem{remark}{Remark}%
\newtheorem{assumption}{Assumption}

\endlocaldefs

\begin{document}

\begin{frontmatter}
\title{Structural Causal Models for Extremes:  an  Approach Based on Exponent Measures}
\runtitle{Structural Causal Models for Extremes}

\begin{aug}
\author[A]{\fnms{Shuyang}~\snm{Bai} \ead[label=e1]{bsy9142@uga.edu}},
\author[B]{\fnms{Fei}~\snm{Fang}\ead[label=e2]{feifang@oakland.edu}}
\and
\author[C]{\fnms{Tiandong}~\snm{Wang}\ead[label=e3]{td\_wang@fudan.edu.cn}}
\address[A]{Department of Statistics,
University of Georgia \printead[presep={ ,\ }]{e1}}

\address[B]{Department of Decision and Information Sciences, Oakland University \printead[presep={,\ }]{e2}}

\address[C]{Shanghai Center for Mathematical Sciences,
Fudan University \printead[presep={,\ }]{e3}}
\end{aug}

\begin{abstract}
We introduce a new formulation of structural causal models for extremes, called the extremal structural causal model (eSCM). Unlike conventional structural causal models, where randomness is governed by a probability distribution, eSCMs use an exponent measure, an infinite-mass law that naturally arises in the analysis of multivariate extremes. Central to this framework are activation variables, which abstract the single-big-jump principle, along with additional randomization that enriches the class of eSCM laws.
 {The eSCM provides a common foundation for the two main existing approaches to extremal causal modeling, the max-linear and the sum-linear structural causal models, recovering both as special cases within a single asymptotic formulation. More broadly, it encompasses all possible laws of directed graphical models under the recently introduced notion of extremal conditional independence.}
We also identify an inherent asymmetry in eSCMs under natural assumptions, enabling the identifiability of causal directions, a central challenge in causal inference. Finally, we propose a method that exploits this causal asymmetry and demonstrate its effectiveness on both simulated and real datasets.
\end{abstract}



\end{frontmatter}

\section{Introduction}\label{sec:Intro}

Investigating causal relationships is a central goal in many scientific disciplines. The \emph{structural causal model (SCM)}, also known as the \emph{structural equation model}, is a widely used approach for modeling causal interactions among variables. An SCM consists of a set of equations structured according to a \emph{directed acyclic graph (DAG)} $\cl{G}=(V,E)$, where the node set $V$ indexes the variables of interest, and $E$ denotes the set of directed edges such that
\begin{equation}\label{eq:usual SCM}
X_v= f_v(  \mbf{X}_{\mathrm{pa}(v)}, e_v),\quad  v\in V. 
\end{equation} 
Each variable $X_v$ is determined by a structural function $f_v$ of its parent variables, $\pa(v)\subset V$ (nodes with edges pointing to $v$), and an exogenous noise term $e_v$. The $e_v$'s are assumed to be mutually independent. If $\mathrm{pa}(v)=\emptyset$, then $\mbf{X}_{\mathrm{pa}(v)}$ is considered absent. For comprehensive discussions of the central role SCMs play in causal modeling, see \cite{pearl2009causality,peters2017elements}.

Under certain circumstances, causal relationships are only evident at extreme values, or there is specific interest in exploring causality at these extremes. Such considerations arise in fields including finance \citep{chuang2009causality}, Earth and environmental sciences \citep{sun2021causal,mhalla2020causal}, public health \citep{chuang2009causality,chernozhukov2011inference,zhang2012causal}, genetics \citep{duncan2011genome}, and neuroscience \citep{zanin2016causality}, among others.
Recently, there has been growing interest in linking SCMs with extreme value analysis. One line of work focuses on the \emph{max-linear} structural causal model introduced in \cite{gissibl2018max}, with further developments in \cite{kluppelberg2021estimating,gissibl2021identifiability,amendola2021markov,amendola2022conditional,asenova2022max,buck2021recursive,krali2023heavy,tran2024estimating,adams2025inference,kluppelberg2025causal,amendola2025pc}. Another line is based on the heavy-tailed \emph{sum-linear} structural causal model \citep{gnecco2021causal,pasche2023causal,zhou2024efficient,krali2025causal,jiang2025separation}. A recent review \citep{chavez2024causality} summarizes these active developments in causal analysis of extremes.


In this work, we introduce a new formulation of SCMs tailored to extreme values. Specifically, we disentangle extremal causal modeling from standard SCMs by constructing models in an asymptotic regime relevant to multivariate extremes. This separation is motivated by the fact that data informative about extremal behavior typically consist of a small set of outliers, making it difficult to extrapolate causal models fitted to the bulk of the distribution into the tails. A similar perspective was recently adopted in \cite{engelke2025extremes}, and we highlight connections to that work throughout.

Unlike conventional SCMs, where randomness is governed by a joint probability distribution (e.g., the law of $\pp{X_v}_{v\in V}$ in \eqref{eq:usual SCM}), we propose the extremal structural causal model (eSCM), in which randomness is governed by an exponent measure, an infinite-mass law that naturally arises in multivariate extreme value theory (see Definition \ref{Def:MRV} below). {Though infinite in mass, the exponent measure serves as the analogue of a ``joint distribution'' that captures the joint tail dependence among multiple variables, and it is commonly treated as the target population distribution for statistical inference.} 
At the core of this formulation are activation variables, which follow infinite-mass laws and abstract the single-big-jump principle, along with additional randomization that enriches the eSCM structure. Readers may refer to Definition \ref{def:eSCM} for a quick overview.

{The eSCM framework provides a principled and unifying foundation for the two major existing approaches to extremal causal modeling, the max- and sum-linear SCMs, by embedding them into a common asymptotic setting.
Moreover, we identify a natural form of causal asymmetry in eSCMs that enables directionally identifiable causal inference. Leveraging this property, we propose a consistent causal  discovery algorithm based on estimating the support of the bivariate angular measures, efficiently capturing the underlying extremal causal order}.

The rest of the paper is organized as follows.  Section~\ref{Sec:eSCM} develops the general theory of eSCMs, beginning with their formulation, basic properties, and illustrative examples in Sections~\ref{subsec:MEVT}–\ref{sec:example}. Section~\ref{Sec:limit} shows how eSCMs can arise as limits of certain probabilistic SCMs.{Section~\ref{Sec:intervention} discusses interventions in the eSCM framework.} In Section~\ref{Sec:causal markov}, we establish the causal Markov properties of eSCMs with respect to the notion of extremal conditional independence \citep{engelke2020graphical,engelke2025graphical}. Section~\ref{Sec:causal asym} turns to causal direction learning, with Section~\ref{Sec:causal asym discussion} highlighting an inherent asymmetry under natural assumptions that enables identifiability. In Section~\ref{Sec:stats caus dir}, we propose a statistical estimator exploiting this asymmetry, which forms the basis of a consistent causal order learning algorithm detailed in Section~\ref{sec:causal order}. Section~\ref{Sec:simdata} demonstrates the effectiveness of the proposed methodology using both simulated and real data. All the proofs are deferred to Appendix \ref{Sec:proof and tech disc}.

\section{Extremal  structural causal models}\label{Sec:eSCM}
Throughout the rest of the paper, all vectors are by default column vectors. We use $\|\cdot \|$ to denote a generic norm on $\bb{R}^d$, $d\in \bb{Z}_+$, while $\|\cdot \|_p$ denotes the $p$-norm, $p\in (0,\infty]$.  For nonempty index sets $I\subset J$,  and a vector $\mbf{y}\in \bb{R}^J$, we write $\mbf{y}_I$ for the subvector of $\mbf{y}$ formed by the indices in $I$. An indicator is denoted by   $\mbf{1}_{\pp{\cdot }}$.   
  Suppose $\cl{G}=\pp{V,E}$ denotes a DAG with nonempty node set $V$ and edge set $E$. Let $v\in V$. 
  \begin{table}[h]
\centering
\caption{Graph-theoretic notation associated with a node $v$ in a DAG $(V,E)$.}\label{tab:notation}
\begin{tabular}{ll}
\toprule
\textbf{Notation} & \textbf{Description} \\
\midrule
$\pa(v)$ & Parent set of $v$: $\{\,u\in V : (u\to v)\in E\,\}$ \\[0.3em]
$\an(v)$ & Ancestor set of $v$: $\{\,u\in V : \text{there exists a directed path } u \to \cdots \to v \,\}$ \\[0.3em]
$\An(v)$ & Ancestors including $v$: $\an(v)\cup\{v\}$ \\[0.3em]
$\de(v)$ & Descendant set of $v$: $\{\,u\in V : \text{there exists a directed path } v \to \cdots \to u \,\}$ \\[0.3em]
$\De(v)$ & Descendants including $v$: $\de(v)\cup\{v\}$ \\[0.3em]
$\nd(v)$ & Non-descendants of $v$: $V\setminus \De(v)$ \\[0.3em]
$\cl{A}(v)$ & Ancestral sub-DAG of $v$: node set $\An(v)$, edge set = edges along directed paths from $\An(v)$ to $v$ \\
\bottomrule
\end{tabular}

\end{table}
{We summarize 
the notation involved in Table~\ref{tab:notation}, and give
a graphical illustration in Figure \ref{fig:DAG illu}.}
 
\begin{figure}
\tikzset{
  VertexStyle/.append style = {minimum size=18pt},
  EdgeStyle/.append style   = {->, >=Latex}
}
 \begin{tikzpicture}[scale=1] 
    \SetGraphUnit{2} 
   \Vertex{2}  
    \EA(2){3}  
    \SOWE(2){1}  
    \SOEA(3){4}  

    \NOWE(2){5}   
    \NOEA(3){6}  

    \Edges(1,2,3,4)  
    \Edges(5,2)        
    \Edges(3,6) 
    \end{tikzpicture}
    \caption{Illustration of DAG notation. $\pa(2)=\{1,5\}$. $\an(3)=\{1,2,5\}$. $\An(3)=\{1,2,3,5\}$. $\de(3)=\{4,6\}$.  $\De(3)=\{3,4,6\}$. $\nd(3)=\{1,2,5\}$. 
    The    sub-DAG $\cl{A}(4)$  consists of the node set $\An(4)=\{1,2,3,4,5\}$ and the   edge set $\{(1\rightarrow 2), (5\rightarrow 2),  (2\rightarrow 3), (3\rightarrow 4)\}$.  
    }\label{fig:DAG illu}
\end{figure}

\subsection{Background on multivariate extremes and exponent measure}
\label{subsec:MEVT}
We start by recalling some important concepts  from the multivariate extreme value theory that will be used throughout the rest of the paper. We refer to \cite{beirlant2006statistics, resnick2007heavy} for more details.

Suppose $\mathbf{X}=(X_v)_{v\in V}\in [0,\infty)^V$ is a $d$-dimensional   random vector indexed by $V:=\{1,\ldots,d\}$ with continuous marginal distributions. 
Each coordinate $X_v$ represents one component of a multivariate sample. We focus on the nonnegative orthant suitable for analyzing one-sided extremes, which is widely encountered in practice, although extensions to two-sided extremes can be naturally achieved. 
 As a common practice in the analysis of multivariate extremes, 
 we assume that the marginal distribution of $\mathbf{X}$ satisfies
 \begin{equation}\label{eq:pareto marginal} 
 \lim_{x\rightarrow\infty} x^{\alpha}\Prt(X_v>x)= s_v,\qquad  v=1,\ldots,d,
\end{equation} 
where $\alpha>0$, and $s_v\in (0,\infty)$ is a constant. 
Also note that for {data not satisfying the marginal assumption \eqref{eq:pareto marginal}  such as  light-tailed data}, we may apply the transformation
\begin{equation}\label{eq:marg trans}
X_v\mapsto \pb{1- F_v(X_v)}^{-1/\alpha},
\end{equation}
where $F_v$ denotes the marginal CDF of $X_v$, $v\in V$, to obtain standard $\alpha$-Pareto marginals. In practice, $F_v$ will be replaced by its empirical counterpart. {Furthermore, in our empirical studies, we set $\alpha = 2$ when applying the transform in \eqref{eq:marg trans}.}


 Now we introduce the concept of multivariate regular variation (MRV), which is a key assumption for analysis of joint tail behaviors; {see for instance \cite[Chapter 6]{resnick2007heavy}.}
\begin{definition}\label{Def:MRV}
Let $\mbf{0}_V$ be the origin in $[0,\infty)^V$, and $\overset{\mathbf{v}}{\rightarrow}$ denote the vague convergence (see, e.g., \cite[Appendix B]{kulik2020heavy}) of  measures  on $$\bb{E}_V:= [0,\infty)^V\setminus \{\mathbf{0}_V\},$$ then $\mathbf{X}$ is said to be multivariate regularly varying (MRV) if 
\begin{equation}\label{eq:Lambda}
t \Prt\left(  t^{-1/\alpha} \mathbf{X}   \in  \cdot  \right) \overset{\mathbf{v}}{\rightarrow}  \Lambda(\cdot), \quad \text{as }t\rightarrow\infty,
\end{equation}
where $\Lambda$ is an infinite measure defined on the Borel $\sigma$-field of $\bb{E}_V$  that is finite on any Borel set separated from $\mbf{0}_V$ (i.e., $\mbf{0}_V$ does not belong to its closure in $[0,\infty)^V$), known as the \emph{exponent measure}.
\end{definition}
 

{In particular, the convergence in \eqref{eq:Lambda}  is characterized by $\lim_{t\rightarrow\infty} t\Prt\pp{t^{-1/\alpha} \mbf{X}\in B}=\Lambda(B)$ for Borel $B$ separated from $\mbf{0}_V$ with $\Lambda\pp{\partial B}=0$, where $\partial B$ denotes the  boundary of $B$.  The limit measure \( \Lambda(B) \) can thus be interpreted (up to scaling) as capturing the asymptotic probability that \( \mathbf{X} \) falls into the extreme region \( t^{1/\alpha} B \) for large $t$.  
The   limit relation  \eqref{eq:Lambda} also implies that the exponent measure $\Lambda$ satisfies the homogeneity property: 
\begin{equation}\label{eq:Lambda homo}
    \Lambda(c \, \cdot )=c^{-\alpha}\Lambda(\cdot),\quad c>0.
\end{equation}
It also follows from \eqref{eq:pareto marginal} that
\begin{equation}\label{eq:Lambda marginal}
\Lambda(\{\mbf{y}\in \bb{E}_V:\  y_v>1\})=s_v\in (0,\infty), \quad v\in V,
\end{equation}
 The fact that $\Lambda\pp{ \bb{E}_V}=\infty$ can be inferred from \eqref{eq:Lambda homo} and \eqref{eq:Lambda marginal}. }
Conversely, any Borel measure $\Lambda$ on $\bb{E}_V$ that satisfies \eqref{eq:Lambda homo} and \eqref{eq:Lambda marginal} 
is an exponent measure which arises from \eqref{eq:Lambda} for some multivariate regularly varying $\mbf{X}$ satisfying \eqref{eq:pareto marginal}.  
As is common in the literature, one may also include slowly varying functions (such as logarithmic terms) in the scaling relations \eqref{eq:pareto marginal} and \eqref{eq:Lambda}. {In this work, however, we exclude such factors for simplicity in order to focus on the core ideas.}

Another key concept for describing extremal dependence structures is
\emph{extremal independence}; {see for example \cite[Section 2.1.2]{kulik2020heavy}}. 
\begin{definition}\label{Def:ext indep}
The exponent measure $\Lambda$ is said to be (component-wise) extremaly independent, if $\Lambda$  concentrates on the coordinate axes $$\bb{A}_V:=\{\mbf{y}\in \bb{E}_V:\ y_v>0 \text{ for exactly one }v=1,\ldots,d \},$$ or equivalently, $\Lambda(y_u>0,y_v>0)=0$ for any distinct $u,v\in V$.    
\end{definition}
 Extremal independence  can also be characterized by the bivariate tail dependence coefficients:  $\lim_{x\rightarrow\infty} \Prt(X_u>x|X_v>x)=0$ for any distinct $u,v\in V$, where $\mbf{X}=(X_v)_{v\in V}$  is related to $\Lambda$ as in \eqref{eq:Lambda}. 
{While pairwise probabilistic independence between \( X_u \) and \( X_v \) implies   extremal independence,  at the level of the exponent measure \( \Lambda \), the notion of extremal independence is fundamentally different in nature from classical probabilistic independence. In particular, extremal independence does not correspond to a product measure factorization of \( \Lambda \).}
The intuition behind extremal independence connects to the well-known ``single big jump principle'' for heavy-tailed distributions: when the vector exhibits an extreme, it is because one component is extreme and others are not, rather than multiple components being large together.

\subsection{The formulation of extremal structural causal model}
As mentioned before, the exponent measure $\Lambda$ in  \eqref{eq:Lambda}, albeit an infinite measure,  may be viewed as the  ``extremal distribution'' of sample $\mbf{X}$. We therefore regard an exponent measure as the joint law governing the extremal causal structural model to be formulated. Motivated by \eqref{eq:usual SCM}, we consider replacing the independent random variables $(e_v)_{v\in V}$ with those exhibiting extremal independence as defined in Definition \ref{Def:ext indep}, which we now explain.


Let $\Lambda^{\perp}$ denote the exponent measure on $\bb{E}_V$  such that 
\begin{equation}\label{eq:Lambda_perp}
\Lambda^{\perp}\pp{\{\mbf{y}\in \bb{E}_V:\ y_v>y \}}=s y^{-\alpha}, \ s>0,\ 
v=1,\ldots,d,\text{  and  } \Lambda^{\perp}\pp{\bb{E}_V\setminus \bb{A}_V }=0.
\end{equation}
The infinite measure $\Lambda^{\perp}$ may be interpreted as the joint law of extremally independent and identically distributed improper random variables with improper Pareto marginals, the latter arising as a direct consequence of the homogeneity condition~\eqref{eq:Lambda homo}.
A simple example satisfying \eqref{eq:Lambda_perp} in terms of the limit relation \eqref{eq:Lambda} is  $\mbf{X}$ consisting of extremally independent components $X_v$ with $P(X_v>x)\sim s x^{-\alpha}$, $x\rightarrow\infty$. 

To formulate extremal structural causal models with rich exponent measure laws, it turns out that we need extra randomness in addition to $\Lambda^{\perp}$ (see the discussion around \eqref{eq:Y_1 -> Y_2} below). Let $\Prt_{\sbf{\theta}}$ denote the joint law on $[0,1]^V$ of a $d$-dimensional random vector with i.i.d.\ Uniform$(0,1)$ components. Note that
the choice of Uniform$(0,1)$ as the randomization distribution is without loss of generality, since any probability distribution can be obtained from a uniform distribution via the inverse transform of the CDF. 

  Now we introduce improper random variables, termed as \emph{activation variables},  denoted by $\boldsymbol{\eta}=(\eta_v)_{v\in V}$, which are jointly distributed according to $\Lambda^{\perp}$. Furthermore, let $\boldsymbol{\theta}=\pp{\theta_v}_{v\in V}$ be a random vector consisting of i.i.d.\  Uniform$(0,1)$ random variables that are ``independent'' of $\boldsymbol{\eta}$.  Formally, this means 
$(\boldsymbol{\eta},\boldsymbol{\theta})$ is measurable map from an underlying (infinite) measure space $(\Omega,\mathcal{F},\mu)$ to $\bb{E}_V\times [0,1]^V$, such that the push-forward measure
$\mu\pp{ \pp{\boldsymbol{\eta},\boldsymbol{\theta} }\in \cdot }=(\Lambda^{\perp}\otimes \Prt_{\sbf{\theta}})(\cdot)$, where $\Lambda^{\perp}\otimes \Prt_{\sbf{\theta}}$ denotes the  product measure on $\bb{E}_V\times [0,1]^V$.   
As a canonical choice, one may take $(\Omega,\mathcal{F},\mu) = (\mathbb{E}_V\times [0,1]^V ,\mathcal{B},\Lambda^{\perp}\otimes \mathrm{Leb}^d)$, where $\mathcal{B}$ denotes the Borel $\sigma$-field of $\mathbb{E}_V\times [0,1]^V$, and $(\boldsymbol{\eta},\sbf{\theta})$ is taken as the identity map on $\mathbb{E}_V\times [0,1]^V$.

Next, we write $\mathbf{Y} = (Y_v)_{v\in V}$ for the \emph{extremal variables}, which may be viewed as the extremal counterpart of the usual sample variables $\mathbf{X}=(X_v)_{v\in V}$ in \eqref{eq:usual SCM}, which are improper random variables governed by an infinite measure. In what follows, we formulate a causal structural model for $\mathbf{Y}$. Specifically, the definition will describe $\mathbf{Y}$ as a measurable function of $(\boldsymbol{\eta},\boldsymbol{\theta})$ through recursive relations analogous to \eqref{eq:usual SCM}, thereby yielding $\mathbf{Y}$ as a measurable map from $(\Omega,\mathcal{F},\mu)$ to $[0,\infty)^V$ (see also \eqref{eq:Y=F(eta,theta)} below).

\begin{definition}[eSCM]\label{def:eSCM}
Let $\cl{G}=(V=\{1,\ldots,d\}, E)$, $d\in\bb{Z}_+$ be a DAG. Suppose $\sbf{\eta}=\pp{\eta_v}_{v\in V}$ and $\sbf{\theta}=\pp{\theta_v}_{v\in V}$ are respectively activation variables and Uniform$(0,1)$ variables defined on an underlying measure space $(\Omega,\cl{F},\mu)$ as described above. 
An eSCM associated with the DAG $\cl{G}$ is given by
  \begin{equation}\label{eq:general eSCM}
Y_v= f_v(  \mbf{Y}_{\mathrm{pa}(v)},\eta_v,  \theta_v):= a_v  \eta_v +  h_v\pp{ \mbf{Y}_{\pa(v)},\theta_v }  
,\    v\in V=\{1,\ldots,d\}, 
\end{equation}
where the nonrandom coefficient $a_v\in [0,\infty)$, and each  $h_v:[0,\infty)^{\mathrm{pa}(v) }\times [0,1]  \mapsto[0,\infty)$ is a measurable function such that: 
\begin{enumerate}
    \item  $ h_v(c\mbf{y}_{\mrm{pa}(v)}, \theta )=c h_v(\mbf{y}_{\mrm{pa}(v)}, \theta)$ for any  $\theta\in [0,1]$, $\mbf{y}_{\mrm{pa}(v)}\in[0,\infty)^{\mathrm{pa}(v) }$ and  $c\in[0,\infty)$;
    \item  $\mu(Y_v>1)\in (0,\infty)$     for all $v\in V$. 
\end{enumerate}
 In \eqref{eq:general eSCM}, we refer to $a_v$ as the \emph{activation coefficient}, $h_v$  the \emph{proper structural function}, and $f_v$ the \emph{total structural function} associated with node $v$. In addition,  the law $\cl{L}(\mbf{Y})$ refers to the push-forward measure $\mu(\mbf{Y}\in\cdot )$ restricted to $\bb{E}_V$.
\end{definition}
Condition 1 guarantees the homogeneity property of the exponent measure $\Lambda=\cl{L}(\mbf{Y})$ in \eqref{eq:Lambda homo} holds, {as clarified in Proposition \ref{Pro:basic} below}. Note that when $c=0$, we have $h_v(\mbf{0},\theta)=f_v(\mbf{0},0,\theta)=0$ for $\theta\in [0,1]$. Since $\mbf{Y}_{\pa(v)}$ does not depend on $\eta_v$, the two terms $a_v\eta_v$ and $h_v\pp{ \mbf{Y}_{\pa(v)},\theta_v }$ cannot be simultaneously nonzero due to the nature of $\sbf{\eta}$; see   the discussion below \eqref{eq:Y=F(eta,theta)}.

Condition 2 ensures non-trivial marginal laws, and the restriction of $\mu(\mbf{Y}\in\cdot )$ to $\bb{E}_V$ in Definition~\ref{def:eSCM} is imposed to exclude the origin $\mbf{0}_V$, as required by the definition of an exponent measure.  Moreover, it is possible to have  $\mu(\mbf{Y}=\mbf{0}_V)>0$, and detailed discussion is deferred to Section \ref{sec:law eSCM}.


 {\begin{remark}[Relation to \cite{engelke2025extremes}]\label{Rem:engelke}
The idea of modeling causality directly in the extremal limit, rather than extrapolating a structural model fitted to the bulk, was recently advanced in \cite{engelke2025extremes}, and the eSCM in Definition \ref{def:eSCM} shares this motivation. The two frameworks nevertheless differ in scope. First, \cite{engelke2025extremes} concentrates on a single activation at a unique root node, under which the resulting law is absolutely continuous on $\bb{E}_V$; the H\"usler--Reiss model of Example \ref{Eg:HR eSCM} is of this type and arises as the special case of an eSCM with a single root and a single nonzero activation coefficient. However, an eSCM admits multiple root nodes and multiple nonzero activation coefficients, yielding laws that are generally mixtures supported on lower-dimensional coordinate faces (cf. \eqref{eq:mix act}); this added generality is precisely what is needed for the existence results in Theorem \ref{Thm:Lambda expr by eSCM} and Corollary \ref{cor:equiv}. Second, the additional randomizers $\sbf{\theta}$ enrich the class of attainable exponent measures beyond mass concentration on finitely many rays, which is essential for the causal asymmetry exploited in Section \ref{Sec:causal asym discussion}. Third, our treatment of interventions in Section \ref{Sec:intervention} introduces an explicit intervention node within the eSCM itself, whereas \cite{engelke2025extremes} obtains the analogous phenomenon through a limiting argument from an intervened probabilistic SCM (their Theorem~2). Finally, we work throughout on the Pareto marginal scale, whereas \cite{engelke2025extremes} adopts the exponential scale; the two are related by exponentiation, as noted in Theorem \ref{Thm:limit eSCM}.
\end{remark}}

\begin{remark}\label{Rem:red f_v}
 One may assume a more general form of $f_v$ than \eqref{eq:general eSCM}, i.e. $f_v: [0,\infty)^{\mathrm{pa}(v)} \times [0,\infty) \times [0,1]  \mapsto[0,\infty)$ that satisfies $f_v(c \mbf{y},c\eta,\theta)=cf_v( \mbf{y},\eta,\theta)$ for any $\theta\in [0,1]$, $\mbf{y}\in[0,\infty)^{\mathrm{pa}(v) }$ and $c\in [0,\infty)$. However, we argue that it effectively reduces to the form \eqref{eq:general eSCM}. When $\eta_v>0$, $\eta_u=0$ for $u\in \an(v)$, which implies $\mbf{Y}_{\an(v)}=\mbf{0}_{\an(v)}$; see the discussion below \eqref{eq:Y=F(eta,theta)}. Therefore, by the homogeneity property, we have
\[
f_v\pp{ \mbf{Y}_{\mathrm{pa}(v)},\eta_v,  \theta_v}= \eta_v f_v\pp{ \mbf{0}_{\mathrm{pa}(v)},1,  \theta_v} + f_v\pp{ \mbf{Y}_{\mathrm{pa}(v)}, 0,  \theta_v} \mbf{1}_{\pc{\eta_v=0}}.
\]
The second term above can be viewed as the $h_v(\cdot)$ function in \eqref{eq:general eSCM}. 
For the first term, let $A_v=f_v\pp{ \mbf{0}_{\mathrm{pa}(v)},1,  \theta_v}$, we then have by Fubini that $\mu(A_v\eta_v>y)=s y^{-\alpha} \E_{\sbf{\theta}}[ A_v^\alpha ]$, $y>0$, where $\E_{\sbf{\theta}}$  denotes the expectation with respect to $\Prt_{\sbf{\theta}}$. Hence, as long as $\E_{\sbf{\theta}}[ A_v^\alpha ]<\infty$, the law of $\mbf{Y}$ remains unchanged if $A_v$ is replaced by 
$
a_v:=\pp{\E_{\sbf{\theta}}[ A_v^\alpha ]}^{1/\alpha}.
$
\end{remark}

\begin{remark}\label{Rem:PPP int eSCM}
Another instructive way to interpret   eSCMs governed by infinite-mass laws is through a Poisson point process. One may regard  a sample $\mbf{Y}_i$ of an eSCM \eqref{eq:general eSCM} as a point from the Poisson point process $\sum_{i=1}^\infty {\delta_{\mbf{Y}_i}}$ with mean measure $\cl{L}(\mbf{Y})$ on $\bb{E}_V$, which is the weak limit of a rescaled empirical point process $\sum_{i=1}^n \delta_{\mbf{X}_i/n^{1/\alpha}}$ as $n\rightarrow\infty$, and $\{\mbf{X}_i: i\ge 1\}$ are i.i.d.\ samples from $\mbf{X}$ (see for instance \cite[Theorem 6.2]{resnick2007heavy}). Hence, the eSCM \eqref{eq:general eSCM} describes a relation that approximately governs the rescaled sample points $\mbf{Y}_i\approx\mbf{X}_i/n^{1/\alpha}$ for those extremal $\mbf{X}_i$'s whose magnitudes are of order $n^{1/\alpha}$. {See also Corollary \ref{Cor:PPP eSCM} below.}
\end{remark}

Next, we highlight the importance of including the randomizers  $\pp{\theta_v}_{v\in V }$ in eSCMs. We say an eSCM in Definition \ref{def:eSCM} is \emph{simple}, if the proper structural functions $h_v$ in \eqref{eq:general eSCM} does not depend on the randomizer $\theta_v$ for all $v\in V$.
 Then consider the following simple eSCM, corresponding to the DAG $V=\{1,2\}$ and $E=\{1\rightarrow 2\}$:
\begin{equation}\label{eq:Y_1 -> Y_2}
    Y_1 = \eta_1,\quad 
     Y_2   = \beta Y_1 +  \eta_2, 
    \quad \beta>0.
\end{equation} 
Its exponent measure law concentrates only on two directions: the ray $\{y_2=\beta y_1\}$ direction when $\eta_1$ is active (i.e., becomes nonzero), and the $y_2$-axis direction when $\eta_2$ is active. See the left panel of Figure \ref{fig:sim eSCM vs gen eSCM} for a graphical illustration. 
However, a randomized  $\beta=\beta(\theta_2)$ in \eqref{eq:Y_1 -> Y_2}, if distributed on an interval with a continuous distribution, may induce a continuum of directions $\{y_2=\beta(\theta_2) y_1\}$ (cf. the right panel of Figure \ref{fig:sim eSCM vs gen eSCM}).


\begin{figure}[h]
    \centering

        



            

    \begin{tikzpicture}[scale=0.5]

    \draw[->] (0,0) -- (5,0) node[anchor=north] {$y_1$};
    \draw[->] (0,0) -- (0,5) node[anchor=east] {$y_2$};
    
    \draw[-, line width=0.6mm, red] (0,0) -- (0,4.9);
    \draw[-, line width=0.6mm, red] (0,0) -- ({4.9*cos(30)}, {4.9*sin(30)});

    \begin{scope}[xshift=14cm]

        \draw[->] (0,0) -- (5,0) node[anchor=north] {$y_1$};
        \draw[->] (0,0) -- (0,5) node[anchor=east] {$y_2$};

        \draw[-, line width=0.6mm, red] (0,0) -- (0,4.9);

        \fill[red, opacity=0.3] (0,0) -- ({4.9*cos(20)}, {4.9*sin(20)}) 
                               arc[start angle=20, end angle=90, radius=4.9] 
                               -- cycle;

    \end{scope}
\end{tikzpicture}
    \caption{\footnotesize{Illustration of the law of $(Y_1,Y_2)$ in \eqref{eq:Y_1 -> Y_2} when $\beta$ is fixed (left) v.s. when it is randomized (right). A thick solid line denotes a  mass concentration, whereas the shaded cone illustrates randomization.}  
    }
    \label{fig:sim eSCM vs gen eSCM}
\end{figure}



{
Before discussing further properties of eSCMs, we briefly comment on the rationale for formulating them on an infinite measure space. Although the exponent measure $\Lambda$ is conceptually fundamental, much of the literature on multivariate extremes works instead with probability distributions characterizing $\Lambda$, such as the angular measure on the unit sphere $\mathbb{S}_+^{V}:=\{\mbf{y}\in[0,\infty)^V:\|\mbf{y}\|=1\}$ 
and the multivariate Pareto threshold exceedance law  given by the normalized probability measure $\Lambda(\cdot\cap\bb{L})/\Lambda(\bb{L})$ on $\bb{L}:=\{\mbf{y}\in[0,\infty)^V:\|\mbf{y}\|_\infty\ge 1\}$.} 

In principle, eSCMs can also be formulated on such subspaces, thereby working with probabilistic random variables rather than improper ones. However, despite the tradeoff of dealing with an infinite measure, working directly on $\bb{E}_V$ offers both intuitive clarity and mathematical elegance. For instance, the function $h_v$ in \eqref{eq:general eSCM} is easily specified with only a homogeneity requirement, whereas additional constraints are required in the subspace formulations. Moreover, all reasoning and definitions of eSCMs can be developed within a unified framework on $\bb{E}_V$, while still allowing transitions to subspace representations; see, for example, Section~\ref{sec:pf mult Pareto HR} of the Appendix.

In the sequel, although a complete description of an eSCM involves the data \\$(\mbf{Y},\cl{G},\sbf{\eta},\sbf{\theta},(\Omega,\cl{F},\mu),(a_v)_{v\in V}, (h_v, v\in V))$ in Definition \ref{def:eSCM}, we shall simply use the extremal variable symbol $\mbf{Y}$ to refer to an eSCM.

\subsection{Basic properties of the Law of eSCM}\label{sec:law eSCM}

By a recursion of \eqref{eq:general eSCM} tracing back through ancestral relations, we have 
\begin{equation}\label{eq:Y=F(eta,theta)}
    \mbf{Y}=\pp{Y_v}_{v\in V}=\mbf{F}_{\cl{G}}(\boldsymbol{\eta},\sbf{\theta}):= \pp{F_{\cl{A}(v)}\pp{ \sbf{\eta}_{\mrm{An}(v)}, \sbf{\theta}_{\mrm{An}(v)} }}_{v\in V},
\end{equation}
for some measurable functions $F_{\cl{A}(v)}:[0,\infty)^{\mrm{An}(v)}\times [0,1]^{\mrm{An}(v)} \mapsto [0,\infty)$ such that $$F_{\cl{A}(v)}\pp{ c \, \cdot , \sbf{\theta}_{\mrm{An}(v)} }= c F_{\cl{A}(v)}\pp{ \,   \cdot \, , \sbf{\theta}_{\mrm{An}(v)} },$$ for any $c\ge 0$, $\sbf{\theta}_{\mrm{An}(v)}\in [0,1]^{ \mrm{An}(v)}$, $v\in V$. This, in particular, implies that $F_{\cl{A}(v)}\pp{ \mbf{0}_{\An(v)},  \cdot }\equiv 0$. In Proposition \ref{Pro:basic} below, we give a moment-type characterization of Condition 2 in Definition \ref{def:eSCM}, as well as the confirmation of $\cl{L}(\mbf{Y})$ as an exponent measure in the sense of Section \ref{subsec:MEVT}. 

\begin{proposition}\label{Pro:basic}
 Following the construction in Definition \ref{def:eSCM}, we have 
\begin{equation}\label{eq:s_v moment}
 s_v:= \mu(Y_v>1)=s\sum_{u\in  \mrm{An}(v)}\E_{\sbf{\theta}}\pb{F_{\cl{A}(v)}\pp{ \pp{\mbf{1}_{\{w=u\}}}_{w\in \mrm{An}(v)}, \sbf{\theta}_{\mrm{An}(v)} }^\alpha}, ~\ v\in V,
 \end{equation}
 where $s>0$ is as in \eqref{eq:Lambda_perp},    $\E_{\sbf{\theta}}$   denotes the expectation with respect to $\Prt_{\sbf{\theta}}$.
 In addition,   the law $\Lambda=\cl{L}(\mbf{Y})$ is an exponent measure that satisfies \eqref{eq:Lambda homo} and \eqref{eq:Lambda marginal} with $s_v$ as in \eqref{eq:s_v moment}. Moreover, a sufficient condition for $s_v<\infty$ for all $v\in V$ is that $h_v\pp{ \mbf{Y}_{\pa(v)},\theta_v }\le C_v(\theta_v) \|\mbf{Y}_{\pa(v)}\|$ $\mu$-a.e.\ for some measurable $C_v:[0,1]\mapsto [0,\infty)$ such that $\E_{\sbf{\theta}} |C_v(\theta_v)|^{\alpha}<\infty$, for all $v\in V$.
\end{proposition}
  As a consequence of the homogeneity property of $\cl{L}(\mbf{Y})$, we also have 
$$
\mu(Y_v>y) =s_v y^{-\alpha}, \quad  y\in (0,\infty).
$$
Furthermore, the single-activation nature of $\sbf{\eta}$ induces a decomposition of $\Lambda$.
Given a DAG $\cl{G}=(V,E)$ and a node $v\in V$, let $\de(v)$ denote the set of descendants of $v$, i.e. nodes that $v$ can reach through directed paths. 
In addition, we use $\cl{D}(v)$ to denote the descendant sub-DAG formed by the node set $\De(v)$ and the edge set consisting of the edges of all directed paths from $v$ to $\de(v)$.

On the event $\{\eta_v>0\}$, $v\in V$, since $\eta_w=0$ for any $w\neq v$, we see that $\mbf{Y}_{\nd(v)} =\mbf{0}_{\nd(v)}$ in view of \eqref{eq:Y=F(eta,theta)}. Therefore, on $\{\eta_v>0\}$, 
\begin{equation}\label{eq:single eta}
Y_v=a_v \eta_v ,\quad Y_u= h_u\pp{\pp{\mbf{Y}_{\pa(u)\cap \De(v) }, \mbf{0}_{\pa(u)\cap \nd(v)} }, \theta_u},  \quad u\in \de(v),
\end{equation}
where $h_u$ is specified in \eqref{eq:general eSCM}. The vector 
$\big(\mathbf{Y}_{\pa(u)\cap \De(v)}, \, \mathbf{0}_{\pa(u)\cap \nd(v)} \big)$ 
denotes the $\pa(u)$-indexed vector whose $(\pa(u)\cap \De(v))$-components are given by 
$\mathbf{Y}_{\pa(u)\cap \De(v)}$, while the remaining components corresponding to 
$\pa(u)\cap \nd(v)$ are set to $0$.
Equation \eqref{eq:single eta} explains that on $\{\eta_v>0\}$ with $a_v>0$, the eSCM essentially reduces to a sub-eSCM indexed by the descendant sub-DAG $\cl{D}(v)$ with a single root $v$.  Therefore, the total eSCM can be viewed as a mixture of sub-SCMs induced by these activations. In particular, $\Lambda=\cl{L}(\mbf{Y})$ can be decomposed as  
\begin{equation}\label{eq:mix act}
\Lambda=\sum_{v\in V} \Lambda_v=\sum_{v\in V, a_v> 0} \Lambda_v,
\end{equation}
where $\Lambda_v:=\mu(\mbf{Y}\in \cdot,\ \mbf{Y}\neq \mbf{0}_V ,\ \eta_v>0 )$ is supported on the coordinate face 
$ 
\{\mbf{y}\in \bb{E}_V: \  \mbf{y}_{\nd(v)}=\mbf{0}_{\nd(v)}  \}$. 

To understand the second equality in \eqref{eq:mix act}, consider the case where $a_v=0$ for some $v\in V$. This cannot happen if $\pa(v)=\emptyset$, e.g., if $v$ is a root node in $\cl{G}$ or $v$ is an isolated node, since otherwise one would have $Y_v\equiv 0$, contradicting Condition 2 in Definition \ref{def:eSCM}. 
Then assume $\pa(v)\neq\emptyset$ and $a_v=0$. In this case, $Y_v>0$ is possible only when $\mbf{Y}_{\pa(v)}\neq \mbf{0}_{\pa(v)}$, which requires $\eta_u>0$ for some $u\in \an(v)$.
Therefore, on $\{\eta_v>0\}$, we have $\mbf{Y}=\mbf{0}_V$.
Furthermore, $\cl{L}(\mbf{Y})$ excludes the origin $\mbf{0}$, so that when $a_v=0$, we do not observe $\{\eta_v>0\}$ from 
$\cl{L}(\mbf{Y})$, and the associated component $\Lambda_v$  in \eqref{eq:mix act} is zero. {On the other hand, allowing $a_v=0$ adds flexibility to the law $\mathcal{L}(\mathbf{Y})$: in this case, $Y_v$ is nonzero only through input from its parent nodes. This feature is indispensable for the existence results in Theorem~\ref{Thm:Lambda expr by eSCM} and Corollary~\ref{cor:equiv}.}

Furthermore, the decomposition in \eqref{eq:mix act} also reveals that $\cl{L}(\mbf{Y})$ governed by an eSCM is typically \emph{not} absolutely continuous (thus it does not admit a density) throughout $\bb{E}_V$, but rather possibly a mixture of laws that are absolutely continuous with respect to lower-dimensional Lebesgue measure on coordinate faces.
 A noteworthy exceptional case occurs when the DAG $\cl{G}$ has only a single root node with single nonzero activation coefficient, as was essentially considered in \cite{engelke2025extremes}.

\subsection{Examples}\label{sec:example}
We now give some concrete examples of   eSCMs. Consider the simple sum- and max-linear eSCMs, whose proper structural functions $h_v$ in \eqref{eq:general eSCM} are given by
\begin{equation}\label{eq:Lin-eSCM}
h_v(\mbf{y}_{\mrm{pa}(v)},\theta_v)=\sum_{u\in \mrm{pa}(v)}   \beta_{uv}  y_u 
\end{equation}
and  
\begin{equation}\label{eq:MaxLin-eSCM}
h_v(\mbf{y}_{\mrm{pa}(v)},\theta_v)=\bigvee_{u\in \mrm{pa}(v)}   \beta_{uv}  y_u,
\end{equation}
respectively, with coefficients $\beta_{uv}  \in (0,\infty)$, and $\mbf{y}_{\pa(v)}\in [0,\infty)^{\pa(v)}$, $v\in V$.   Equations~\eqref{eq:Lin-eSCM} and \eqref{eq:MaxLin-eSCM} correspond to non-extremal SCMs considered in \cite{gnecco2021causal} and \cite{gissibl2018max}, respectively. In fact, the law $\cl{L}(\mbf{Y})$ given by these eSCMs arises exactly through the scaling relation \eqref{eq:Lambda} when $\mbf{X}$ is given by the SCMs in \cite{gnecco2021causal} and \cite{gissibl2018max}, under appropriate heavy-tail assumptions on the innovation variables; we will elaborate on this in Section~\ref{Sec:limit}.

In addition to \eqref{eq:Lin-eSCM} and \eqref{eq:MaxLin-eSCM}, 
we further discuss two specific examples motivated by models in the existing literature.  Let $(\Omega,\cl{F},\mu)$, $\pp{\sbf{\eta},\sbf{\theta}}=\pp{ \pp{\eta_v}_{v\in V},\pp{\theta_v}_{v\in V}}$, and $\Prt_{\sbf{\theta}}$ be as in Definition \ref{def:eSCM}.

\begin{example}\label{Eg:ML Noise}\emph{(Max-linear eSCM with propagating noise.)}
   {\rm This example is motivated by \cite{buck2021recursive}; see also \cite{tran2024estimating}.  Let $F_\epsilon$ be the CDF of a random variable $\epsilon\in (0,\infty)$ with $\E \pb{\epsilon^\alpha}<\infty$. Let $(\epsilon_v)_{v\in V}:=\pp{ F^{-1}_\epsilon(\theta_v)}_{v\in V}$, where $F^{-1}_\epsilon$ is the generalized inverse of $F_\epsilon$. The variables $(\epsilon_v)_{v\in V}$ under $\Prt_{\sbf{\theta}}$  are i.i.d.\ following $F_\epsilon$.  Consider a DAG $\cl{G}=(V,E)$ with $d=|V|\in \bb{Z}_+$, and we associate each $(u,v)\in E$ a positive coefficient $a_{uv}>0$, and let $a_{uv}=0$ for $(u,v)\in V^2$ but $(u,v)\notin E$. 
  Suppose the eSCM \eqref{eq:general eSCM} has a proper structural function $h_v$ of the max-linear form: 
\begin{equation}\label{eq:h_v ML Noise}
h_v\pp{ \mbf{y}_{\pa(v)}, \theta_v }=   \epsilon_v\pp{ \bigvee_{u\in \pa(v)} a_{uv}    y_u  }.  
\end{equation}
When $\epsilon_v$ is a non-random constant, combining \eqref{eq:h_v ML Noise} with \eqref{eq:general eSCM} gives the simple max-linear eSCM \eqref{eq:MaxLin-eSCM}. The finiteness of $s_v$ in \eqref{eq:s_v moment} is satisfied due to the sufficient condition in Proposition \ref{Pro:basic}, since 
we have imposed $\E\pb{\epsilon^\alpha}<\infty$. For instance, one may assume $\epsilon$ follows a log-normal distribution as in \cite{tran2024estimating}.}


\end{example}

\begin{example}\label{Eg:HR eSCM}\emph{(H\"usler-Reiss eSCM).}
 {\rm This example is due to \cite{engelke2025extremes}, although not formally described within the eSCM framework.  Assume that the causal DAG $\cl{G}$ has a single root node, say node $1$, with an activation coefficient $a_1>0$, which implies that $\cl{G}$ has a single connected component. 
Suppose also $a_v=0$ for all non-root nodes $v\neq 1$.  These assumptions are necessary, as remarked in the discussion following \eqref{eq:mix act} to ensure that $\cl{L}\left(\mbf{Y}\right)$ is absolutely continuous with respect to the Lebesgue measure on $\bb{E}_V$.

Let $\Phi:\bb{R}\mapsto (0,1)$ denote the standard normal CDF, and $(Z_v)_{v\in V}:=\pp{ \mu_v+\sigma_v \Phi^{-1}(\theta_v)}_{v\in V}$, are independent normal random variables following $N(\mu_v,\sigma_v^2)$, $\mu_v\in \bb{R}$, $\sigma_v>0$, for $v\in V$, under $\Prt_{\sbf{\theta}}$  .  
 Consider a DAG $\cl{G}=(V,E)$, and we associate each $(u,v)\in E$ with a nonzero real coefficient $b_{uv}$, and set $b_{uv}=0$ for $(u,v)\in V^2$ but $(u,v)\notin E$. Impose the following  normalization condition:
 \begin{equation}\label{eq:HR coef norm}
 \sum_{u\in \pa(v)} b_{uv}=1,  \quad \text{for } v\in \{2,\ldots,d\}.   
 \end{equation}
Then suppose the eSCM \eqref{eq:general eSCM} admits a proper structural function $h_v$ of the form 
\begin{equation}\label{eq:h_v HR}
h_v\pp{ \mbf{y}_{\pa(v)}, \theta_v }= \exp\pp{ \sum_{u\in \pa(v)}b_{uv}  \log y_u  + Z_{v}  }= \pc{\prod_{u\in \pa(v)} y_u^{b_{uv}}} \exp(Z_v),
\end{equation}
if   $y_u>0$ for all $u\in \pa(v)$, and   $h_v\pp{ \mbf{y}_{\pa(v)}, \theta_v }=0$ if $\mbf{y}_{u}=0$ for some $u\in \pa(v)$.  

Since the root node $1$ is the only node with a nonzero activation coefficient, we have $\eta_1>0$ if and only if $Y_v>0$ for some $v\in V$, which is also equivalent to $Y_v>0$ for all $v\in V$. Observe that on the log-transformed scale of $\mbf{y}$ variables, \eqref{eq:h_v HR} specifies a linear structural relation with Gaussian noise. 
 The normalization \eqref{eq:HR coef norm} is to ensure that the function $h_v(\cdot, \theta_v)$ is homogeneous. 
 We call the resulting eSCM \eqref{eq:general eSCM} with $h_v$ in \eqref{eq:h_v HR} a \emph{H\"usler-Reiss eSCM}. 
 The name is justified by the fact that $\cl{L}(\mbf{Y})$ corresponds to a H\"usler-Reiss generalized multivariate Pareto law (e.g., \cite{rootzen2006multivariate,rootzen2018multivariate,kiriliouk2019peaks}). See Section \ref{sec:pf mult Pareto HR} in the Appendix for more details.}


 
\end{example}

\subsection{Approximation of eSCMs by probabilistic SCMs}\label{Sec:limit}

 The scaling relation \eqref{eq:Lambda} connects the exponent measure $\Lambda$ to the probabilistic law of the data $\mathbf{X}$. Meanwhile, the law of an eSCM has been formulated directly in terms of an exponent measure. This naturally raises the question: Can an eSCM \eqref{eq:general eSCM} emerge as the scaling limit of a probabilistic structural equation model (SCM) \eqref{eq:usual SCM}? This question is also of practical value. While an eSCM serves as an idealized model capturing the limiting extremal behavior, statistical analysis is conducted on finite-sample (pre-limit) data. It is therefore desirable to develop pre-limit models, such as probabilistic SCMs, that approximate eSCMs in the limit, enabling realistic simulations.
 We note that a similar idea appears in \cite{engelke2025extremes}.  However, unlike \cite{engelke2025extremes} which focuses on the single activation at a unique root node, we formulate a scheme that incorporates more general cases with multiple root nodes in the causal DAG and multiple nonzero activations.  
 
Suppose that a DAG $\cl{G}$ is given with a vertex set $V$. Motivated by the eSCM in \eqref{eq:general eSCM}, we also consider i.i.d.\ Uniform$(0,1)$ random variables $(\theta_v)_{v\in V}$, and let $\pp{\zeta_v}_{v\in V}$ be nonnegative random variables independent of $(\theta_v)_{v\in V}$, such that $\Prt\left(\zeta_v>x\right)\sim s x^{-\alpha}$, $\alpha>0$,  $s>0$, and $\Prt\pp{\zeta_u>x \mid \zeta_v> x }\rightarrow 0$, as $x\rightarrow\infty$ for distinct $u,v\in V$, i.e. $\zeta_v$'s are extremally independent. The assumptions on $\sbf{\zeta}:=\pp{\zeta_v}_{v\in V}$ imply that $\sbf{\zeta}$ is MRV and $t\rightarrow\infty$,
 \begin{equation}\label{eq:zeta asymp indep}
 t\Prt\pp{  t^{-1/\alpha} \sbf{\zeta}\in  \cdot } \overset{\mbf{v}}{\rightarrow} \Lambda^\perp(\cdot),
 \end{equation}
 where $\Lambda^\perp$ is as in \eqref{eq:Lambda_perp}; see \cite[Proposition 2.1.8]{kulik2020heavy}. 
 
 Now consider the probabilistic SCM of the form
 \begin{equation}\label{eq:pre limit SCM}
  X_v=g_v\pp{ \mbf{X}_{\pa(v)}, \zeta_v, \theta_v},  \quad v\in V,   
 \end{equation}
 for some suitable function $g_v$ (see Theorem \ref{Thm:limit eSCM} below). We assume that $g_v$'s and, consequently, the variables $X_v$'s  are nonnegative, which is a reasonable assumption when interpreting $\mbf{X}$ as the post-marginal-transform data as discussed in Section \ref{subsec:MEVT}. See also \cite{engelke2025extremes} for a similar consideration.

 Comparing \eqref{eq:pre limit SCM} with \eqref{eq:usual SCM}, we observe that the random innovation $e_v$ has been effectively split into two components, $\pp{\zeta_v,\theta_v}$. 
However, since we do \emph{not} require the $\zeta_v$'s to be probabilistically independent, the model \eqref{eq:pre limit SCM} goes beyond the framework of a conventional probabilistic SCM.
Theorem~\ref{Thm:limit eSCM} shows that $\mbf{X}$ defined in \eqref{eq:pre limit SCM} has a scaling limit with law $\cl{L}(\mbf{Y})$.

 \begin{theorem}\label{Thm:limit eSCM}
Suppose the setup in \eqref{eq:pre limit SCM} holds, and
 we further assume the following.
 \begin{enumerate} 
 \item Each measurable function $g_v:[0,\infty)^{\pa(v)}\times [0,\infty)\times [0,1] \mapsto [0,\infty),  \pp{\mbf{x}_{\pa(v)}, \zeta, \theta  }\mapsto g_v\pp{\mbf{x}_{\pa(v)}, \zeta, \theta  }$, $v\in V$, is \emph{asymptotically homogeneous} in its $(\mbf{x}_{\pa(v)},\zeta)$-component in the following sense.  
  There exists a measurable function 
\[
f_v^* : [0,\infty)^{\pa(v)} \times [0,\infty) \times [0,1] \to [0,\infty)
\]
 satisfying $f_v^*(\mathbf{0}_{\pa(v)},0,\theta)=0$  for any $\theta\in [0,1]$,
such that for any maps $
t \mapsto \mathbf{x}_{\pa(v)}(t) \in [0,\infty)^{\pa(v)}$,
$t \mapsto \zeta(t) \in [0,\infty)$, $t>0$,
with 
\[
\mathbf{x}_{\pa(v)}(t) \to \mathbf{y}_{\pa(v)} \in [0,\infty)^{\pa(v)},
\qquad 
\zeta(t) \to \eta \in [0,\infty),
\qquad t \to \infty,
\]
we have,  as $t\rightarrow\infty$
\[
t^{-1}\, g_v\!\left(t\,\mathbf{x}_{\pa(v)}(t), \; t\,\zeta(t), \; \theta\right) 
\;\longrightarrow\; f_v^*(\mathbf{y}_{\pa(v)}, \eta, \theta), \quad \text{for any $\theta \in [0,1]$.}
\]

 \item For each $v\in V$, there exists measurable  $C_v:[0,1]\mapsto [0,\infty)$, such that $g_v(\mbf{X}_{\pa(v)},\zeta_v, \theta_v )\le C_v(\theta_v) \|\pp{\mbf{X}_{\pa(v)},\zeta_v}\|$ a.s.,  and $\E_{\sbf{\theta}} \pb{C_v(\theta_v)^\alpha}<\infty$.
 
 \item For each $v\in V$, {$\liminf_{t\rightarrow\infty}t\Prt(t^{-1/\alpha}X_v>x)>0$ for some $x>0$.} 
 \end{enumerate}
  Then  each $f_v^*$  satisfies $f_v^*(c\mbf{y}_{\pa(v)},c\eta,\theta)=c f_v^*(\mbf{y}_{\pa(v)},\eta,\theta)$ for any $c\ge 0$, $\theta\in [0,1]$. 
  Furthermore, with the eSCM $\mbf{Y}$ constructed as in \eqref{eq:general eSCM}, but with $f_v$ replaced by $f_v^*$, we have as $t\rightarrow\infty$:
\begin{equation}\label{eq:X limit Y}
      t\Prt\pp{  t^{-1/\alpha} \mbf{X}\in  \cdot } \overset{\mbf{v}}{\rightarrow} \cl{L}(\mbf{Y}).
\end{equation}
 \end{theorem}
 We note that although $f_v^*$ is not readily of the form in \eqref{eq:general eSCM}, it can be reduced to that form via the modification in Remark \ref{Rem:red f_v}. A similar asymptotic homogeneity assumption is used in \cite{engelke2025extremes}. Asymptotic homogeneity of $g_v$ in its $(\mbf{x}_{\pa(v)},\zeta)$-component follows if exact homogeneity holds 
 and $g_v$ is continuous.     
 This   applies, for instance, when $g_v$ has a sum-linear or max-linear form as in \eqref{eq:Lin-eSCM} or \eqref{eq:MaxLin-eSCM} respectively, where $g_v$ does not depend on the randomization variable $\theta_v$.

Some examples of $\mbf{X}$  in \eqref{eq:pre limit SCM} can be found in Section \ref{sec:sim} below. See also \cite{engelke2025extremes} for further examples of nontrivial asymptotic homogeneity, noting that their descriptions on the exponential marginal scale can be translated to our Pareto marginal scale via suitable exponentiation.  

 { 
As a  consequence of Theorem \ref{Thm:limit eSCM}, we obtain the following   result  which may be viewed as a realization of the interpretation given in Remark \ref{Rem:PPP int eSCM}.
\begin{corollary}\label{Cor:PPP eSCM}
In the setup of Theorem \ref{Thm:limit eSCM},
  suppose $\left(\mbf{X}_i\right)_{i\in \bb{Z}_+}=\left(\pp{\mbf{X}_{i,v}}_{v\in V}\right)_{i\in \bb{Z}_+}$  are i.i.d.\ samples from the SCM $\mbf{X}$. Let $\sum_{i=1}^\infty\delta_{\sbf{\eta}_i}=\sum_{i=1}^\infty\delta_{\pp{ {\eta}_{i,v}}_{v\in V}}$ be a Poisson point process on $\bb{E}_V$ with   mean measure given by $\Lambda^{\perp}$ in \eqref{eq:Lambda_perp}.  Let $(\sbf{\theta}_i)_{i\in \bb{Z}_+}$ be i.i.d.\ random vectors each consisting of $d=|V|$ i.i.d.\ Uniform$(0,1)$ random variables, which are independent of $\sum_{i=1}^\infty\delta_{\sbf{\eta}_i}$. Define $Y_{i,v}=f_v^*(\mathbf{Y}_{i,\pa(v)},\eta_{i,v},\theta_{i,v})$, where $f_v^*$ is as in Theorem \ref{Thm:limit eSCM}. Then $\sum_{i=1}^\infty \delta_{\mathbf{Y}_i}$ (understood as restricted to the state space $\bb{E}_V$) is a Poisson point process  with mean measure $\cl{L}(\mathbf{Y})$ as described in Theorem \ref{Thm:limit eSCM},
  and   we have $$\sum_{i=1}^n \delta_{\mbf{X}_i/n^{1/\alpha}}\Rightarrow \sum_{i=1}^\infty \delta_{\mathbf{Y}_i}$$ as $n\rightarrow\infty$, where ``$\Rightarrow $'' denotes weak convergence of point processes on $\bb{E}_V$.
\end{corollary}
 }


\subsection{Interventions of eSCM}\label{Sec:intervention}

Assessing interventional effects, identified by \cite{pearl2009causality} as the second level in the ladder of causal modeling, beyond the level of statistical associations, has a well-established formalism in the framework of the usual structural equation model \eqref{eq:usual SCM}; see, for example, \cite[Definition 6.8]{peters2017elements}. In this section, we provide an initial discussion of the interventional properties of the eSCM introduced in Definition \ref{def:eSCM}, while leaving a more comprehensive treatment, including counterfactual analysis at the third level of the causal ladder, for future work.

  We adopt an idea similar to that of the \emph{intervention variable} (cf. \cite[Section 3.2.2]{pearl2009causality} and \cite{dawid2010beware}). Suppose that an eSCM $\mbf{Y}$ with respect to a DAG $\cl{G}=(V,E)$ is given as in Definition \ref{def:eSCM}. Let $V_{\mathrm{Int}}$ be a nonempty subset of $V$, which consists of all variables to be intervened. 
   To model the intervention, we introduce a new  node $d+1$ that has directed edges to each $v\in V_{\mathrm{Int}}$. For each $v$, we associate a measurable function  $h_{v}^\star:[0,\infty)\times [0,1]\mapsto  [0,\infty)$ that satisfies $h_v^\star(cy, \theta)=ch_v^\star(y, \theta)$ for any $\theta\in [0,1]$, $c\ge 0$, $y\ge 0$. We refer to these $h_v^\star$, $v\in V_{\mathrm{Int}}$, as \emph{intervention functions}.   
   Denote the  DAG  with the added node and edges as $\cl{G}^\star=(V^\star=(1,\ldots,d+1),E^\star)$.  Now we construct an eSCM with respect to $\cl{G}^\star$ that incorporates the intervention node $d+1$, which is a root node in $\cl{G}^\star$.
\begin{definition}\label{Def:interv}
    Let $\pp{\sbf{\eta}^\star=(\eta_v)_{v\in {V^\star}},\sbf{\theta}^\star= (\theta_v)_{v\in V^\star}}$ be as in Definition \ref{def:eSCM} with the role of $V$ replaced by $V^\star$ above,   the underlying measure denoted by $\mu^\star$, and the activation coefficients $a_v$ and proper structural functions $h_v$, $v\in V$, as in  Definition \ref{def:eSCM}.  The intervention eSCM $\mbf{Y}^\star=\pp{\mbf{Y},Y_{d+1}}$ is defined as
\begin{align}\label{eq:intervention eSCM}
Y_v=& f_v^\star\pp{\mbf{Y}_{\pa^\star(v)}, \eta_v, \theta_v}  := 
\begin{cases}
 a_v^\star \eta_v+ h^\star_v(Y_{d+1}, \theta_{v}),\  & v\in V_{\mathrm{Int}},\\
a_v\eta_v+ h_v\pp{\mbf{Y}_{\pa(v)}, \theta_v}, \  &v\in V\setminus V_{\mathrm{Int}},\\
 a_{d+1}\eta_{d+1}, &v=d+1,
\end{cases}
\end{align}
where the activation coefficient  $a_{d+1}>0$,   $h_v^\star$, $v\in V_{\mathrm{Int}}$, are the intervention functions  as described above, and $a_v^\star\in [0,\infty)$, $v\in V_{\mathrm{Int}}$, are post-intervention activation coefficients, i.e., the new activation coefficients introduced by the intervention that can be different from the original activation coefficients $a_v$'s.   
\end{definition}
Observe that the replacement of the structural relations for a node $v\in V_{\mathrm{Int}}$ by $ a_v^\star \eta_v+h^\star(Y_{d+1},\theta_{v})$ can be viewed as erasing the original directed edges pointing to $v$ (i.e., $\pa(v)$ in $\cl{G}$ ceases to affect $v$). 
This is consistent with the interpretation of an intervention enforced by an external source node $d+1$ on the original eSCM system. See Figure \ref{fig:DAG intervene illu}.

Next, note that if the marginal non-degeneracy
$\mu^\star\pp{Y_v>1}\in(0,\infty)$ holds for every $v\in V^\star$, then the
intervention eSCM $\mathbf{Y}^\star$ is itself an eSCM satisfying
Definition~\ref{def:eSCM}. Thus, this formulation of intervention is still contained in the eSCM framework. Furthermore,  one may extend  Definition \ref{Def:interv} to incorporate more intervention variables in addition to $Y_{d+1}$. For simplicity, we   restrict the discussion here to the single intervention variable case, which suffices to cover the important scenario of  deterministic (or say atomic) intervention of the original eSCM below.

 For \( V_{\mathrm{Int}} \subset V \), a  \emph{deterministic (or atomic) intervention}  involves setting each \( Y_v \) to a fixed value \( \xi_v \in (0, \infty) \) for \( v \in V_{\mathrm{Int}} \), which is commonly denoted in the literature as \( \mathrm{do}(Y_v = \xi_v, \, v \in V_{\mathrm{Int}}) \).
It is important to note that, without introducing the additional intervention node \( d+1 \), if two nodes \( u, v \in V_{\mathrm{Int}} \) are extremally independent, which happens  when \( \An(u) \cap \An(v) = \emptyset \)  (e.g., $u=1$, $v=5$ in Figure \ref{fig:DAG illu}), it is impossible for both \( Y_v > 0 \) and \( Y_u > 0 \) to occur simultaneously in the original eSCM. However, with the introduction of the intervention node \( d+1 \) which links to both $u$ and $v$ (e.g.,  $u=1$, $v=5$ in Figure \ref{fig:DAG illu}), this issue can be circumvented.
To do so, we set the intervention functions \( h_v^\star \) in equation \eqref{eq:intervention eSCM} as:
\[
h_v^\star(y_{d+1}, \theta_v) = \xi_v   y_{d+1}.
\]
Thus, the law of the intervened eSCM \( \left( \mathbf{Y} \mid \mathrm{do}(Y_v = \xi_v, v \in V_{\mathrm{Int}}) \right) \), with \( \mathbf{Y}^\star = \left( \mathbf{Y}, Y_{d+1} \right) \), can be interpreted through the conditional law  of \( \mathbf{Y} \mid (Y_{d+1} = 1) \), or equivalently \( \mathbf{Y} \mid \left( \eta_{d+1} = a_{d+1}^{-1} \right) \).  {Here,  by the conditional law of  \( \mathbf{Y} \mid \left( \eta_{d+1} = a_{d+1}^{-1} \right) \), we mean the probabilistic joint law of $\mathbf{Y}$ induced by the uniform random variables $\pp{\theta_v}_{v\in V^\star}$ from the recursive functional relations in  \eqref{eq:intervention eSCM} when we set $ \eta_{d+1} = a_{d+1}^{-1} $ and $\eta_{1}=\ldots=\eta_d=0$}.

Note that when \( \eta_{d+1} > 0 \), we have \( \eta_v = 0 \) for \( v \neq d+1 \). Consequently, \( Y_v = 0 \) if \( d+1 \notin \An(v) \), based on  \eqref{eq:Y=F(eta,theta)}. This leads to an unusual phenomenon: the intervention \( \mathrm{do}(Y_v = \xi_v, v \in V_{\mathrm{Int}}) \) causes an extremal variable to vanish whenever it is not a descendant of the nodes in \( V_{\mathrm{Int}} \). See Figure \ref{fig:DAG intervene illu} again for an illustration.

\begin{figure}
\tikzset{
  VertexStyle/.append style = {minimum size=18pt},
  EdgeStyle/.append style   = {->, >=Latex}
}
\begin{tikzpicture}[scale=1] 
    \SetGraphUnit{2} 
    \Vertex{2}  
    \EA(2){3}  

    \SOWE(2){1}  
    \SOEA(3){4}

    \NOWE(2){5}   
    \NOEA(3){6}

    \Edges(1,2)
    \Edges(3,4)  
    \Edges(5,2)        
    \Edges(3,6) 

    \NO(3){7} 
    
 \node[above=0.1cm of 7] {Intervention node};

    \Edge(7)(3) 

    \node[draw, dashed, fit=(1)(2)(5), inner sep=8pt, rounded corners] (boxA) {};
    \node[above=0.2cm of boxA] {Unobservable when $\eta_7>0$};
\end{tikzpicture}
    \caption{Illustration of intervening an eSCM with the DAG in Figure \ref{fig:DAG illu}. The new node $7$ represents the intervention node $d+1$ in Definition \ref{Def:interv}. The original edge $(2\to 3)$ is erased. When node $3$ is intervened to be a nonzero value through conditioning on a positive value of the activation variable $\eta_7$, the nodes $1$, $2$ and $5$ in the dashed box become unobservable and take values $0$.}\label{fig:DAG intervene illu}
\end{figure}
This result may seem counterintuitive, as one might not expect an intervention to affect non-descendants. However, we argue that this does not lead to a contradiction. 
For an eSCM, conditioning on a nonzero intervention value of \( Y_{d+1} \) does not exert a direct effect on non-descendants of \( V_{\mathrm{Int}} \); instead, it renders them \emph{unobservable}. In fact, when the activation variable \( \eta_{d+1} \) associated with the intervention node is $0$, these non-descendants of \( V_{\mathrm{Int}} \) and their structural relations may become observable again, although they no longer influence $V_{\mathrm{Int}}$. Furthermore, since we are primarily concerned with how the deterministic intervention propagates through the descendants of \( V_{\mathrm{Int}} \), the unobservable non-descendants essentially do not matter in terms of inference on intervention effect. 
It is worth mentioning that a similar phenomenon has been observed in \cite{engelke2025extremes}, although their discovery was made through a limiting argument from an intervened probabilistic SCM  \cite[Theorem 2]{engelke2025extremes}, rather than by introducing an extra intervention variable.   

Note that the above discussion concerns a nonzero intervention value $\xi_v$.
If instead $\mathrm{do}(Y_v=0)$, i.e., when the variable $v$ is forced to take a non-extremal value, this can be interpreted as conditioning on the event $\{\eta_{d+1}=0\}$, and additionally assuming $a_{v}^\star=0$. 
In this situation, positive extremal values of non-descendants become observable, since $\eta_u$ for $u\neq d+1$ is allowed to be nonzero on $\{\eta_{d+1}=0\}$. For the example in Figure~\ref{fig:DAG intervene illu}, under the intervention $\mathrm{do}(Y_3=0)$, nodes $1,2,$ and $5$ are allowed to take nonzero values; moreover, nodes  $4$ and $6$ may also take positive values provided that their  activation coefficients are nonzero.

\subsection{Extremal causal Markov condition}\label{Sec:causal markov}

A causal structural model \eqref{eq:usual SCM} satisfies the causal Markov condition: a node is conditionally independent (in the usual probabilistic sense) of all its non-descendants given its parents; see, for example, \cite[Theorem 1.4.1]{pearl2009causality} and \cite[Theorem 6.3]{bongers2021foundations}.
This condition is stated locally (the directed local Markov property; see \eqref{eq:local markov} below). As shown in \cite{lauritzen1990independence}, it can also be expressed globally (the directed global Markov property) using separation in moralized subgraphs or d-separation; see  \cite{lauritzen1996graphical} for more details.

The causal Markov condition is crucial for causal learning in SCMs (see, e.g., \cite{glymour2019review}).
Analogously, one may expect a causal Markov condition to hold for the eSCMs introduced in Definition \ref{def:eSCM}. However, since eSCMs are governed by infinite-mass laws (exponent measures), the conventional notion of probabilistic conditional independence does not apply. Nevertheless, we will show that a causal Markov property holds with respect to a recently defined notion of extremal conditional independence \cite{engelke2020graphical,engelke2025graphical,bai2026conditional}, which we briefly recall here.

For an exponent measure $\Lambda$ on $\bb{E}_V$, we define $\Lambda_I$ on $\bb{E}_{I}=[0,\infty)^I\setminus \{\mbf{0}\}$ by \begin{equation}\label{eq:Lambda I}
\Lambda_{I}(\cdot ):=\Lambda(\mbf{y}_I\in \cdot \, ,\ \mbf{y}_I\neq \mbf{0} ).
\end{equation} 
 Note that $\Lambda_I$ is an exponent measure on $\bb{E}_I$ satisfying \eqref{eq:Lambda homo} and \eqref{eq:Lambda marginal} (with obvious modification of indices).
The following definition is a special case of the  conditional independence formulated for more general infinite-mass measures in  \cite{engelke2025graphical}; see Definition 3.1,  Theorem 4.1 and Remark 4.2 therein.  

\begin{definition}\label{Def:e cond indep}
Let $\Lambda$ be an exponent measure on $\bb{E}_V$ satisfying \eqref{eq:Lambda homo} and \eqref{eq:Lambda marginal}.  
Suppose that $A$, $B$ and $C$ are disjoint subsets of $V=\{1,\ldots,d\}$. Assume first $A,B\neq \emptyset$ and set $ D=A\cup B\cup C$ and $\cl{R}_D^{(v)}=  \{\mbf{y}_D\in \bb{E}_D:\ y_v\ge 1\} $, $v\in D$.  
Let $\mbf{Y}^{(v)}$ denote a random vector that takes the value in $\cl{R}_D^{(v)}$ whose  probability distribution is given by $\Lambda_D\left(\cdot \, \cap \cl{R}^{(v)}_D\right)/\Lambda_D\left(\cl{R}^{(v)}_D\right)$. 

Then $A,B$ are \emph{extremally conditionally independent} given $C$, denoted as $A\perp B\mid C\, [\Lambda]$, if the probabilistic conditional independence $\mbf{Y}_A^{(v)}\perp\mbf{Y}_{B}^{(v)}\mid \mbf{Y}_C^{(v)}$ holds for all $v\in D$. 
Furthermore,  the case $C=\emptyset$ is understood as probabilistic independence $\mbf{Y}_A^{(v)}\perp\mbf{Y}_{B}^{(v)}$, $v\in A\cup B$, which may alternatively be denoted as $A\perp B \, [\Lambda]$.  In addition, the relation $A\perp B\mid C\, [\Lambda]$ is understood to hold trivially whenever $A$ or $B=\emptyset$. 
\end{definition}

 \begin{remark}\label{Rem:e cond indep}
In contrast to the punctured spaces $\bb{E}_D$,    the rectangular shape of the  test  subspaces $\cl{R}_D^{(v)}$  ensures that one can work with product measures, which is indispensable for describing the probabilistic conditional independence relation. The extremal conditional independence above can also be described by different   test  rectangular subspaces different from $\cl{R}_D^{(v)}$; see \cite[Definition 3.1 and Section 4.1]{engelke2025graphical}.

In addition, with the same notation as above,  $A \perp B \mid C [\Lambda] $ is equivalent to $A \perp B \mid C [\Lambda_{D}] $ with $D=A\cup B\cup C$   \cite[relation (11)]{engelke2025graphical}, and hence one may assume without loss of generality that $A,B,C$ forms a partition of $V$. This aligns with the idea that a conditional independence relation among nodes in $A\cup B\cup C$ should remain unaffected by nodes outside this set.
Furthermore, the unconditional extremal independence $A\perp B \, [\Lambda]$ can be characterized by $\Lambda(\{\mbf{y}\in \bb{E}_V:\  \mbf{y}_A\neq \mbf{0}_A \text{ and }  \mbf{y}_B\neq \mbf{0}_B\})=0$ \cite[Proposition 5.1]{engelke2025graphical}.

\end{remark}

In \cite{engelke2025graphical}, it has been shown that the extremal conditional independence relation defined above satisfies the so-called semi-graphoid axiom, which further ensures the aforementioned equivalence between the directed local and global Markov properties \cite[Corollary 7.2]{engelke2025graphical}.   In the following, we shall simply use \emph{extremal causal Markov property} to refer to the two equivalent Markov properties with respect to the extremal conditional independence relation described in Definition \ref{Def:e cond indep}. 

 {To clarify the connection between extremal conditional independence and eSCMs, we recall a characterization recently established in \cite{bai2026conditional}. Let
\[
\Pi=\sum_{i=1}^{\infty}\delta_{\mbf{Y}_i}
\]
be a Poisson point process on \(\mathbb{E}_V\) with mean measure \(\Lambda\), arising as the limit process of scaled extreme sample point masses discussed in Remark~\ref{Rem:PPP int eSCM}. According to \cite{bai2026conditional}, the extremal conditional independence relation
$
A\perp B\mid C\, [\Lambda]
$
is equivalent to the ordinary probabilistic conditional independence of the corresponding projected Poisson point processes:
\[
\Pi_A \perp \Pi_B \mid \Pi_C,
\qquad
\Pi_I := \sum_{i=1}^\infty\delta_{\mbf{Y}_{i,I}},
\quad I\subseteq V,
\]
where for nonempty $I\neq V$ , $\Pi_I$ is viewed as a Poisson point process with state space $[0,\infty)^I$ (so the origin $\mbf{0}_I$ is included). 
Thus, the notion introduced in Definition~\ref{Def:e cond indep} may be viewed as the canonical conditional independence notion for describing dependence relations among the coordinates of extreme samples. Since an eSCM can be interpreted as the structural causal model governing the points of \(\Pi\), the following result is entirely analogous to the Markov property for standard SCMs. }

\begin{theorem}\label{Thm:e causal Markov}
Suppose $\Lambda=\cl{L}(\mbf{Y})$ is the law of an eSCM  $\mbf{Y}$  associated with the DAG $\cl{G}$ as  in Definition \ref{def:eSCM}. Then $\Lambda$ satisfies   the extremal causal Markov property with respect to $\cl{G}$, that is,  
\begin{equation}\label{eq:local markov}
\{v\}\perp  \pp{\nd(v)\setminus  \pa(v)}  \mid \pa(v) [\Lambda], \quad  v\in V.
\end{equation} 
\end{theorem}

In fact, the following converse of Theorem~\ref{Thm:e causal Markov} also holds.
\begin{theorem}\label{Thm:Lambda expr by eSCM}
Suppose $\Lambda$ is an arbitrary exponent measure on $\bb{E}_V$ satisfying \eqref{eq:Lambda homo} and \eqref{eq:Lambda marginal}, which obeys the extremal causal Markov property {\eqref{eq:local markov}}, with respect to a DAG $\cl{G}$. 
Then there exists an eSCM $\mbf{Y}$ as in Definition \ref{def:eSCM} associated with   $\cl{G}$   such that $\cl{L}(\mbf{Y})=\Lambda$.   
\end{theorem}

Here we emphasize that no additional assumptions are imposed on $\Lambda$ beyond the basic conditions \eqref{eq:Lambda homo} and \eqref{eq:Lambda marginal}, suggesting that {both theorems}  apply not only when $\Lambda$ is absolutely continuous with respect to the Lebesgue measure (thus admitting a density) but also when $\Lambda$ is singular, e.g., when $\Lambda$ is supported on a finite number of rays in $\bb{E}_V$. Consequently, the class of eSCM models described in Definition \ref{def:eSCM} is sufficiently broad to accommodate any law $\Lambda$ that satisfies the extremal causal Markov property.  

Theorems \ref{Thm:e causal Markov} and \ref{Thm:Lambda expr by eSCM} also entail that from the perspective of an exponent measure $\Lambda$,  directed graphical models (or a Bayesian network; see \cite{lauritzen1996graphical}) formulated based on extremal conditional independence (Definition \ref{Def:e cond indep}) and  eSCMs  (Definition \ref{def:eSCM}) are  equivalent.      
We mention an immediate consequence of Theorem \ref{Thm:Lambda expr by eSCM} in the following.
\begin{corollary}\label{cor:equiv}
Suppose $\Lambda$ is an arbitrary exponent measure on $\bb{E}_V$ satisfying \eqref{eq:Lambda homo} and \eqref{eq:Lambda marginal}. Then there exists an eSCM $\mbf{Y}$ as in Definition \ref{def:eSCM} associated with a suitable DAG $\cl{G}$  such that $\cl{L}(\mbf{Y})=\Lambda$.  
\end{corollary}

Corollary~\ref{cor:equiv} follows from Theorem \ref{Thm:Lambda expr by eSCM} by considering a DAG $\cl{G}=(V,E)$ for which any pair of nodes is connected by a directed edge, e.g.,  $E=\{(u,v)\in V^2: u<v\}$.  Such a $\cl{G}$ does not impose any nontrivial causal Markov restriction on $\Lambda$ so that any extremal law $\Lambda$ can be fit by an eSCM in theory.  {Results analogous to Theorem \ref{Thm:Lambda expr by eSCM} and Corollary \ref{cor:equiv}} for ordinary probabilistic SCMs can be found in Proposition 7.1 of \cite{peters2017elements}.

As noted above, the causal Markov condition plays a central role in the statistical learning of the underlying causal DAG. When combined with the faithfulness assumption (i.e., the joint distribution exhibits no conditional independence relations beyond those implied by the DAG), one can use observed conditional independence relations to infer structural features of the causal model and, in some cases, recover the DAG itself. This principle underlies constraint-based causal discovery methods such as the popular PC algorithm \citep{spirtes2000causation}.

In the setting of causal analysis for extremes, recent studies \cite{jiang2025separation,engelke2025extremes,adams2025inference} have investigated this problem within certain probabilistic SCMs, which can be reformulated as specific parametric families of eSCMs (see Example~\ref{Eg:HR eSCM}). Extending such analyses beyond parametric models for extremes is of considerable interest. This parallels developments in the broader SCM literature (cf. \cite{zhang2012kernel}). A key challenge is the design of nonparametric statistical tests or decision rules for extremal conditional independence introduced by \cite{engelke2025graphical}. Developing such tools represents a promising direction for future research.

While conditional independence plays a central role in characterizing causal structures, it does not in general determine causal directions. For example, the three DAGs $1\rightarrow 2\rightarrow 3$, $1\leftarrow 2\rightarrow 3$, and $1\leftarrow 2\leftarrow 3$ all entail the same conditional independence relation $\{1\}\perp \{3\}\mid \{2\}$. In the next section, we discuss assumptions and statistical approaches for identifying causal direction in the framework of eSCMs.

 \section{Extremal causal asymmetry and causal direction learning}\label{Sec:causal asym}

\subsection{Extremal causal asymmetry}\label{Sec:causal asym discussion}

 {In a causal DAG $\cl{G}=(V,E)$, we say node $u\in V$ causes node $v\in V$, or say $u$ is a cause and $v$ is an effect, if $u\in \mrm{an}(v)$, $v\in V$.}  
For probabilistic SCM \eqref{eq:usual SCM}, it is well-known that distinguishing cause and effect based on the statistical law of $\mbf{X}=\pp{X}_{v\in V}$ is impossible
unless more detailed assumptions are made. 
For instance, Chapter 4 of \cite{peters2017elements} gives a survey of assumptions on the structural function $f_v$ and noise $e_v$ that ensure the identifiability.
In general, the same comment applies to the eSCMs in Definition \ref{def:eSCM}.


Now we impose some interpretable assumptions to guarantee the identifiability of cause and effect.
 Given the extremal variables $\mbf{Y}$ as defined in Definition \ref{def:eSCM} with law $\cl{L}(\mbf{Y})=\Lambda$, for a non-empty subset of nodes $I\subset V=\{1,\ldots,d\}$, the $I$-marginal law $\cl{L}(\mbf{Y}_I)$ refers to  $\Lambda_I$ in \eqref{eq:Lambda I}.

\begin{assumption}(Nonzero Activation.)\label{ass:eta act}
The activation coefficient $a_v>0$  for any $v\in V$ in \eqref{eq:general eSCM}. 
\end{assumption}


\begin{assumption}(Nonzero Parent Effect.) \label{ass:weak asym}
  For any $v\in V$ satisfying $\mrm{pa}(v)\neq\emptyset$,  with the proper structural function $h_v$ in \eqref{eq:general eSCM},  we require $\mu\pp{ h_v \pp{  \mbf{Y}_{\mrm{pa}(v)} , \theta_v}=0, \mbf{Y}_{\mrm{pa}(v)}\neq \mbf{0}_{\pa(v)} }=0$.
\end{assumption}

Assumption~\ref{ass:eta act} suggests that any extremal variable has an intrinsic activation randomness, so one variable may become extremal (i.e., nonzero) even though its parent variables are not. 
Meanwhile, Assumption~\ref{ass:weak asym} specifies a causal minimality-type condition (see, e.g., \cite[Section 6.5.2]{peters2017elements}): Once a parent extremal variable is nonzero, it always generates a nonzero effect on its descendants.

Given Assumptions~\ref{ass:eta act} and \ref{ass:weak asym}, the result below describes the  pairwise causal asymmetry induced.

\begin{proposition}\label{Pro:causal asym}
   Consider an eSCM as in Definition \ref{def:eSCM} with law $\Lambda=\cl{L}(\mbf{Y})$. Let $\Lambda_{\{u,v\}}$ be the marginal law  as in \eqref{eq:Lambda I} with $I=\{u,v\}$, and distinct $u,v\in V$.  Then     
   Assumption \ref{ass:eta act}  implies
   $\Lambda_{\pc{u,v}}(y_u>0,y_v=0)=\mu(Y_u>0,Y_v=0)>0$ when $u\notin \an(v)$, $v\in V$ (i.e., when $u$ does not cause $v$). 
   Also, Assumption \ref{ass:weak asym} gives $\Lambda_{\pc{u,v}}(y_u>0, y_v=0)=\mu(Y_u>0,Y_v=0)=0$ when $Y_u$ causes $Y_v$. 
\end{proposition}

In particular, the proposition implies that under Assumptions  \ref{ass:eta act} and 
\ref{ass:weak asym}, the causal-effect relation is identifiable from $\cl{L}(\mbf{Y})$ through the following criterion.

\begin{corollary}\label{Cor:cause effect 1}
Suppose Assumptions \ref{ass:eta act} and 
\ref{ass:weak asym} hold. 
Then $Y_u$ causes $Y_v$ if and only if $\cl{L}(Y_u,Y_v)$ has mass on the $y_v$ axis, but does not have mass on the $y_u$ axis. 
\end{corollary}

There is an appealing causal interpretation of the corollary.  An extreme in $Y_u$ always leads to an extreme in $Y_v$, but not vice versa --- the mass along the $y_v$-axis direction means that $Y_v$ can be extremal alone without $Y_u$.
However, the asymmetry in Corollary \ref{Cor:cause effect 1} can be too subtle to explore statistically.  
To enhance the prominence of this asymmetry for practical statistical identification, we further introduce the following working assumption.     
\begin{assumption}\label{ass:asym}(Enhanced Causal Asymmetry.)
For any $v\in V$ and $u\in \mrm{an}(v)$, there exists $c_{uv}\in (0,\infty)$, such that $\Lambda_{\{u,v\}}(y_v<c_{uv}y_u)=0$.
\end{assumption}

{The two subplots in Figure \ref{fig:sim eSCM vs gen eSCM} both give
an illustration of Assumption~\ref{ass:asym} with $u=1$ and $v=2$, where the lower boundary of each cone can be regarded as the ray $\{y_2=c_{12}y_1\}$. }

Next, we present a convenient sufficient condition for  Assumption \ref{ass:asym}.

\begin{proposition}\label{pro:char ass}

A sufficient condition for Assumption \ref{ass:asym}  is that for all $v\in V$ with $\pa(v)\neq \emptyset$, the proper structural function $h_v$ in \eqref{eq:general eSCM} satisfies $h_v(\mbf{Y}_{\pa(v)},\theta_v)\ge d_v \|\mbf{Y}_{\pa(v)}\|$ $\mu$-a.e.\ for some constant $d_v>0$.
\end{proposition}
 An example where the sufficient condition in Proposition \ref{pro:char ass} holds is when the eSCM is simple (i.e.,  each $h_v$ in \eqref{eq:general eSCM} does not depend on $\theta_v$) and Assumption \ref{ass:weak asym} holds, once noting that $\cl{L}(\mbf{Y}_{\pa(v)})$ concentrates on a finite number of rays in this case. Another such example can be found by considering Example \ref{Eg:ML Noise}, once assuming that the support of the distribution $\epsilon_v$ in \eqref{eq:h_v ML Noise} is separated from $0$. On the other hand, Example \ref{Eg:HR eSCM} does not satisfy Assumption \ref{ass:asym}.  {We also note that the qualitative asymmetry in Assumption~\ref{ass:weak asym} or Assumption~\ref{ass:asym} is invariant under the coordinatewise monotone marginal transformations used to put the data \(\mbf{X}\) on the same Pareto scale as long as   MRV is preserved; in particular, it does not depend on the choice of \(\alpha\) in \eqref{eq:marg trans}.}

\subsection{Statistical identification of extremal causal direction}\label{Sec:stats caus dir}

In this section, we propose an approach to statistically identify the cause-effect order based on Assumptions \ref{ass:eta act} and \ref{ass:asym}. We first formulate the causal asymmetry implied by the assumptions in terms of the \emph{angular measure}, from which a natural measure of causal asymmetry arises.

Recall the exponent measure  $\Lambda$, due to its homogeneity, admits a polar decomposition into angular and radial components.  More specifically, recall $\|\cdot\|$ denotes a norm on $\bb{R}^V$, $V=\{1,\ldots,d\}$.
Slightly abusing the notation,  using still $\Lambda$  to denote the push-forward measure of $\Lambda$ under the mapping $[0,\infty)^V\setminus \{\mathbf{0}\} \mapsto (0,\infty)\times \mathbb{S}_+^{V}$,   $\mathbf{y}\mapsto (r,\mathbf{z}=(z_1,\ldots,z_d)):= \left(\|\mathbf{y}\|, \mathbf{y}/\|\mathbf{y}\| \right)$, where $\mathbb{S}_+^{V}=\{\mbf{y}\in [0,\infty)^V:\  \|\mbf{y}\|=1 \}$, we have the product measure factorization
\begin{equation}\label{eq:polar}
\Lambda(dr,d\mbf{z})= \nu_\alpha(dr) S(d\mathbf{z}),
\end{equation}
where the radial measure $\nu_\alpha(dr)=  c_0 \alpha r^{-\alpha-1}dr$ with $c_0=\Lambda(\{\mbf{y}\in [0,\infty)^V:\ \|\mbf{y}\|>1  \})$, and $S$ is a probability measure on $\mathbb{S}_+^{V}$ known as  the \emph{angular (or spectral) measure}. The measure $S$   describes the directional distribution of the concurrence of the extreme values and characterizes the extremal dependence. See \cite[Chapter 6]{resnick2007heavy} for more details. 

To proceed, we specifically work with the case where $d=2$  and $\|\cdot\|= \|\cdot\|_1$. In this case, we parameterize  $\bb{S}_+^{\{u,v\}}$, $u,v\in V$, $u\neq v$, by the map $[0,1]\mapsto\bb{S}_+^{\{u,v\}}, w\mapsto \pp{w,1-w}$,  and regard $S$  as a probability measure on $[0,1]$ through the pullback of the parameterization map.  Then \eqref{eq:polar} becomes
\begin{equation}\label{eq:polar 2d}
\Lambda(dr,dw)= \nu_\alpha(dr)  S(dw).
\end{equation}
Let $a=\sup\{w\in [0,1]:\  S([0,w))=0  \}$, $b=\inf\{w\in [0,1]:\  S((w,1])=0  \}$. We refer to $[a,b]\subset [0,1]$ as the \emph{angular support interval}, which is the smallest closed interval containing the support of $S$.  See Figure \ref{fig:ab cone} for an illustration.
 \begin{figure}
\centering
 \begin{tikzpicture}[scale=0.7]

    \draw[->] (0,0) -- (5,0) node[anchor=north] {$y_u$};
    \draw[->] (0,0) -- (0,5) node[anchor=east] {$y_v$};

   
  \fill[red, opacity=0.3] 
      (0,0) -- ({4.9*cos(10)}, {4.9*sin(10)}) 
      arc[start angle=10, end angle=70, radius=4.9] -- cycle;

    \draw[dashed] (0,4) -- (4,0);

    \node[left] at (0,4) {$1$};
    \node[below] at (4,0) {$1$};

    \coordinate (I1) at ({4/(tan(70)+1)}, {4 - 4/(tan(70)+1)});
    \coordinate (I2) at ({4/(tan(10)+1)}, {4 - 4/(tan(10)+1)});

    \draw[dashed] (I1) -- (I1|-0,0);
    \draw[dashed] (I2) -- (I2|-0,0);

    \node[below] at (I1|-0,-0.15) {$a$};
    \node[below] at (I2|-0,0) {$b$};

\end{tikzpicture}
\caption{\footnotesize{Illustration of angular support interval $[a,b]$. The shaded area represents the smallest cone/sector containing the support of $\Lambda_{\{u,v\}}$.}} \label{fig:ab cone}
\end{figure}


Now consider  an eSCM $\mbf{Y}$  with respect to a DAG $\cl{G}$ as in Definition \ref{def:eSCM}. Then  under Assumptions  \ref{ass:eta act} and \ref{ass:asym}, one obtains the following cause-effect identification criterion which enhances Corollary \ref{Cor:cause effect 1}.
\begin{corollary}\label{Cor:cause effect 2}
Suppose Assumptions  \ref{ass:eta act}     
 and \ref{ass:asym} hold. Then  
$Y_u$ causes $Y_v$  if and only if the angular support interval $[a,b]$ of  $\cl{L}(Y_u,Y_v)$ satisfies $a=0$ and $b<1$.
\end{corollary}
In particular,  if $c_{uv}$ in Assumption \ref{ass:asym} is the maximum slope that satisfies $\Lambda_{\{u,v\}}(y_v<c_{uv}y_u)=0$, then $b=1/(1+c_{uv})$.

Corollary~\ref{Cor:cause effect 2} motivates the introduction of the following \emph{angular asymmetry coefficient (AAC)}. For distinct nodes $u,v\in V$, define
\begin{equation}\label{eq:AAC}
 \tau(u,v)= 1-b-a.  
\end{equation}
Note that in view of Proposition \ref{Pro:causal asym}, when there is no causal relation between $u$ and $v$ ($u\notin \an(v)$ and $v\notin \an(u)$), we have $a=0$ and $b=1$.  
Meanwhile, the sign of AAC aligns with the causal direction.  In addition, when the roles of $u$ and $v$ switch, so do the roles of $a$ and $1-b$. Hence, we have the skewed symmetric property: $\tau(u,v)=-\tau(v,u)$; see Figure \ref{fig:tau} for a summary of the behavior of AAC under Assumptions  \ref{ass:eta act} and \ref{ass:asym}.

\begin{figure}
\centering
\begin{tikzpicture}[scale=0.5]

\draw[->] (0,0) -- (5,0) node[anchor=north] {$y_u$};
\draw[->] (0,0) -- (0,5) node[anchor=east] {$y_v$};

\draw[-, line width=0.6mm, red] (0,0) -- (0,4.9);

\fill[red, opacity=0.2] (0,0) -- ({4.9*cos(20)}, {4.9*sin(20)}) 
    arc[start angle=20, end angle=90, radius=4.9] 
    -- cycle;

\node at (2.5, -1.4) {\footnotesize{$u$ causes $v$; $\tau(u,v)>0$.}};

\begin{scope}[xshift=10cm]
    \draw[->] (0,0) -- (5,0) node[anchor=north] {$y_u$};
    \draw[->] (0,0) -- (0,5) node[anchor=east] {$y_v$};
    
    \draw[-, line width=0.6mm, red] (0,0) -- (0,4.9);
    \draw[-, line width=0.6mm, red] (0,0) -- ({4.9*cos(0)}, {4.9*sin(0)});
    
    \fill[red, opacity=0.2] (0,0) -- ({4.9*cos(0)}, {4.9*sin(0)}) 
        arc[start angle=0, end angle=90, radius=4.9] 
        -- cycle;

    \node at (2.5, -1.4) {\small{No causal relation; $\tau(u,v)=0$.}};
\end{scope}

\begin{scope}[xshift=20cm]
    \draw[->] (0,0) -- (5,0) node[anchor=north] {$y_u$};
    \draw[->] (0,0) -- (0,5) node[anchor=east] {$y_v$};
    
    \draw[-, line width=0.6mm, red] (0,0) -- (4.9,0);
    
    \fill[red, opacity=0.2] (0,0) -- ({4.9*cos(0)}, {4.9*sin(0)}) 
        arc[start angle=0, end angle=70, radius=4.9] 
        -- cycle;

    \node at (2.5, -1.4) {\footnotesize{$v$ causes $u$; $\tau(u,v)<0$.}};
\end{scope}

\end{tikzpicture}
\caption{\footnotesize{Behavior of   angular asymmetry coefficient (AAC) with respect to causal relations under Assumptions \ref{ass:eta act} and \ref{ass:asym}. Solid lines indicate measure  masses, while shaded cones represent angular supports.}}\label{fig:tau}
\end{figure}

Next we propose an estimator of the angular support interval $[a,b]$, which is a modification of the one considered in \cite{WangResnick2024} mainly to ensure a symmetric treatment of the two variables.  Let $\Delta=\{(s,t)\in [0,1]^2,\  s\le t\}$.  Consider the following function $d: [0,1]\times \Delta \mapsto [0,1] $ that serves as a distance from point $w\in [0,1]$ to interval $[s,t]$, $0\le s\le t\le 1$, defined as 
\begin{equation}\label{eq:tw d}
d(w,s,t)=  \pp{s-w} \vee  \pp{w-t} \vee 0.  
\end{equation}
Consider also a function $L:[1,\infty)\mapsto [0,\infty)$  defined as
$
L(r)= r \log r,
$
which will play the role of weighting the observations according to their  radial locations.
Let $(X_{i,1},X_{i,2})_{i=1,\ldots,n}$ be i.i.d.\ observations of a random vector $(X_1,X_2)$ that satisfies the MRV condition \eqref{eq:Lambda}. Order them as random vectors $(X_{(1),1},X_{(1),2}),\ldots,(X_{(n),1},X_{(n),2})$, so that $R_{(1)}\ge  \ldots \ge  R_{(n)}$, $R_{(i)}:= X_{(i),1}+X_{(i),2}$. Set $W_{(i)}=X_{(i),1}/R_{(i)}$. Here and below, we often suppress the dependence on $n$ for the brevity of notations.

{Let $k\equiv k_n$ denote the \emph{extremal subsample size},  $1\le k\le n$},   define
$$
D_k(s,t)=\frac{1}{k} \sum_{i=1}^k  d(W_{(i)},s,t) L(R_{(i)}/ R_{(k)}),
$$
and set the objective function 
\begin{equation}\label{eq:g_n(s,t)}
 g_n(s,t)=   t-s   +  \lambda k^{1/2}   D_k(s,t),
\end{equation}
where $\lambda\in (0,\infty)$ is a tuning parameter. The first term $t - s$ reflects the size of the cone, whereas the second term 
$\lambda k^{1/2} D_k(s,t)$ penalizes deviations of the extremal subsample from the cone.
Note that the objective function $g_n$ is continuous.
The estimator of $a$ and $b$ is formulated as follows:
\begin{equation*} 
\pp{\wh{a}_n,\wh{b}_n
}=\argmin_{s,t\in \Delta} g_n(s, t),
\end{equation*}
where the operation $\argmin$ is understood as selecting a measurable representative of the minimizer if the latter is not unique.  A larger   $\lambda$ value encourages a wider $\pb{\wh{a}_n,\wh{b}_n}$ interval. Empirically, we find that the range $1 \leq \lambda \leq 5$ typically yields good performance. In our numerical study, the minimization is performed using
the Nelder--Mead method within the base \textsf{R} function
\texttt{constrOptim} \citep{R2024}, which enforces the
constraints $0\leq s\leq t\leq1$ through an adaptive barrier
algorithm.

In view of \cite{WangResnick2024},
the estimator $\pp{\wh{a}_n,\wh{b}_n
}$  is consistent under a hidden regular variation condition \citep{resnick2024art}, which, loosely speaking, says that the radial tail of $(X_1,X_2)$ outside the angular support interval $[a,b]$ is lighter than the one inside.  In the Appendix, we include a self-contained treatment of the consistency of $\pp{\wh{a}_n,\wh{b}_n
}$ under a second-order condition we refer to as $\mathcal{SO}(\rho)$ (see Section \ref{sec:angular_support_interval} in Appendix), where $\rho>0$ is the second-order parameter. One may understand $(1+\rho)\alpha$ as the tail index outside $[a,b]$, in contrast to the tail index $\alpha$ inside.

To understand the intuition behind the estimation objective function $g_n(s,t)$, note first that if $[a,b]\setminus[s,t]\neq\emptyset$, then $D_k(s,t)$ will include contributions from many extremal sample points from the heavy–tail angular region $[a,b]\setminus[s,t]$, causing $\lambda k^{1/2}D_k(s,t)$ to be very large relative to the interval length $t-s$. In order to reduce $g_n$ in this case, the interval $[s,t]$ must be expanded until it fully covers $[a,b]$. Conversely, if $[a,b]\subsetneq[s,t]$, then the sum in $D_k(s,t)$ only involves a small number of extremal samples from the light–tail region $[0,1]\setminus[s,t]$, making $\lambda k^\gamma D_k(s,t)$ negligible compared with $t-s$ under suitable assumptions. To reduce $g_n$ in this case, one therefore needs to shrink $[s,t]$ so as to decrease its length $t-s$.

The condition $\mathcal{SO}(\rho)$ is slightly weaker than the hidden regular variation condition assumed in  \cite{WangResnick2024}. The consistency holds when  $k=k_n\rightarrow\infty$ and $k=o(n^{\rho/(1/2+\rho)})$ as $n\rightarrow\infty$.
Then plugging the consistent estimates $\wh{a}_n$ and $\wh{b}_n$ into \eqref{eq:AAC}, we get a consistent estimate  of $\tau(u,v)$ as 
\begin{equation}\label{eq:tau est}
\wh{\tau}(u,v)=1-\wh{b}_n-\wh{a}_n .
\end{equation}

\subsection{Extremal causal order identification}\label{sec:causal order}

 Given a causal DAG with node set $V=\{1,\ldots,d\}$, the \emph{causal order} (or topological order) is a permutation  $\pi:V\mapsto V$ satisfying $u\in \an(v) \implies\pi(u)<\pi(v)$. For a causal DAG, there exists at least one causal order, which may not be unique. Even though a causal order does not fully identify a DAG, it provides crucial information on causal relations and reduces the search space for further DAG discovery. See, e.g., \cite[Appendix B]{peters2017elements} and \cite{park2020identifiability}. 
 
 With $\tau(u,v)$ defined in \eqref{eq:AAC}, we provide a method to identify the causal order $\pi$ of an eSCM satisfying Assumptions \ref{ass:eta act} and \ref{ass:asym}.  
 Specifically, we give a variant to the \emph{extremal ancestral search (EASE)} algorithm \citep{gnecco2021causal},
which replaces the causal tail coefficient $\Gamma_{uv}$ (see \cite[Definition 1]{gnecco2021causal}) in the original algorithm by AAC $\tau(u,v)$. For the convenience of the reader, we include the details in Algorithm \ref{Alg:EASE}. We note that  the algorithm essentially relies on the ranks of $\tau(u,v)$, and thus enjoys the tolerance of uncertainty in estimating $\tau(u,v)$ compared to relying on the signs of $\tau(u,v)$ to infer causal order.
Proposition~\ref{Pro:EASE empirical} below provides a consistency result of Algorithm \ref{Alg:EASE}.

\begin{algorithm}
\caption{EASE algorithm with AAC}\label{Alg:EASE}
\begin{algorithmic}[1]
  \Require AACs $\pp{\tau(u,v)}_{u,v\in V, u\neq v}$ associated with node set $V=\{1,\ldots,d\}$
  \Ensure Causal order $\pi: V \mapsto V$
  \State $V_1 \gets V$
  \For{$s = 1$ to $d-1$}
    \ForAll{$v \in V_s$}
      \State $M_v^{(s)} \gets \max_{u \in V_s \setminus \{v\}} \tau(u,v)$
    \EndFor
    \State $v_s \in \argmin_{v \in V_s} M_v^{(s)}$
    \State $\pi(v_s) \gets s$
    \State $V_{s+1} \gets V_s \setminus \{v_s\}$
  \EndFor
  \State \Return permutation $\pi$
\end{algorithmic}
 
\end{algorithm}

\begin{proposition}\label{Pro:EASE empirical}
 {Suppose Assumptions \ref{ass:eta act} and \ref{ass:asym} hold, and that  all the AACs $\tau(u,v)$ in Algorithm \ref{Alg:EASE} are estimated consistently.}  Then with probability tending to $1$,   Algorithm \ref{Alg:EASE} returns a correct causal order.
\end{proposition}

Currently, no asymptotic distributional result is available for the AAC estimator $\wh{\tau}(u,v)$. For the causal tail coefficient, and indeed for a more general version thereof, \cite{bodik2024causality} empirically proposed a bootstrap procedure to facilitate inference. Establishing a theoretically justified inference framework for AAC therefore remains an important direction for future research.
 
 \section{Numerical results}\label{Sec:simdata}

%
%

In this section, we provide a simulation study to analyze the performance of the proposed method, together with its efficacy while applied to two real data examples. Additional simulation can be found in Section~\ref{append:sim} of the Appendix as well. The R code to reproduce these results is available at \url{https://github.com/feifang1/eSCM_code}. 

To evaluate the accuracy of the estimated causal order $\wh{\pi}$ for a DAG, we consider the \emph{ancestral violation rate} $\frac{1}{|E_{\cl{A}}|}
\sum_{(u,v)\in E_{\cl{A}}}
\mbf{1}\{\wh{\pi}(u)>\wh{\pi}(v)\}$, where $E_{\cl{A}}
=
\{(u,v)\in V^2:\ u\in\an(v)\}$. For an edgeless DAG, we define the ancestral violation rate to be $0$. Thus, the ancestral violation rate is the proportion of ancestor--descendant pairs whose relative ordering is reversed by $\wh{\pi}$. In particular, a smaller value indicates a causal ordering that is more consistent with the underlying ancestral relations, with a value of zero corresponding to an ordering that respects all such relations. For estimation of the AAC, the subsample size is chosen as $k=\left\lfloor n^a+\frac{1}{2}\right\rfloor$. We select the exponent $a$ through a preliminary simulation study over the grid $a\in\{0.20,0.25,\ldots,0.80\}$. The study considers a broad range of data-generating settings, including different sample sizes, probabilistic structural causal models, numbers of nodes, and model parameters. For each candidate value of $a$, we compute the ancestral violation rate across these settings and select the value attaining the smallest average violation rate. This procedure yields $a=0.7$. Accordingly, unless stated otherwise, we use $k=\left\lfloor n^{0.7}+\frac{1}{2}\right\rfloor$ throughout this section. Further details of this preliminary study are given in Section \ref{NR:choose_k} of the Appendix. 
  
\subsection{Simulation studies of extremal causal order discovery}\label{sec:sim}

We start with a simulation study on Algorithm \ref{Alg:EASE}. In view of Theorem~\ref{Thm:limit eSCM}, we simulate some probabilistic SCMs as realistic approximations of eSCMs. In particular,  following notations in Section \ref{Sec:limit}, we consider the sum-linear (SL) probabilistic SCMs 
\begin{equation}\label{eq:sum_linear sim}
X_v=\sum_{u\in \mrm{pa}(v)}   \beta_{uv}(\theta_v)  X_u +  \zeta_v 
\end{equation}
and the max-linear (ML) probabilistic SCMs
\begin{equation}\label{eq:max_linear sim}
X_v=\bigvee_{u\in \mrm{pa}(v)}  \pp{ \beta_{uv}(\theta_v)   X_u}  \vee \zeta_v,
\end{equation}
where each $\beta_{uv}(\theta_v)\ge 0$ is a randomized coefficient as a measurable function of the uniform random variable $\theta_v$. 

Assume also that $\beta_{uv}(\theta_v)$'s  are i.i.d.\  across $v\in V$ and $u\in \pa(v)$ with distribution $F_\beta$.  Note that even with the single randomizer $\theta_v$, it is possible to generate $|\pa(v)|$ independent variables \cite[Theorem 4.19]{kallenberg:2021:foundations}. Furthermore, $\pp{\zeta_v}_{v\in V}$ are i.i.d.\ random variables with a Pareto distribution and $F_{\zeta}(x)=1-x^{-\alpha_0}$, $x\ge 1$, $\alpha_0\in (0,\infty)$. The tail index $\alpha_0$ controls how prominently the effects of the activation variables $\sbf{\eta}$ are exhibited; the lower $\alpha_0$,  the more prominent the effect of ``single big jump'' is shown in a finite sample.

In the simulation, we consider DAGs with node size $d\in\{5,10,15,30\}$. Random DAGs are generated using the \texttt{randDAG} function in the \texttt{pcalg} R package \citep{Kalisch2012Causal}, with an average node degree of 3.
For each simulation experiment (repeated 500 times per $d$),  based on the DAG, we simulate one data set of size $n=1000$ from one of four model setups: SL0,  SL1, ML0 and ML1.
Both SL0 and SL1 correspond to the sum-linear SCM \eqref{eq:sum_linear sim}. For SL0,  $F_\beta=\mathrm{Uniform}(l,u)$ with $l=0.04$ and $u=0.4$. For SL1, the edge weights follow a lognormal distribution,
with $\log\beta\sim N(\mu,\sigma^2)$, where
$\mu=-1.5286$ and $\sigma=0.4682$.
Under this specification,
$\Prt\pp{l\le\beta\le u}\approx 0.9$. SL0 strictly satisfies Assumption \ref{ass:asym}, while SL1 only approximately satisfies it, allowing us to test robustness to moderate deviations. ML0 and ML1 both use the max-linear SCM \eqref{eq:max_linear sim}, with $F_\beta$ specified in the same way. 

For each simulated dataset, denoting $(z_i)_{i=1}^n$ as the original values of a node component with descending order statistics $z_{(1)}\ge z_{(2)}\ge\ldots\ge z_{(n)}$, we apply the marginal transform  in
 \eqref{eq:marg trans} with {$\alpha=2$} to $(z_i)_{i=1}^n$ to ensure that the marginal distribution follows a standard Pareto  distribution with parameter $\alpha=2$. Following a routine practice in extreme value analysis, the CDF $F$ of $(z_i)_{i=1}^n$  is estimated semi-parametrically as
\begin{equation}
\label{est_CDF_semi_para}
\wh{F}(z)=
\begin{cases}
\dfrac{1}{n}\sum_{i=1}^n \mathbf{1}_{\{z_i<z\}}, & z \in (-\infty, z_{(m)}], \\
1-\dfrac{m}{n} \left(1+ \hat{\gamma}_{m} \dfrac{z-z_{(m)}}{\hat{\sigma}_{m}}\right)^{-1/\hat{\gamma}_m}, & z \in (z_{(m)}, \infty),
\end{cases} \end{equation}
where {$m=\lfloor n^{0.5}+\frac{1}{2}\rfloor$} is the extremal subsample size used to fit a generalized Pareto distribution to the upper tail with estimated shape parameter $\hat{\gamma}_m$ and scale parameter $\hat{\sigma}_m$. For the implementation of these estimators, we use the function \texttt{fit.gpd} with its default settings from the \texttt{R} package \texttt{mev} \cite{mev_package}. The ancestral violation rate is computed by comparing the causal order inferred from Algorithm \ref{Alg:EASE} to the true DAG, using { $k=\lfloor n^{0.7}+\frac{1}{2}\rfloor$}
, and the penalty parameter in \eqref{eq:g_n(s,t)} is set to { $\lambda=4$}. { Note that the subsample size $m$ used for the marginal tail estimation needs not equal the subsample size $k$ used for the estimation of AAC. And we choose $m<k$ because the empirical cumulative distribution becomes less reliable only in the most extreme upper tail.}  

Table \ref{tab:mean_CI_n1000_d1} summarizes the simulation results for $\alpha_0 = 3$, comparing the performance of the AAC method to that of the causal tail coefficient (CTC) introduced in \cite{gnecco2021causal}. {  For CTC, we set $k=\lfloor n^{0.4}\rfloor$, following the choice recommended in \cite{gnecco2021causal} based on an extensive simulation study over a range of data-generating settings. Moreover, the CTC has been shown to identify the causal order under both the sum-linear SCM \citep{gnecco2021causal} and max-linear SCM \citep{routh2026heavy}. It is therefore of interest to compare the performance of AAC and CTC under both structural model specifications.} For AAC, we observe that it provides more accurate estimates of causal orders for the SL models than for the ML models, a pattern also seen with the CTC approach. {Moreover, compared to CTC, our AAC method consistently yields lower ancestral violation rates for both ML models, whereas CTC consistently performs better under both SL models. A similar pattern is observed when $\alpha_0=5$. By contrast, this pattern is less evident when $\alpha_0=1$. The simulation results for $\alpha_0=1$ and $5$ are reported in Section~\ref{append:sim} of the Appendix.} 

\begin{table}[ht]
{
\centering
\caption{\footnotesize{Simulation study with $\alpha_0=3$. Each numerical result is the average ancestral violation rate across $500$ simulation instances. The asterisk marks the better performing method between AAC (angular asymmetry coefficient) and CTC (causal tail coefficient).}}
\begin{scriptsize}
\begin{tabular}{c
                @{\hskip 0.4cm}l
                @{\hskip 0.3cm}l
                @{\hskip 0.4cm}l
                @{\hskip 0.3cm}l
                @{\hskip 0.4cm}l
                @{\hskip 0.3cm}l
                @{\hskip 0.4cm}l
                @{\hskip 0.3cm}l}
\toprule
$d$ & \multicolumn{2}{c}{SL0} & \multicolumn{2}{c}{ML0}
    & \multicolumn{2}{c}{SL1} & \multicolumn{2}{c}{ML1} \\
\cmidrule(lr){2-3} \cmidrule(lr){4-5}
\cmidrule(lr){6-7} \cmidrule(lr){8-9}
    & AAC & CTC & AAC & CTC & AAC & CTC & AAC & CTC \\
\midrule
5  & 0.0184 & {0.0040}*
   & {0.0914}* & 0.1354
   & 0.0169 & {0.0076}*
   & {0.0950}* & 0.1480 \\

10 & 0.0380 & {0.0161}*
   & {0.1372}* & 0.1928
   & 0.0379 & {0.0162}*
   & {0.1328}* & 0.1921 \\

15 & 0.0454 & {0.0175}*
   & {0.1514}* & 0.2146
   & 0.0466 & {0.0182}*
   & {0.1523}* & 0.2087 \\

30 & 0.0697 & {0.0271}*
   & {0.1969}* & 0.2450
   & 0.0710 & {0.0266}*
   & {0.1934}* & 0.2442 \\
\bottomrule
\end{tabular}
\label{tab:mean_CI_n1000_d1}
\end{scriptsize}
}
\end{table}

\subsection{River discharge data}\label{sec:river}

In this section, we apply Algorithm~\ref{Alg:EASE} to the river discharge data used in \cite{gnecco2021causal}, available via the \texttt{causalXtreme} package. The dataset contains $n=4600$ daily summer discharges from 12 stations along a river basin, pre-processed to reduce seasonality and temporal dependence.
Figure 7 of \cite{gnecco2021causal} provides a DAG representing the stations and river flow connections, while Figure 5 in their Supplementary Material shows a geographic map of the study area. The known river flow directions serve as ground truth for evaluating extremal causal directions. Additionally, \cite{gnecco2021causal} show that the data exhibits heavy tails with a common marginal tail index $\alpha$, satisfying the requirement in \eqref{eq:pareto marginal}.


Figure \ref{fig:river} (left) shows the ancestral violation rates for the causal order learned by the EASE algorithm using three approaches: (i) the CTC method from \cite{gnecco2021causal}; (ii) { the AAC computed from the marginally transformed data using the transformation in \eqref{eq:marg trans} with $\alpha=2$, as described in Section~\ref{sec:sim}, and with $m=\lfloor n^{0.5}+\frac{1}{2} \rfloor$ in \eqref{est_CDF_semi_para}}; and (iii) the AAC computed from data without marginal transformation. The ancestral violation rate is plotted against $k$, and the penalty parameter in \eqref{eq:g_n(s,t)} is chosen as {$\lambda=4$}. 

We observe that the AAC without marginal transformation consistently achieves 100\% accuracy in identifying the correct causal order across a substantial range of $k$. In addition, the AAC method with marginal transformation achieves  
stable accuracy as $k$
increases, performing comparably to the CTC method. {At their respective empirically-found optimal values of $k$ marked by the larger points, both CTC and AAC without marginal transformation achieve zero ancestral violation rate, whereas AAC with marginal transformation yields a higher ancestral violation rate.}

 \begin{figure}[ht]
   \centering
   \begin{minipage}[t]{0.49\textwidth}
     \centering
     \includegraphics[width=\linewidth]{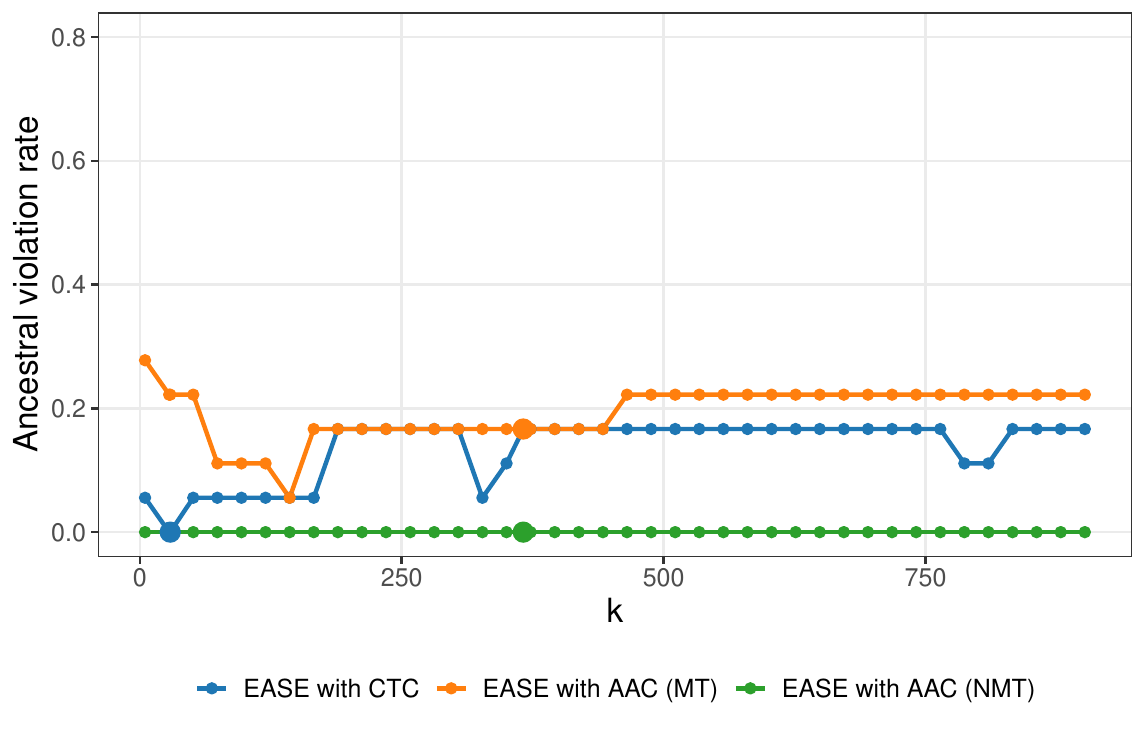}
   \end{minipage}
   \hfill
   \begin{minipage}[t]{0.49\textwidth}
     \centering
\includegraphics[width=\linewidth]{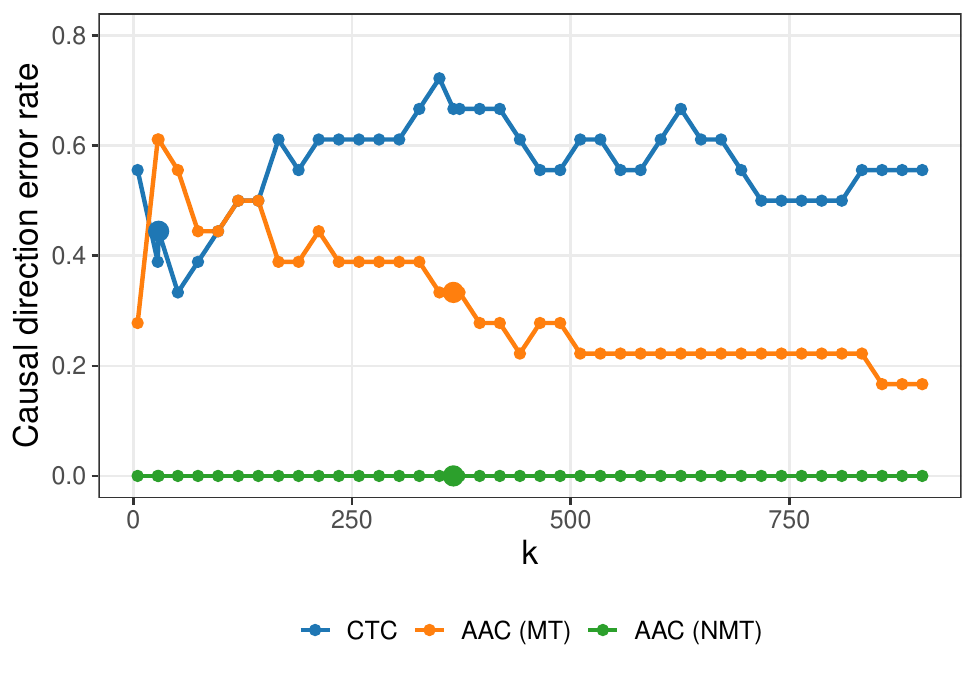}
   \end{minipage}
  \caption{\footnotesize{Left: ancestral violation rate for river discharge data. Right: pairwise causal direction identification error rate for river discharge data. 
  CTC: causal tail coefficient. AAC: angular asymmetry coefficient. MT: marginally transformed, NMT: not marginally transformed. {The larger points indicate the error rates obtained at the respective empirically selected optimal values of $k$ for CTC and AAC.}} }\label{fig:river}
 \end{figure}

Furthermore, for all 18 pairs of station nodes connected by a directed path (i.e., river flow), evaluate the accuracy with which AAC and CTC predict the true flow direction. This pairwise decision is more challenging than the discovery of causal order via Algorithm~\ref{Alg:EASE}: the latter exploits ranks and enjoys tolerance for potential errors in pairwise decisions.
Recall that for two nodes \( u \) and \( v \), under the setting of Corollary~\ref{Cor:cause effect 2}, the AAC satisfies \( \tau(u,v) > 0 > \tau(v,u) \) if \( u \) causes \( v \), with \( \tau(v,u) = -\tau(u,v) \). Meanwhile, for the CTC, \( \Gamma_{uv} \), Table~1 of \cite{gnecco2021causal} shows that \( \Gamma_{uv} > \Gamma_{vu} \) when \( u \) causes \( v \). Applying this rationale to predict flow directions yields the results shown in the right panel of Figure~\ref{fig:river}. The AAC without marginal transformation achieves perfect accuracy across all values of $k$. In comparison, the AAC with marginal transformation and the CTC show similar performance for small $k$, but as $k$ increases, the AAC with marginal transformation stabilizes at a lower error rate than the CTC. At their respective empirically selected optimal values of $k$, AAC without marginal transformation achieves zero causal direction error rate, whereas CTC and AAC with marginal transformation yield comparable error rates, with AAC under marginal transformation attaining a lower error rate than CTC.

The surprisingly perfect accuracy of the AAC without marginal transformation in both studies may be attributed to the inherent scaling differences in river discharge between upstream and downstream stations. In general, downstream discharge tends to be greater due to accumulated flow, and this magnitude difference is a meaningful signal for causal direction. Without applying a marginal transformation, the AAC retains this scale information, allowing the angular support $[a,b]$ to tilt toward the downstream variable, thus improving the accuracy of direction inference. However, marginal transformations normalize the data and may remove such valuable cues, leading to less stable performance.


\subsection{\texttt{CauseEffectPairs} benchmark}\label{sec:pair benchmark}
In this section, 
we apply Algorithm \ref{Alg:EASE} to the case $d=2$. This means that given 2 variables, we simply use the sign of estimated AAC  $\tau$ to identify which is the cause and which is the effect, as summarized in Figure \ref{fig:tau}. We shall test this out on the benchmark data
\texttt{CauseEffectPairs}  
\citep{mooij2016distinguishing}, which consists of real-life data pairs,  say each  of the form $(x_{1,i},x_{2,i})_{i=1}^{n}$,  where the ground truth of causal directions is provided. {  Our goal is to test whether some benchmark pairs contain directional information in their tails.} Here, we selected $94$ data sets out of the 108 available, excluding the categorical ones and the ones where $x_{1,i}$ or $x_{2,i}$ is vector-valued. 
Since it is possible that the causal relationship may manifest in different combinations of extremal directions, we shall consider the following 4 different combinations: $(z_{1,i},z_{2,i})=(x_{1,i},x_{2,i})$, $(-x_{1,i},x_{2,i})$, $(x_{1,i},-x_{2,i})$ or $(-x_{1,i},-x_{2,i})$. For each case, we then apply the same marginal transform as
in Section~\ref{sec:sim}. The extremal subsample size $k$ used for
estimation of AAC $\tau$ is given by
$k=\max(\lfloor n^{0.7}+\frac{1}{2}\rfloor,15)$,
and the penalty parameter in \eqref{eq:g_n(s,t)} is set to
$\lambda=4$. As the sample size varies across datasets,
we select $m$ in \eqref{est_CDF_semi_para} via a multiple-threshold goodness-of-fit procedure for the generalized Pareto distribution. Specifically, we employ the sequential testing procedure
implemented in the \texttt{gpdSeqTests} function of the
R package \texttt{eva} \cite{eva_package}. Writing $m_0=\lfloor n^{0.7}+\frac{1}{2}\rfloor$, for
$m_0>15$ we sequentially test thresholds corresponding
to the decreasing candidate tail counts
$m_0,m_0-h_m,m_0-2h_m,\ldots$, retaining those no smaller
than $15$, where $h_m=\max\left\{1,
\left\lfloor\frac{m_0-15}{50}+\frac{1}{2}\right\rfloor
\right\}$. Thus, the candidate thresholds are tested from lowest
to highest. For $m_0\leq15$, we test only the threshold
corresponding to the tail count $m_0$. Let $J=\max\bigl(\{0\}\cup
\{j:p_j^{\mathrm{SS}}\leq0.05\}\bigr)$, where $p_j^{\mathrm{SS}}$ is the StrongStop-adjusted
$p$-value for the $j$th candidate threshold. If $J$ is
smaller than the number of candidates, we select $m$ as
the number of strict exceedances above the $(J+1)$st
candidate threshold. Thus, when no candidate is rejected,
we retain the first candidate. If all candidates are
rejected, the procedure fails, or any adjusted $p$-value
is nonfinite, we set $m=m_0$. In addition, we set $m = 15$ whenever the procedure results in a value that is smaller than $15$. 
The accuracy is calculated by $\sum_{\ell=1}^{{94}} w_\ell \mbf{1}_{\pc{\text{correct for } \ell\text{th data pair}}}$, where the weights $w_\ell$'s  are supplied by \textsc{CauseEffectPairs} which we re-normalize so that $\sum_{\ell}w_\ell=1$.  

The results are summarized in Figure~\ref{fig:causeeffectpair}, where $95\%$ confidence intervals are computed using a normal approximation; {see Section \ref{append:sim_pairwise} in the Appendix for details.} The observed accuracies suggest that the AACs, particularly in the direction $(-x_{1,i},-x_{2,i})$, tend to agree with the true causal direction. The corresponding detection rate is comparable to the accuracy of $63\% \pm 10\%$ reported by \cite{mooij2016distinguishing} for ANM-pHSIC, a general-purpose causal discovery method evaluated on $100$ datasets.
Note that some combinations of extremal directions may exhibit
a weak extremal causal signal. In such cases, inference based
on the AAC sign may perform close to random guessing, as observed
in the directions $(-x_{1,i},x_{2,i})$ and $(x_{1,i},-x_{2,i})$. 

 \begin{figure}[ht]
    \centering
\includegraphics[width=0.5\linewidth]{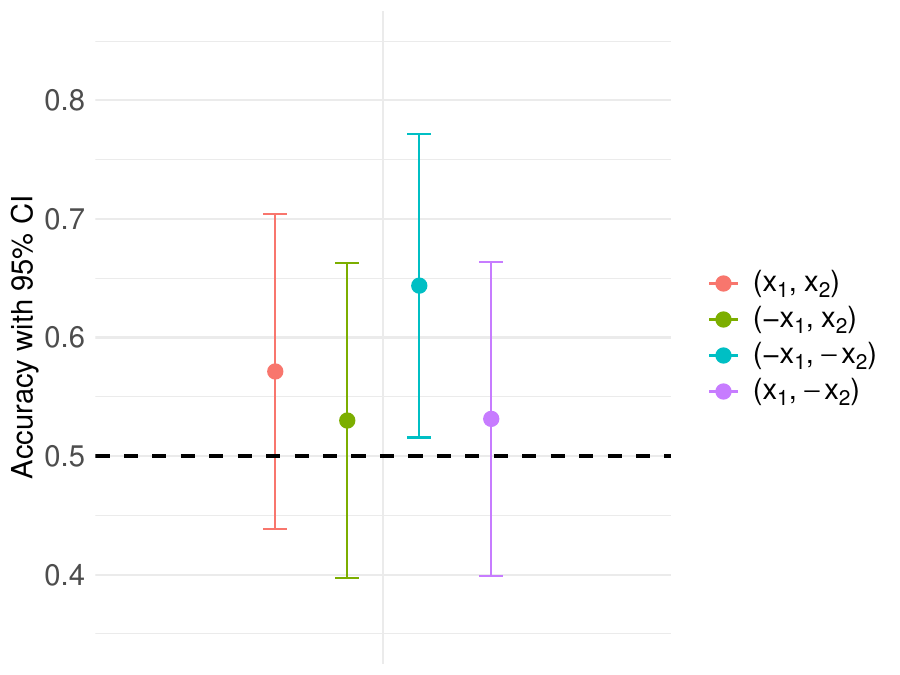}
    \caption{\footnotesize Accuracy of causal direction identification in four extremal directions.}
    \label{fig:causeeffectpair}
\end{figure}

\section{Summary}

In this paper, we propose a novel class of structural causal models for analyzing extreme values, the extremal structural causal models (eSCMs). Unlike classical SCMs, which model randomness via probability distributions, eSCMs are driven by exponent measures, infinite-mass measures that naturally arise in multivariate extreme value theory under multivariate regular variation. While eSCMs do not directly model the data-generating process, they capture asymptotic causal relationships among extreme values.

We show that eSCMs satisfy a well-defined causal Markov property based on extremal conditional independence, extending the link between structural equations and directed graphical models to the domain of extremes. We also identify a fundamental causal asymmetry inherent in the eSCM structure. Exploiting this asymmetry, we develop a consistent causal discovery algorithm tailored to the geometric and probabilistic features of extreme value behavior. 

We believe the eSCM framework offers a promising foundation for future research on causality in extreme values. Potential directions include: i) extending eSCMs to $\bb{R}^d\setminus \{\mbf{0}\}$ to handle two-sided extremes; ii) developing a comprehensive theory for interventional and counterfactual interpretations; and iii) designing statistical methods that leverage the extremal Markov property for causal discovery.

\bibliographystyle{imsart-number} 
\bibliography{reference}

@article{bai2026conditional,
  title={Conditional Independence under Infinite Measures and Poisson Point Processes},
  author={Bai, Shuyang and Routh, Vishal},
  journal={arXiv preprint arXiv:2604.01399},
  year={2026}
}

@book{last2017lectures,
  title     = {Lectures on the Poisson Process},
  author    = {Last, G{\"u}nter and Penrose, Mathew},
  series    = {Institute of Mathematical Statistics Textbooks},
  volume    = {7},
  publisher = {Cambridge University Press},
  address   = {Cambridge},
  year      = {2017},
  doi       = {10.1017/9781316104477}
}

@inproceedings{dawid2010beware,
  author    = {Dawid, A. Philip},
  title     = {Beware of the {DAG}!},
  booktitle = {Proceedings of Workshop on Causality: Objectives and Assessment at {NIPS} 2008},
  editor    = {Guyon, Isabelle and Janzing, Dominik and Sch{\"o}lkopf, Bernhard},
  series    = {Proceedings of Machine Learning Research},
  volume    = {6},
  pages     = {59--86},
  publisher = {PMLR},
  location  = {Whistler, Canada},
  date      = {2010},
  eventdate = {2008-12-12},
  url       = {https://proceedings.mlr.press/v6/dawid10a.html}
}

@Manual{mev_package,
  title        = {mev: Multivariate Extreme Value Distributions},
  author       = {Belzile, Léo},
  year         = {2024},
  note         = {R package version 1.17},
  url          = {https://CRAN.R-project.org/package=mev}
}

@Manual{eva_package,
    title = {eva: Extreme Value Analysis with Goodness-of-Fit Testing},
    author = {Brian Bader and Jun Yan},
    year = {2020},
    vers = {R package version 0.2.6},
  }

@article{zhou2024efficient,
  title={Efficient Learning of DAG Structures in Heavy-tailed Data},
  author={Zhou, Wei and Kang, Xueqian and Zhong, Wei and Wang, Junhui},
  journal={Statistica Sinica},
  volume={37},
  number={3},
  year={2024}
}

@article{amendola2025pc,
  title={A PC Algorithm for Max-Linear Bayesian Networks},
  author={Am{\'e}ndola, Carlos and Hollering, Benjamin and Nowell, Francesco},
  journal={arXiv preprint arXiv:2508.13967},
  year={2025}
}

@Manual{R2024,
    title = {R: A Language and Environment for Statistical Computing},
    author = {{R Core Team}},
    organization = {R Foundation for Statistical Computing},
    address = {Vienna, Austria},
    year = {2024},
    url = {https://www.R-project.org/},
  }

@Article{Kalisch2012Causal,
    title = {Causal Inference Using Graphical Models with the {R}
      Package {pcalg}},
    author = {{Markus Kalisch} and {Martin M\"achler} and {Diego
      Colombo} and {Marloes H. Maathuis} and {Peter B\"uhlmann}},
    journal = {Journal of Statistical Software},
    year = {2012},
    volume = {47},
    number = {11},
    pages = {1--26},
    doi = {10.18637/jss.v047.i11},
  }

@article{mooij2016distinguishing,
  title={Distinguishing cause from effect using observational data: methods and benchmarks},
  author={Mooij, Joris M and Peters, Jonas and Janzing, Dominik and Zscheischler, Jakob and Sch{\"o}lkopf, Bernhard},
  journal={Journal of Machine Learning Research},
  volume={17},
  number={32},
  pages={1--102},
  year={2016}
}

@article{kluppelberg2025causal,
  title={Causal analysis of extreme risk in a network of industry portfolios},
  author={Kl{\"u}ppelberg, Claudia and Krali, Mario},
  journal={arXiv preprint arXiv:2504.00523},
  year={2025}
}

@article{adams2025inference,
  title={Inference for max-linear Bayesian networks with noise},
  author={Adams, Mark and Ferry, Kamillo and Yoshida, Ruriko},
  journal={arXiv preprint arXiv:2505.00229},
  year={2025}
}

@article{jiang2025separation,
  title={Separation-based causal discovery for extremes},
  author={Jiang, Junshu and Richards, Jordan and Huser, Rapha{\"e}l and Bolin, David},
  journal={arXiv preprint arXiv:2505.08008},
  year={2025}
}

@article{WangResnick2024,
  author    = {Tiandong Wang and Sidney I. Resnick},
  title     = {Distinguishing Forms of Asymptotic Dependence in Heavy Tailed Data},
  journal   = {Statistica Sinica},
  year      = {2024},
  note      = {To appear.},
  doi       = {10.5705/ss.2022024.0196},
  author_email = {td_wang@fudan.edu.cn}
}

@book{resnick2024art,
  title={The Art of Finding Hidden Risks: Hidden Regular Variation in the 21st Century},
  author={Resnick, Sidney},
  year={2024},
  publisher={Springer Nature}
}

@article{buck2021recursive,
  title={Recursive max-linear models with propagating noise},
  author={Buck, Johannes and Kl{\"u}ppelberg, Claudia},
  journal={Electronic Journal of Statistics},
  volume={15},
  number={2},
  pages={4770--4822},
  year={2021},
  publisher={The Institute of Mathematical Statistics and the Bernoulli Society}
}

@article{kiriliouk2019peaks,
  title={Peaks over thresholds modeling with multivariate generalized Pareto distributions},
  author={Kiriliouk, Anna and Rootz{\'e}n, Holger and Segers, Johan and Wadsworth, Jennifer L},
  journal={Technometrics},
  volume={61},
  number={1},
  pages={123--135},
  year={2019},
  publisher={Taylor \& Francis}
}

@article{rootzen2018multivariate,
  title={Multivariate generalized Pareto distributions: Parametrizations, representations, and properties},
  author={Rootz{\'e}n, Holger and Segers, Johan and Wadsworth, Jennifer L},
  journal={Journal of Multivariate Analysis},
  volume={165},
  pages={117--131},
  year={2018},
  publisher={Elsevier}
}

@article{rootzen2006multivariate,
  title={Multivariate generalized Pareto distributions},
  author={Rootz{\'e}n, Holger and Tajvidi, Nader},
  journal={Bernoulli},
  volume={12},
  number={5},
  pages={917--930},
  year={2006},
  publisher={Bernoulli Society for Mathematical Statistics and Probability}
}

@article{krali2025causal,
  title={Causal discovery in heavy-tailed linear structural equation models via scalings},
  author={Krali, Mario},
  journal={arXiv preprint arXiv:2502.13762},
  year={2025}
}

@article{engelke2025extremes,
  title={Extremes of structural causal models},
  author={Engelke, Sebastian and Gnecco, Nicola and R{\"o}ttger, Frank},
  journal={arXiv preprint arXiv:2503.06536},
  year={2025}
}

@article{lauritzen1990independence,
  title={Independence properties of directed Markov fields},
  author={Lauritzen, Steffen L and Dawid, A Philip and Larsen, Birgitte N and Leimer, H-G},
  journal={Networks},
  volume={20},
  number={5},
  pages={491--505},
  year={1990},
  publisher={Wiley Online Library}
}

@article{bongers2021foundations,
  title={Foundations of structural causal models with cycles and latent variables},
  author={Bongers, Stephan and Forr{\'e}, Patrick and Peters, Jonas and Mooij, Joris M},
  journal={The Annals of Statistics},
  volume={49},
  number={5},
  pages={2885--2915},
  year={2021},
  publisher={Institute of Mathematical Statistics}
}

@book{peters2017elements,
  title={Elements of Causal Inference: Foundations and Learning Algorithms},
  author={Peters, Jonas and Janzing, Dominik and Sch{\"o}lkopf, Bernhard},
  year={2017},
  publisher={The MIT Press}
}

@article{kluppelberg2021estimating,
  title={Estimating an extreme Bayesian network via scalings},
  author={Kl{\"u}ppelberg, Claudia and Krali, Mario},
  journal={Journal of Multivariate Analysis},
  volume={181},
  pages={104672},
  year={2021},
  publisher={Elsevier}
}

@article{chavez2024causality,
  title={Causality and extremes},
  author={Chavez-Demoulin, Val{\'e}rie and Mhalla, Linda},
  journal={arXiv preprint arXiv:2403.05331},
  year={2024}
}

@article{bodik2024causality,
  title={Causality in extremes of time series},
  author={Bodik, Juraj and Palu{\v{s}}, Milan and Pawlas, Zbyn{\v{e}}k},
  journal={Extremes},
  volume={27},
  number={1},
  pages={67--121},
  year={2024},
  publisher={Springer}
}

@article{tran2024estimating,
  title={Estimating a directed tree for extremes},
  author={Tran, Ngoc Mai and Buck, Johannes and Kl{\"u}ppelberg, Claudia},
  journal={Journal of the Royal Statistical Society Series B: Statistical Methodology},
  pages={qkad165},
  year={2024},
  publisher={Oxford University Press US}
}

@article{pcalg,
    title = {Causal Inference Using Graphical Models with the {R} Package {pcalg}},
    author = {{Markus Kalisch} and {Martin M\"achler} and {Diego Colombo} and {Marloes H. Maathuis} and {Peter B\"uhlmann}},
    journal = {Journal of Statistical Software},
    year = {2012},
    volume = {47},
    number = {11},
    pages = {1--26},
    doi = {10.18637/jss.v047.i11}
}

@article{duncan2011genome,
  title={Genome-wide association study using extreme truncate selection identifies novel genes affecting bone mineral density and fracture risk},
  author={Duncan, Emma L and Danoy, Patrick and Kemp, John P and Leo, Paul J and McCloskey, Eugene and Nicholson, Geoffrey C and Eastell, Richard and Prince, Richard L and Eisman, John A and Jones, Graeme and others},
  journal={PLoS genetics},
  volume={7},
  number={4},
  pages={e1001372},
  year={2011},
  publisher={Public Library of Science San Francisco, USA}
}

@article{zanin2016causality,
  title={On causality of extreme events},
  author={Zanin, Massimiliano},
  journal={PeerJ},
  volume={4},
  pages={e2111},
  year={2016},
  publisher={PeerJ Inc.}
}

@article{zhang2012causal,
  title={Causal inference on quantiles with an obstetric application},
  author={Zhang, Zhiwei and Chen, Zhen and Troendle, James F and Zhang, Jun},
  journal={Biometrics},
  volume={68},
  number={3},
  pages={697--706},
  year={2012},
  publisher={Wiley Online Library}
}

@article{chernozhukov2011inference,
  title={Inference for extremal conditional quantile models, with an application to market and birthweight risks},
  author={Chernozhukov, Victor and Fern{\'a}ndez-Val, Iv{\'a}n},
  journal={The Review of Economic Studies},
  volume={78},
  number={2},
  pages={559--589},
  year={2011},
  publisher={Oxford University Press}
}

@article{sun2021causal,
  title={Causal inference for quantile treatment effects},
  author={Sun, Shuo and Moodie, Erica EM and Ne{\v{s}}lehov{\'a}, Johanna G},
  journal={Environmetrics},
  volume={32},
  number={4},
  pages={e2668},
  year={2021},
  publisher={Wiley Online Library}
}

@article{chuang2009causality,
  title={Causality in quantiles and dynamic stock return--volume relations},
  author={Chuang, Chia-Chang and Kuan, Chung-Ming and Lin, Hsin-Yi},
  journal={Journal of Banking \& Finance},
  volume={33},
  number={7},
  pages={1351--1360},
  year={2009},
  publisher={Elsevier}
}

@article{mhalla2020causal,
  title={Causal mechanism of extreme river discharges in the upper Danube basin network},
  author={Mhalla, Linda and Chavez-Demoulin, Val{\'e}rie and Dupuis, Debbie J},
  journal={Journal of the Royal Statistical Society Series C: Applied Statistics},
  volume={69},
  number={4},
  pages={741--764},
  year={2020},
  publisher={Oxford University Press}
}

@book{kulik2020heavy,
  title={Heavy-tailed time series},
  author={Kulik, Rafal and Soulier, Philippe},
  year={2020},
  publisher={Springer}
}

@article{pasche2023causal,
  title={Causal modelling of heavy-tailed variables and confounders with application to river flow},
  author={Pasche, Olivier C and Chavez-Demoulin, Val{\'e}rie and Davison, Anthony C},
  journal={Extremes},
  volume={26},
  number={3},
  pages={573--594},
  year={2023},
  publisher={Springer}
}

@article{krali2023heavy,
  title={Heavy-tailed max-linear structural equation models in networks with hidden nodes},
  author={Krali, Mario and Davison, Anthony C and Kl{\"u}ppelberg, Claudia},
  journal={arXiv preprint arXiv:2306.15356},
  year={2023}
}

@article{asenova2022max,
  title={Max-linear graphical models with heavy-tailed factors on trees of transitive tournaments},
  author={Asenova, Stefka and Segers, Johan},
  journal={arXiv preprint arXiv:2209.14938},
  year={2022}
}

@inproceedings{amendola2021markov,
  title={Markov equivalence of max-linear Bayesian networks},
  author={Am{\'e}ndola, Carlos and Hollering, Benjamin and Sullivant, Seth and Tran, Ngoc},
  booktitle={Uncertainty in Artificial Intelligence},
  pages={1746--1755},
  year={2021},
  organization={PMLR}
}

@article{gissibl2021identifiability,
  title={Identifiability and estimation of recursive max-linear models},
  author={Gissibl, Nadine and Kl{\"u}ppelberg, Claudia and Lauritzen, Steffen},
  journal={Scandinavian Journal of Statistics},
  volume={48},
  number={1},
  pages={188--211},
  year={2021},
  publisher={Wiley Online Library}
}

@article{gnecco2021causal,
  title={Causal discovery in heavy-tailed models},
  author={Gnecco, Nicola and Meinshausen, Nicolai and Peters, Jonas and Engelke, Sebastian},
  journal={The Annals of Statistics},
  volume={49},
  number={3},
  pages={1755--1778},
  year={2021},
  publisher={Institute of Mathematical Statistics}
}

@article{amendola2022conditional,
  title={Conditional independence in max-linear Bayesian networks},
  author={Am{\'e}ndola, Carlos and Kl{\"u}ppelberg, Claudia and Lauritzen, Steffen and Tran, Ngoc M},
  journal={The Annals of Applied Probability},
  volume={32},
  number={1},
  pages={1--45},
  year={2022},
  publisher={Institute of Mathematical Statistics}
}

@article{gissibl2018max,
  title={Max-linear models on directed acyclic graphs},
  author={Gissibl, N and Kl{\"u}ppelberg, C},
  journal={Bernoulli},
  volume={24},
  number={4A},
  pages={2693--2720},
  year={2018}
}

@article{zhang2012kernel,
  title={Kernel-based conditional independence test and application in causal discovery},
  author={Zhang, Kun and Peters, Jonas and Janzing, Dominik and Sch{\"o}lkopf, Bernhard},
  journal={arXiv preprint arXiv:1202.3775},
  year={2012}
}

@book{spirtes2000causation,
  title={Causation, prediction, and search},
  author={Spirtes, Peter and Glymour, Clark N and Scheines, Richard},
  edition={2},
  year={2000}
}

@article{glymour2019review,
  title={Review of causal discovery methods based on graphical models},
  author={Glymour, Clark and Zhang, Kun and Spirtes, Peter},
  journal={Frontiers in genetics},
  volume={10},
  pages={524},
  year={2019},
  publisher={Frontiers Media SA}
}

@book{pearl2009causality,
  title={Causality},
  edition={2nd},
  author={Pearl, Judea},
  year={2009},
  publisher={Cambridge University Press}
}

@article{engelke2025graphical,
  author       = {Sebastian Engelke and Jevgenijs Ivanovs and Kirstin Strokorb},
  title        = {Graphical Models for Infinite Measures With Applications to Extremes and Lévy Processes},
  journal      = {Annals of Applied Probability},
  note         = {To appear},
  year         = {2025}
}

@article{engelke2020graphical,
  title={Graphical models for extremes},
  author={Engelke, Sebastian and Hitz, Adrien S},
  journal={Journal of the Royal Statistical Society Series B: Statistical Methodology},
  volume={82},
  number={4},
  pages={871--932},
  year={2020},
  publisher={Oxford University Press}
}

@book{resnick2007heavy,
  title={Heavy-Tail Phenomena: Probabilistic and Statistical Modeling},
  author={Resnick, Sidney I},
  year={2007},
  publisher={Springer Science \& Business Media}
}

@book{beirlant2006statistics,
  title={Statistics of Extremes: Theory and Applications},
  author={Beirlant, Jan and Goegebeur, Yuri and Segers, Johan and Teugels, Jozef L},
  year={2006},
  publisher={John Wiley \& Sons}
}

@book{lauritzen1996graphical,
  title={Graphical Models},
  author={Lauritzen, Steffen L},
  volume={17},
  year={1996},
  publisher={Clarendon Press}
}

@book{bingham:1989:regular,
  title={Regular Variation},
  author={Bingham, Nicholas H and Goldie, Charles M and Teugels, Jef L},
  volume={27},
  year={1989},
  publisher={Cambridge university press}
}

@book{kallenberg:2021:foundations,
  title={Foundations of Modern Probability},
  author={Kallenberg, Olav},
  year={2021},
  edition={3},
  publisher={Springer Science \& Business Media}
}

@article{park2020identifiability,
  title={Identifiability of additive noise models using conditional variances},
  author={Park, Gunwoong},
  journal={Journal of Machine Learning Research},
  volume={21},
  number={75},
  pages={1--34},
  year={2020}
}

@article{routh2026heavy,
  title={Heavy Tailed Homogeneous Structural Causal Models},
  author={Routh, Vishal and Bai, Shuyang},
  journal={arXiv preprint arXiv:2604.04118},
  year={2026}
}

\newpage
\begin{appendix}




\renewcommand{\thesection}{A.\arabic{section}}
\setcounter{section}{0}  
{ 
\section{Preliminary results for selecting the subsample size $k$}
\label{NR:choose_k}
In this section, we investigate the choice of the subsample size  $k=\left\lfloor n^a+\frac{1}{2}\right\rfloor$, where $a\in\{0.20,0.25,\ldots,0.80\}$. We conduct a preliminary simulation study over a broad range of data-generating settings. Specifically, we consider sample sizes $n\in\{500,1000,2000\}$, numbers of nodes $d\in\{5,10,15,30\}$, and shape parameters $\alpha_0\in\{1,3,5\}$. We further consider the four probabilistic SCMs SL0, SL1, ML0, and ML1 introduced in Section~\ref{sec:sim}. Throughout this study, we set $\lambda=4$ in \eqref{eq:g_n(s,t)} and $\alpha=2$ in the marginal transform 
\eqref{eq:marg trans}. For each combination of $(n,d,\alpha_0)$ and probabilistic SCM, we generate $M=50$ Monte Carlo replications. For each replication and each candidate value of $a$, we compute the ancestral violation rate defined in Section~\ref{Sec:simdata}. We then average the ancestral violation rate over the Monte Carlo replications and over all data-generating settings. The resulting average ancestral violation rates are displayed in Figure~\ref{fig:choose_k}.
 \begin{figure}[ht]
    \centering
\includegraphics[width=0.5\linewidth]{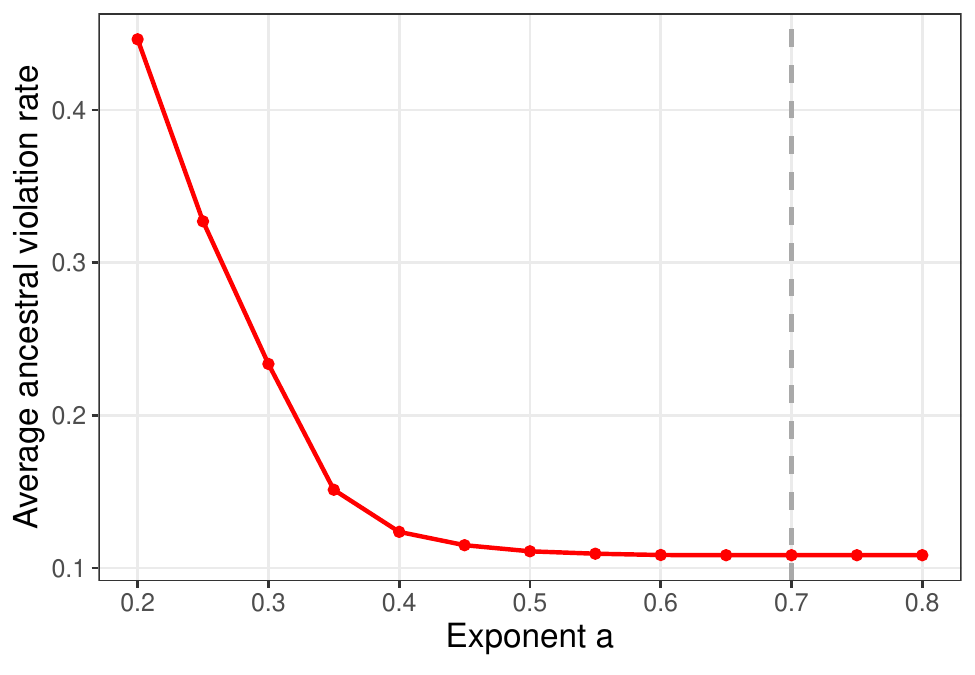}
    \caption{\footnotesize Average ancestral violation rate as a function of the exponent $a$ in $k=\left\lfloor n^a+\frac{1}{2}\right\rfloor$. The vertical gray line indicates the value {$a=0.7$}, which attains the minimum average ancestral violation rate.}
    \label{fig:choose_k}
\end{figure}

As shown in Figure~\ref{fig:choose_k}, the average ancestral violation rate is minimized at $a=0.7$. We therefore set $k=\left\lfloor n^{0.7}+\frac{1}{2}\right\rfloor$ throughout the numerical studies in Section~\ref{Sec:simdata} and Appendix~\ref{append:sim}.
}

\section{Additional numerical results}\label{append:sim}

\subsection{Additional simulation results}

In this section, we provide additional simulation results that complement those presented in Section~\ref{sec:sim}. Specifically, we vary the tail parameter $\alpha_0$ of the $\zeta_v$ variables of the models \eqref{eq:sum_linear sim} and \eqref{eq:max_linear sim}, setting $\alpha_0 = 1$ and $\alpha_0 = 5$. In Section~\ref{sec:sim}, the results correspond to   $\alpha_0=3$.

\begin{table}[H]
{
\centering
\caption{\footnotesize{Simulation study with $\alpha_0=1$. Each numerical result is in the form of average ancestral violation rate across $500$ simulation instances.}}
\begin{scriptsize}
\begin{tabular}{c
                @{\hskip 0.4cm}l
                @{\hskip 0.3cm}l
                @{\hskip 0.4cm}l
                @{\hskip 0.3cm}l
                @{\hskip 0.4cm}l
                @{\hskip 0.3cm}l
                @{\hskip 0.4cm}l
                @{\hskip 0.3cm}l}
\toprule
$d$ & \multicolumn{2}{c}{SL0} & \multicolumn{2}{c}{ML0}
    & \multicolumn{2}{c}{SL1} & \multicolumn{2}{c}{ML1} \\
\cmidrule(lr){2-3} \cmidrule(lr){4-5}
\cmidrule(lr){6-7} \cmidrule(lr){8-9}
    & AAC & CTC & AAC & CTC & AAC & CTC & AAC & CTC \\
\midrule
5
& {0.0000}* & {0.0000}*
& {0.0000}* & 0.0001
& {0.0000}* & {0.0000}*
& {0.0000}* & 0.0001 \\

10
& {0.0002}* & 0.0004
& 0.0013 & {0.0002}*
& {0.0003}* & {0.0003}*
& {0.0002}* & 0.0007 \\

15
& 0.0004 & {0.0001}*
& 0.0002 & {0.0001}*
& 0.0005 & {0.0003}*
& 0.0010 & {0.0005}* \\

30
& 0.0006 & {0.0004}*
& 0.0009 & {0.0007}*
& 0.0010 & {0.0008}*
& 0.0010 & {0.0005}* \\
\bottomrule
\end{tabular}
\label{tab:mean_CI_alpha1}
\end{scriptsize}
}
\end{table}

\begin{table}[H]
{
\centering
\caption{\footnotesize{Simulation study with $\alpha_0=5$. Each numerical result is in the form of average ancestral violation rate across $500$ simulation instances.}}
\begin{scriptsize}
\begin{tabular}{c
                @{\hskip 0.4cm}l
                @{\hskip 0.3cm}l
                @{\hskip 0.4cm}l
                @{\hskip 0.3cm}l
                @{\hskip 0.4cm}l
                @{\hskip 0.3cm}l
                @{\hskip 0.4cm}l
                @{\hskip 0.3cm}l}
\toprule
$d$ & \multicolumn{2}{c}{SL0} & \multicolumn{2}{c}{ML0}
    & \multicolumn{2}{c}{SL1} & \multicolumn{2}{c}{ML1} \\
\cmidrule(lr){2-3} \cmidrule(lr){4-5}
\cmidrule(lr){6-7} \cmidrule(lr){8-9}
    & AAC & CTC & AAC & CTC & AAC & CTC & AAC & CTC \\
\midrule
5
& 0.0445 & {0.0125}*
& {0.2757}* & 0.3668
& 0.0408 & {0.0180}*
& {0.2934}* & 0.3940 \\

10
& 0.0821 & {0.0333}*
& {0.3619}* & 0.4122
& 0.0806 & {0.0325}*
& {0.3525}* & 0.4305 \\

15
& 0.0988 & {0.0384}*
& {0.3691}* & 0.4371
& 0.1032 & {0.0398}*
& {0.3759}* & 0.4491 \\

30
& 0.1393 & {0.0532}*
& {0.4218}* & 0.4726
& 0.1339 & {0.0499}*
& {0.4130}* & 0.4621 \\
\bottomrule
\end{tabular}
\label{tab:mean_CI_n1000_shape_5_strong_sig}
\end{scriptsize}
}
\end{table}

\subsection{Normal approximation for the weighted accuracy on the \texttt{CauseEffectPairs} benchmark}
\label{append:sim_pairwise}
{ We assume that
\(\mbf{1}_{\pc{\text{correct for } \ell\text{th data pair}}}\),
\(\ell=1,\ldots,94\), are independently and identically distributed as
\(\operatorname{Bernoulli}(p)\).  We have the normal approximation $$\sigma_n^{-1}
\left\{
\sum_{\ell=1}^{94} w_\ell
\mbf{1}_{\pc{\text{correct for } \ell\text{th data pair}}}
-p
\right\}
\sim N(0,1),$$ where $\sigma_n^2
=
\operatorname{var}\left[
\sum_{\ell=1}^{94} w_\ell
\left\{
\mbf{1}_{\pc{\text{correct for } \ell\text{th data pair}}}
-p
\right\}
\right]$.
We estimate \(\sigma_n^2\) by
$$\hat{\sigma}_n^2
=
\sum_{\ell=1}^{94} w_\ell^2
\left\{
\mbf{1}_{\pc{\text{correct for } \ell\text{th data pair}}}
-\hat p
\right\}^2,$$ where $\hat p
=
\sum_{\ell=1}^{94} w_\ell
\mbf{1}_{\pc{\text{correct for } \ell\text{th data pair}}}$.}

\section{Proofs and technical discussions}\label{Sec:proof and tech disc}
\subsection{Proof of Proposition \ref{Pro:basic}}\label{sec:pf Pro:basic}
We use $\E_{\sbf{\theta}} $ to denote integration (taking expectation) with respect to $\Prt_{\sbf{\theta}}$.
In view of $\mu\pp{ \pp{\boldsymbol{\eta},\boldsymbol{\theta} }\in \cdot }=(\Lambda^{\perp}\otimes \Prt_{\sbf{\theta}})(\cdot)$, we have by measure-theoretic change of variable \cite[Lemma 1.24]{kallenberg:2021:foundations} and Fubini's theorem that
\begin{align}
    &\mu(Y_v>1) = \E_{\sbf{\theta}} \pb{ \int_{\bb{E}_V }    \mbf{1}_{\pc{ F_{\cl{A}(v)}\pp{  \mbf{X}_{\mrm{An}(v)}, \sbf{\theta}_{\mrm{An}(v)} }>1}}  \Lambda^{\perp}(\mbf{x})}=\notag\\
&\sum_{u\in \An(v)}  s \alpha  \E_{\sbf{\theta}}\pb{  \int_0^\infty     \mbf{1}_{\pc{   F_{\cl{A}(v)}\pp{  \pp{x\mbf{1}_{\{w=u\}}}_{w\in \mrm{An}(v)}, \sbf{\theta}_{\mrm{An}(v)} }>1}} x^{-\alpha-1} dx}, \label{eq:mu(Yv>1) decomp}
\end{align}
where in the last equality we have used    the fact that $\Lambda^{\perp}$ is supported on the coordinate axes $\bb{A}_V$ and \eqref{eq:Lambda_perp}. Then by the homogeneity of $F_{\cl{A}(v)}(\cdot,\sbf{\theta}_{\mrm{An}(v)})$ implied by Condition 1 of Definition \ref{def:eSCM},  we have   $F_{\cl{A}(v)}\pp{  \pp{x\mbf{1}_{\{w=u\}}}_{w\in \mrm{An}(v)}, \sbf{\theta}_{\mrm{An}(v)} } =x F_{\cl{A}(v)}\pp{  \pp{\mbf{1}_{\{w=u\}}}_{w\in \mrm{An}(v)}, \sbf{\theta}_{\mrm{An}(v)} } $ for all $x>0$ and $\sbf{\theta}_{\mrm{An}(v)}\in [0,1]^{\mrm{An}(v)}$. The relation \eqref{eq:s_v moment} then follows from  substituting this relation into \eqref{eq:mu(Yv>1) decomp} and the fact that $\int_{0}^\infty \mbf{1}_{\{a x>1 \}}  \alpha x^{-\alpha-1}dx=a^\alpha$ for $a\ge 0$.

For the second claim, the relation \eqref{eq:Lambda marginal} with $\Lambda=\cl{L}(\mbf{Y})$ follows readily from Condition 2 of Definition \ref{def:eSCM}.  To verify \eqref{eq:Lambda homo},  it suffices to show for any Borel $B\subset \bb{E}_V$   that
$
  \mu( \mbf{Y}\in c B ) = c^{-\alpha}   \mu( \mbf{Y}\in  B )$,  $c\in (0,\infty)$.
To show this, we have similarly as above that
\[
  \mu( \mbf{Y}\in c B ) =\E_{\sbf{\theta}} \pb{ \int_{\bb{E}_V }    \mbf{1}_{\pc{ \mbf{F}_{\cl{G}}\pp{ c^{-1} \mbf{x}, \sbf{\theta} }\in  B}}  \Lambda^{\perp}(d\mbf{x})},
\] 
where we have used the homogeneity of $\mbf{F}_{\cl{G}}(\cdot,\sbf{\theta})$ implied by Condition 1 of Definition \ref{def:eSCM}. The desirable relation then follows from the homogeneity property $\Lambda^{\perp}(\cdot) =c^{-\alpha} \Lambda^{\perp}(c^{-1}\cdot)$ and a change of variable.

Now we prove the last claim. Since all norms are equivalent on a finite-dimensional space,  for convenience, we assume $\|\cdot\| = \|\cdot\|_\infty$. Then applying the sufficient condition in the proposition,
we claim that 
\begin{equation}\label{eq:Y_v upper bound recur}
Y_v=a_v\eta_v+ h_v\pp{ \mbf{Y}_{\pa(v)},\theta_v }\le C_v^*(\theta_v) \|\pp{\eta_v,\mbf{Y}_{\pa(v)}}\|_\infty,   \quad \mu-a.e.
\end{equation}
for some $C_v^*(\theta_v)\ge 1$ with 
\begin{equation}\label{eq:C_v^* alpha mom}
\E_{\sbf{\theta}} |C_v^*(\theta_v)|^{\alpha}<\infty.
\end{equation}
To see this, it suffices to take $C_v^*(\theta_v)=1\vee a_v\vee C_v(\theta_v)$, where $C_v(\theta_v)$ is as in the assumption, and $a_v$ is the activation coefficient, and to note that $\mu\left(\eta_v>0, \mbf{Y}_{\pa(v)}\neq \mbf{0}_{\pa(v)} \right)=0$. Then, by a recursion of  \eqref{eq:Y_v upper bound recur} tracing back through ancestral relations and the fact that $C_v^*(\theta_v)\ge 1$, we have 
\begin{equation} 
Y_v \le C_{\An(v)}^*\pp{ \sbf{\theta}_{\An(v)}}\|\sbf{\eta}_{\An(v)}\|_\infty \quad \mu-a.e.,
\end{equation}
where $ C_{\An(v)}^*\pp{ \sbf{\theta}_{\An(v)}}:= \pp{ \prod_{u\in \An(v)}C_u^*(\theta_u)}  $ satisfies $\E_{\sbf{\theta}}\pb{C_{\An(v)}^*\pp{ \sbf{\theta}_{\An(v)}}^\alpha}<\infty$  due to \eqref{eq:C_v^* alpha mom} and independence of $\theta_u$'s. So applying Fubini similarly as above and the single-activation nature of $\eta_u$'s,
\begin{align*}
\mu(Y_v>1)&\le  \E_{\sbf{\theta}}\pb{C_{\An(v)}^*\pp{ \sbf{\theta}_{\An(v)}}^\alpha} \sum_{u\in\An(v)}\mu\left( \eta_u>1\right) \\&= s |\An(v)| \E_{\sbf{\theta}}\pb{C_{\An(v)}^*\pp{ \sbf{\theta}_{\An(v)}}^\alpha} 
<\infty.
\end{align*}

\subsection{Generalized Pareto representation for the law of H\"usler-Reiss eSCM}\label{sec:pf mult Pareto HR}

Throughout the discussion, we assume $\alpha = 1$ for convenience of comparison with the literature. This does not entail a loss of generality, as the case $\alpha \neq 1$ can be easily reduced to $\alpha = 1$ via a transformation.

Following Example \ref{Eg:HR eSCM}, suppose node $1$ is the unique root node and the associated activation coefficient $a_1>0$, and $a_v=0$ for $v\in \de(1)=\{2,\ldots,d\}$.    Let the matrix $B= \pp{b_{uw} }_{u,w\in V}$, where $u$ indexes rows and $w$ indexes columns, and $b_{uw}=0$ if $u\notin \pa(w)$. Note that $b_{uw}$ can be negative if $|\pa(w)|\ge 2$. Set $\mbf{W}=(W_u)_{u\in V}:=\pp{\log(Y_u)}_{u\in V}$, and  $\mbf{Z}=\pp{Z_u}_{u\in\{2,\ldots,d\}}$, recalling the latter under $\Prt_\theta$ is a multivariate Gaussian with mean $\sbf{\mu}_{\de(1)}$ and covariance matrix $\Sigma_{\de(1)}=\mathrm{Diag}\pp{\sigma_s^2, s\in \de(1)}$.
In view of \eqref{eq:h_v HR},  under $\{\eta_1>0\}$, the sub-eSCM in \eqref{eq:single eta} in this case can be written as 
$$\mbf{W}  = B^\top \mbf{W} +  \mbf{N},$$
where   $\mbf{N}$ is a $V$-indexed vector with the $1$st component   $\log(a_1\eta_1)$ and $(2,\ldots,d)$-component $\mbf{Z}$.  Note that the $1$st row of $B^\top$ is zero. Following \cite{engelke2025extremes},  one can then re-express the last displayed relation   as
\begin{align}\label{eq:HR eSCM re-exp}
   \mbf{W} =\begin{pmatrix}
    W_1   \\
    \mbf{W}_{\de(1)}
   \end{pmatrix} 
   =\left(I -B^\top\right)^{-1} \mbf{N}
   = \begin{pmatrix}
     \mbf{e}_1^\top   \\
    L 
   \end{pmatrix} 
  \mbf{N} 
=\begin{pmatrix}
   1 &\mbf{0}_{\de(1)}^\top     \\
 \mbf{c} & D 
   \end{pmatrix} 
   \begin{pmatrix}
 \log (a_1 \eta_1)  \\
\mbf{Z} 
\end{pmatrix}
\end{align}
where $\mbf{e}_1=(1,0,\ldots,0)^\top$ is the  coordinate unit vector, $L$ is the $(2,\ldots,d)$-rows of  $\left(I -B^\top\right)^{-1}$ with  $I$ denoting  the identity matrix.  Here, each $c_{u}$ in $\mbf{c}=\pp{c_{u}}_{u\in \{2,\ldots,d\}}$ is the sum of distinct $B$-weighted directed paths (i.e., product of the edge weights in $B$ along a directed path) from node $1$ to node $u$, and 
each $d_{u,w}$ in $D=\pp{d_{u,w}}_{u,w\in \{2,\ldots,d\}}=:\pp{\mbf{d}_u^\top }_{u\in \{2,\ldots,d\}}$ ($u$ index rows)  is  the sum of distinct  $B$-weighted directed paths from node $w$ to node $u$.   

First, we claim that, due to the assumption 
\begin{equation}\label{eq:buw sum to 1}
    \sum_{u\in \pa(w)} b_{uw}=1,\quad w\in \{2,\ldots,d\},
\end{equation} we have 
\begin{equation}\label{eq:c 1 vector}
\mbf{c}=\pp{1,\ldots,1}^\top\in  \bb{R}^{\{2,\ldots,d\}}.
\end{equation}
Indeed, this follows from an induction argument. Let $G_0=\{1\}$ and $c_1=1$. For $k \ge 1$, define the $k$th generation $G_k$ as the set of nodes $v$ such that the maximum number of edges in any directed path from a root node to $v$ is exactly $k$. Since the DAG is finite, we can assume for some integer $K\ge 0$ that $G_k\neq \emptyset$ for $0\le k\le K$,  and $G_k=\emptyset$ when $k>K$. 
   First, note that
$c_w=b_{1w}=1$ for any child node $w\in G_1$ of $1$ since node $1$ is its only parent. If $K=1$, we are already done. Assume $K\ge 2$ below.
We make an induction assumption that $c_w=1$ for any $w\in G_0\cup G_1\cup\ldots G_k$, $k<K$. We shall show that for any $v\in G_{k+1}$, we have $c_v=1$ as well. Indeed, 
since any path from $1$ to $v$ must go through $\pa(v)\subset G_0\cup G_1\cup\ldots G_k$, a recursion based on the induction assumption yields
\[
c_{v}=\sum_{u\in \pa(v)} b_{uv } c_u=\sum_{u\in \pa(v)} b_{uv }=1.
\]

Below,    for a vector $\mbf{v}$, we write $\max(\mbf{v})$ and $\min(\mbf{v})$ to represent its maximum and minimum component value, respectively.  
Let $\mbf{L}$ be a random vector with distribution $\mu\left(\mbf{W}\in \cdot \mid  \max\pp{\mbf{W}} >0, \eta_1>0\right)$. We make the following claim, which will be proved below:  $\mbf{L}$ follows
a multivariate generalized Pareto  distribution
(e.g., \cite{rootzen2006multivariate,rootzen2018multivariate}) that takes value in $\{\mbf{z}\in [-\infty,\infty)^{V}:  \max(\mbf{z})>0\}$,  with the following stochastic representation:
\begin{equation}\label{eq:mult Pareto HR}
\mbf{L}\EqD    E +\mbf{S}.  
\end{equation}
Here,  $E$ is a standard exponential random variable independent of $\mbf{S}$, and $\mbf{S}$ is a random vector whose distribution is given by
$
\Prt(\mbf{S}\in \cdot ) =\frac{\E\pb{ \mbf{1}_{\{\mbf{U}-\max(\mbf{U})\in  \cdot \}   } \exp(\max(\mbf{U})) }}{ \E \exp(\max(\mbf{U})) },
$
where $\mbf{U}$ has the same distribution as  $\left(0, \pp{\mbf{d}_u^\top \mbf{Z} }_{u\in \{2,\ldots,d\}}^\top\right)^\top$ under $\Prt_{\sbf{\theta}}$,  that is, a  multivariate normal distribution that is degenerately $0$ in  $1$st component,  and  with mean vector $R\sbf{\mu}_{\de(1)}$ and covariance matrix $R\Sigma_{\de(1)}R^\top$ in the $(2,\ldots,d)$-components, where   $R=\pp{\mbf{d}_u^\top  }_{u\in \{2,\ldots,d\}}$  ($u$ indexes rows). 
So $\mbf{L}$ is a H\"usler-Reiss   generalized Pareto distribution in view of \cite[Section 7.2]{kiriliouk2019peaks}. 

\begin{proof}[Proof of the representation \eqref{eq:mult Pareto HR}]
 Set $\xi_1=\log(a_1\eta_1)$. By \eqref{eq:Lambda_perp} and the assumption $\alpha=1$,  we know $\mu(\xi_1>x)= s a_1 e^{-x}$, $x\in (-\infty,\infty)$. Set 
$\mbf{U}=(0 , \pp{\mbf{d}_u^\top \mbf{Z}}_{u\in \{2,\ldots,d\}}^\top)^\top.
$
Below, for two vectors $\mbf{v}_1$ and $\mbf{v}_1$ of the same dimension, we write $\mbf{v}_1\leq \mbf{v}_2$ to mean that the inequality holds component-wise, and write $\mbf{v}_1\not\leq \mbf{v}_2$ to mean the contrary of the previous one (i.e., the inequality fails for at least one component).
 In view of \eqref{eq:HR eSCM re-exp} and \eqref{eq:c 1 vector}, one has
\begin{align*}
 &\mu\left(  \max\pp{\mbf{W}} >0 ,\ \eta_1>0\right) = \mu\left(  \max \pp{\pp{ \xi_1, \xi_1 \mbf{c}^\top +\mbf{Z}^\top D^\top }^\top} >0 ,\ \eta_1>0\right) \\=& \mu(\xi_1 > \min( -\mbf{U})  ) =s a_1\E_{\sbf{\theta}}\pb{ \exp\pp{\max(\mbf{U})} }.
\end{align*}
Let $\mbf{x}\in [-\infty,\infty)^{V}$ with $\|\mbf{x}\|_\infty>0$. 
Then
\begin{align*}
  &\mu\left(  \mbf{W}\not\leq \mbf{x},  \max\pp{\mbf{W}} >0,  \eta_1>0\right) = \mu(\xi_1 > \min( -\mbf{U}), \xi_1>\min(\mbf{x}-\mbf{U})  )\\
 = &s a_1  \E_{\sbf{\theta}}\pb{ \exp\pp{\max(\mbf{U})} \wedge \exp\pp{\max(\mbf{U}-\mbf{x})}  }. 
\end{align*}
Therefore, the joint CDF of $\mbf{L}$ is given by 
\begin{align*}
  \Prt\pp{ \mbf{L}  \le \mbf{x} }&=1-   \frac{\mu\left(  \mbf{W} \not\leq \mbf{x},  \max\pp{\mbf{W}} >0,  \eta_1>0\right)}{\mu\left(  \max\pp{\mbf{W}} >0 , \eta_1>0\right)}\\
  &= 1- \frac{\E_{\sbf{\theta}}\pb{ \exp\pp{\max(\mbf{U})} \wedge \exp\pp{\max(\mbf{U}-\mbf{x})}  }}{\E_{\sbf{\theta}}\pb{ \exp\pp{\max(\mbf{U})}  } }.
\end{align*}
The conclusion then follows from  \cite[Theorem 7 \& Proposition 9]{rootzen2018multivariate} (there seems to be a typo in \cite[Eq.(30)]{rootzen2018multivariate}, in which the maximum sign $\vee$ should be replaced by a minimum sign $\wedge$ as the last formula displayed above).
\end{proof}

\subsection{Proof of Theorem \ref{Thm:limit eSCM}}

 

The strategy is inspired by the proof of \cite[Theorem 1]{engelke2025extremes}.
To prove the homogeneity of $f_v^*$, suppose $c>0$ and fix $\theta\in[0,1]$. Let    $\mbf{x}_{\pa(v)}(t)\rightarrow \mbf{y}_{\pa(v)}$ within $[0,\infty)^{\pa(v)}$ and $\zeta(t)\rightarrow\eta$ within $[0,\infty)$  as $t\rightarrow\infty$ with $t\in (0,\infty)$. Then using the asymptotic homogeneity of $g_v$, we have
\begin{align*}
 f_v^*\pp{ c \mbf{y}_{\pa(v)}, c\eta_v, \theta  }=\lim_{t\rightarrow\infty} c (ct)^{-1}  g_v\pp{ c t \mbf{x}_{\pa(v)}(t),  ct \zeta_v(t), \theta  }  = c f_v^*\pp{  \mbf{y}_{\pa(v)}, \eta_v, \theta  }.
\end{align*}
The relation also holds when $c=0$ by the assumption $f_v^*\pp{ \mbf{0}_{\pa(v)}, \mbf{0}, \theta  }=0$ for any $\theta\in [0,1]$.

Now we proceed to prove the second claim.
By a recursion of \eqref{eq:pre limit SCM} similarly as \eqref{eq:Y=F(eta,theta)}, one may express 
\begin{equation}
    \mbf{X}=\pp{X_v}_{v\in V}=\mbf{G}_{\cl{G}}\pp{\sbf{\zeta},\sbf{\theta}}:= \pp{ G_{\cl{A}(v)}\pp{ \sbf{\zeta}_{\An(v)}, \sbf{\theta}_{\An(v)}  } }_{v\in V}
\end{equation}
for some measurable functions $G_{\cl{A}(v)}: [0,\infty)^{|\An(v)|}\times [0,1]^{|\An(v)|}\mapsto [0,\infty)$, $v\in V$.  Next, we observe that in view of the  asymptotic homogeneity property imposed on each $g_v$ in \eqref{eq:pre limit SCM} in the first assumption of the theorem,   for any fixed $\sbf{\theta}_{\An(v)}\in [0,1]^{\An(v)}$, the function $G_{\cl{A}(v)}\pp{ \cdot , \sbf{\theta}_{\An(v)}  }$ is asymptotically homogeneous as well, that is, 
\begin{equation}\label{eq:asymp homo G_A}
\lim_{t\rightarrow\infty}  t^{-1} G_{\cl{A}(v)}\pp{ t \mbf{x}(t) , \sbf{\theta}_{\An(v)} } = F_{\cl{A}(v)}^*\pp{ \mbf{x}, \sbf{\theta}_{\An(v)}}
\end{equation}
for any $\mbf{x}(t)\rightarrow \mbf{x}$ within $[0,\infty)^{\An(v)}$ as $t\rightarrow\infty$, where $F_{\cl{A}(v)}^*$ is as defined as ${F}_{\cl{A}(v)}$ in \eqref{eq:Y=F(eta,theta)} but  with $f_v$ replaced by $f_v^*$.

Take a Borel $B\subset \bb{E}_V$ that is separated from the origin (i.e., the closure of $B$ in $[0,\infty)^V$ does not intersect  the origin) such that $\mu\left( \mbf{Y} \in \partial B\right)=0$, and $\epsilon>0$. Assume without loss of generality that $\|\cdot\|=\|\cdot\|_\infty$. We have 
\begin{align*}
t \Pr\pp{t^{-1/\alpha} \mbf{X}\in B   }=   & t  \Pr\pp{t^{-1/\alpha} \mbf{X}\in B , t^{-1/\alpha} \zeta_v> \epsilon \text{ for some }v\in V }\\+ &t  \Pr\pp{t^{-1/\alpha} \mbf{X}\in B , t^{-1/\alpha} \|\sbf{\zeta}\|_\infty\le \epsilon}
\end{align*}
Note that
\begin{align*}
&\sum_{v\in V} t  \Pr\pp{t^{-1/\alpha} \mbf{X}\in B , t^{-1/\alpha} \zeta_v> \epsilon} - \sum_{u,v\in V, u\neq v} t  \Pr\pp{t^{-1/\alpha} \zeta_u> \epsilon , t^{-1/\alpha} \zeta_v> \epsilon}  \\\le  & t  \Pr\pp{t^{-1/\alpha} \mbf{X}\in B , t^{-1/\alpha} \zeta_v> \epsilon \text{ for some }v\in V }\\\le  &\sum_{v\in V} t  \Pr\pp{t^{-1/\alpha} \mbf{X}\in B , t^{-1/\alpha} \zeta_v> \epsilon},
\end{align*}
as well as the limit relations $\lim_{t\rightarrow\infty}t  \Pr\pp{t^{-1/\alpha} \zeta_u> \epsilon , t^{-1/\alpha} \zeta_v> \epsilon} =0$ for $u\neq v$ due to extremal independence, $\lim_{t\rightarrow\infty}t\Prt\pp{ t^{-1/\alpha}\zeta_v>\epsilon }=s \epsilon^{-\alpha}=\mu\pp{\eta_v>\epsilon}$, and $\lim_{\epsilon\downarrow 0} \sum_{v\in V}\mu(\mbf{Y}\in B, \eta_v>\epsilon)=\mu(\mbf{Y}\in B)$.
Combining these relations, in order to show $\lim_{t\rightarrow\infty}t \Pr\pp{t^{-1/\alpha} \mbf{X}\in B   } = \mu\pp{\mbf{Y}\in B}$, it suffices to show  for each $v\in V$ that
\begin{align}\label{eq:limit eSCM goal 1}
\lim_{t\rightarrow\infty } \Pr\pp{t^{-1/\alpha} \mbf{X}\in B \mid t^{-1/\alpha} \zeta_v> \epsilon}=\mu\pp{\mbf{Y}\in B\mid \eta_v>\epsilon},
\end{align}
and  
\begin{align}\label{eq:limit eSCM goal 2}
\lim_{\epsilon\downarrow 0}\limsup_{t\rightarrow\infty } t  \Pr\pp{t^{-1/\alpha} \mbf{X}\in B , t^{-1/\alpha} \|\sbf{\zeta}\|_\infty\le \epsilon}=0.
\end{align}

We first prove \eqref{eq:limit eSCM goal 1}, for which it suffices to show the weak convergence of the conditional law $\cl{L}\pp{t^{-1/\alpha} \mbf{X} \mid  t^{-1/\alpha}\zeta_v>\epsilon}$ toward $\cl{L}\pp{\mbf{Y} \mid \eta_v>\epsilon}$ on $[0,\infty)^V$ as $t\rightarrow\infty$. Suppose $H:[0,\infty)^V\mapsto \bb{R}$ is   bounded and continuous. Let $\mbf{F}^*_{\cl{G}}$ be defined as $\mbf{F}_{\cl{G}}$ in \eqref{eq:Y=F(eta,theta)} but with $F_{\cl{A}(v)}$ replaced by ${F}_{\cl{A}(v)}^*$ in \eqref{eq:asymp homo G_A}.   To prove the aforementioned weak convergence, due to independence and Fubini, it suffices to show that
\begin{equation}\label{eq:eSCM limit goal 1 red}
\lim_{t\rightarrow\infty}  \E_{\mid \zeta_v}^{(t)}  \E_{\sbf{\theta}}  H\pp{ t^{-1/\alpha} \mbf{G}_{\cl{G}}\pp{t^{1/\alpha} \cdot  t^{-1/\alpha} \sbf{\zeta}, \sbf{\theta}  } }       =   \E_{\mid \eta_v} \E_{\sbf{\theta}}  H\pp{  \mbf{F}^*_{\cl{G}}\pp{\sbf{\eta}, \sbf{\theta}  } },   
\end{equation} 
where we lightly abuse the notation to use $\E_{\sbf{\theta}}$ to denote expectation with respect to the uniform random vector $\sbf{\theta}$ in   both contexts of SCM $\mbf{X}$ and eSCM $\mbf{Y}$,   to use $\E_{\mid \sbf{\zeta}}^{(t)}$ to denote the expectation with respect to the conditional law $\cl{L}\pp{  t^{-1/\alpha}\sbf{\zeta}  \mid  t^{-1/\alpha}\zeta_v>\epsilon }$,  and to use $\E_{\mid\eta_v}$ to denote the expectation with respect to the conditional law $\cl{L}(\sbf{\eta}\mid  \eta_v>\epsilon)$. Recall $\eta_v>0$ implies $\eta_u=0$ for $u\neq v$.
Set
\begin{align*}
  \wt{H}_t:  [0,\infty)^V\mapsto \mathbb{R}, \quad  \wt{H}_t(\mbf{x})=  \E_{\sbf{\theta}} \pb{ H\pp{ t^{-1/\alpha} \mbf{G}_{\cl{G}}\pp{t^{1/\alpha} \mbf{x} , \sbf{\theta}  } }  }
\end{align*}
and
\begin{align*}
  \wt{H}:  [0,\infty)^V\mapsto \mathbb{R}, \quad  \wt{H}(\mbf{x})=  \E_{\sbf{\theta}} \pb{  H\pp{  \mbf{F}^*_{\cl{G}}\pp{\mbf{x}, \sbf{\theta}  } }}.
\end{align*}
Since $H$ is bounded, by uniform integrability, to show \eqref{eq:eSCM limit goal 1 red}, it suffices to show
\begin{equation}\label{eq:H tilde t}
\wt{H}_t\pp{ \mbf{Z}_t }\ConvD  \wt{H}\pp{ \mbf{Z} } 
\end{equation}
as $t\rightarrow\infty$, where $\mbf{Z}_t$ has the law $\cl{L}\pp{  t^{-1/\alpha}\sbf{\zeta}  \mid  t^{-1/\alpha}\zeta_v>\epsilon }$ and $\mbf{Z}$ has the law $\cl{L}(\sbf{\eta}\mid  \eta_v>\epsilon)$.
 Note that  due to boundedness and continuity of $H$,   the aforementioned asymptotic homogeneity of each component of $\mbf{G}_{\cl{G}}(\cdot ,\sbf{\theta})$ for each $\sbf{\theta}\in [0,1]^V$ fixed, and the dominated convergence theorem, we have for any $\mbf{x}(t)\rightarrow \mbf{x}$ within $[0,\infty)^V$ that $\wt{H}_t(\mbf{x}(t))\rightarrow \wt{H}(\mbf{x})$ as $t\rightarrow\infty$. So \eqref{eq:H tilde t} follows from the extended continuous mapping theorem (e.g., \cite[Theorem 5.27]{kallenberg:2021:foundations}). Therefore, the relation \eqref{eq:limit eSCM goal 1} is concluded.

Next, we prove \eqref{eq:limit eSCM goal 2}. We claim that  we have  
\begin{equation}\label{eq:X_v bound theta}
X_v=G_{\cl{A}(v)}\pp{ \sbf{\zeta}_{\An(v)} , \sbf{\theta}_{\An(v)}  }\le C_{\An(v)}\pp{ \sbf{\theta}_{\An(v)}}\|\sbf{\zeta}_{\An(v)}\|_\infty \quad a.s.
\end{equation}
for some measurable $ C_{\An(v)}:[0,1]^{\An(v)}\mapsto [0,\infty)$ with $\E_{\sbf{\theta}} \pb{C_{\An(v)}\pp{ \sbf{\theta}_{\An(v)}}^\alpha}<\infty$. Indeed,  $C_{\An(v)}(\sbf{\theta}_{\An(v)})$ can be taken as the multiplication of independent random variables $1\vee C_u(\theta_u)$  with distinct  $u\in \An(v)$, and each  $\E_{\sbf{\theta}} \pb{ C_u(\theta_u)^\alpha}<\infty$ by the second assumption.   Since $B$ in \eqref{eq:limit eSCM goal 2} is separated from the origin,  we have $\delta:=\inf\{\|\mbf{x}\|_\infty: \mbf{x}\in B \}>0$. Therefore, by \eqref{eq:X_v bound theta} and the fact that $\|\sbf{\zeta}_{\An(v)}\|_\infty\le \|\sbf{\zeta}\|_\infty$, we have
\begin{align*}
   & t  \Pr\pp{t^{-1/\alpha} \mbf{X}\in B ,\ t^{-1/\alpha} \|\sbf{\zeta}\|_\infty\le \epsilon} \le     t  \Pr\pp{ t^{-1/\alpha} \|\mbf{X}\|_\infty \ge \delta ,\ t^{-1/\alpha} \|\sbf{\zeta}\|_\infty\le \epsilon}\notag \\
  \le &  \sum_{v\in V} t\Prt\pp{ t^{-1/\alpha} C_{\An(v)}\pp{ \sbf{\theta}_{\An(v)}}\|\sbf{\zeta}\|_\infty\ge \delta,\   t^{-1/\alpha} \|\sbf{\zeta}\|_\infty\le \epsilon }.
\end{align*}
By \eqref{eq:zeta asymp indep} and \cite[Proposition 2.1.12]{kulik2020heavy},  recalling $d=|V|$, we have for any $x>0$ that
\begin{equation}\label{eq:zeta norm RV}
\lim_{t\rightarrow\infty}t\Prt\pp{   t^{-1/\alpha} \|\sbf{\zeta}\|_\infty \ge x  }=\lim_{t\rightarrow\infty}t\Prt\pp{   t^{-1/\alpha} \|\sbf{\zeta}\|_\infty > x  }= ds x^{-\alpha}.
\end{equation}
Then,
\begin{align*}
 &\limsup_{t\rightarrow\infty}   t\Prt\pp{ t^{-1/\alpha} C_{\An(v)}\pp{ \sbf{\theta}_{\An(v)}}\|\sbf{\zeta}\|_\infty\ge \delta,\   t^{-1/\alpha} \|\sbf{\zeta}\|_\infty\le \epsilon }\\\le
& \E \limsup_{t\rightarrow\infty}   t\Prt\pp{ t^{-1/\alpha} C_{\An(v)}\pp{ \sbf{\theta}_{\An(v)}}\|\sbf{\zeta}\|_\infty\ge \delta,\   t^{-1/\alpha} \|\sbf{\zeta}\|_\infty\le \epsilon 
\mid \sbf{\theta} }\\ \le  &  ds \E \pb{ \pp{ \delta^{-\alpha} C_{\An(v)}\pp{ \sbf{\theta}_{\An(v)}}^\alpha-\epsilon^{-\alpha}}_+ }.
\end{align*}
Here,  the first inequality displayed above follows from a reversed Fatou's Lemma since
$    t\Prt\pp{ t^{-1/\alpha} C_{\An(v)}\pp{ \sbf{\theta}_{\An(v)}}\|\sbf{\zeta}\|_\infty\ge \delta  \mid \sbf{\theta} }\le c_0 C_{\An(v)}\pp{\sbf{\theta}_{\An(v)}}^\alpha \delta^{-\alpha}
$
almost surely for some constant $c_0>0$ by \eqref{eq:zeta norm RV}, and   $\E \pb{C_{\An(v)}\pp{\sbf{\theta}_{\An(v)}}^\alpha}<\infty$. The second inequality displayed above follows from \eqref{eq:zeta norm RV} again.
Now,  the final bound displayed above   tends to $0$ if $\epsilon\downarrow 0$ by the dominated convergence theorem. So \eqref{eq:limit eSCM goal 2}  follows combining the relations above.

At last, we note that the third assumption in the theorem ensures that the marginal law of $\mbf{Y}$ is nontrivial, that is, $\mu(Y_v>y_v)=s_v y_v^{-\alpha}$ for some $s_v\in (0,\infty)$.  In fact, since we have already proved the relation \eqref{eq:X limit Y},  we have established joint regular variation of $\mbf{X}$, which by \cite[Proposition 2.1.12]{kulik2020heavy} implies the marginal regular variation of each $X_v$, $v\in V$, given that  the law of $X_v$ is not a constant zero.  

{  
\subsection{Proof of Corollary \ref{Cor:PPP eSCM}}
Note that one can write $\mbf{Y}_i=\mbf{F}_{\cl{G}}^*(\boldsymbol{\eta}_i,\sbf{\theta}_i)$, $i\in \bb{Z}_+$, where the measurable function $\mbf{F}_{\cl{G}}^*$ is obtained by a recursion of the structural equations  $y_{v}=f_v^*(\mbf{y}_{\pa(v)},\eta_v,\theta_v)$, $v\in V$, in the same way  as   $\mbf{F}_{\cl{G}}$  in \eqref{eq:Y=F(eta,theta)}. Recall the law $\cl{L}(\mbf{Y})$ is the push-forward measure of $ (\Lambda^{\perp}\otimes \Prt_{\sbf{\theta}})(\cdot)$ under the map $\mbf{F}_{\cl{G}}^*$ restricted to $\mathbb{E}_V$.
Then the conclusion that $\sum_{i=1}^\infty \delta_{\mathbf{Y}_i}$ is a Poisson point process with mean measure $\cl{L}(\mbf{Y})$ follows from first   applying  the marking theorem (\cite[Theorem 5.6]{last2017lectures}) to obtain a Poisson point process $\sum_{i=1}^\infty \delta_{(\sbf{\eta}_i,\mbf{F}_{\cl{G}}^*(\boldsymbol{\eta}_i,\sbf{\theta}_i))} $, and     then applying the mapping theorem (\cite[Theorem 5.1]{last2017lectures})  for the projected process $\sum_{i=1}^\infty \delta_{\mbf{F}_{\cl{G}}^*(\boldsymbol{\eta}_i,\sbf{\theta}_i)}$. The weak convergence $\sum_{i=1}^n \delta_{\mbf{X}_i/n^{1/\alpha}}\Rightarrow \sum_{i=1}^\infty \delta_{\mathbf{Y}_i}$   directly follows from the conclusion of Theorem \ref{Thm:limit eSCM} and \cite[Theorem 6.2]{resnick2007heavy}.
 }
 
\subsection{Proof of Theorem \ref{Thm:e causal Markov}}\label{sec:pf Thm:e causal Markov}
 We use an alternative characterization of extremal conditional independence for the proof, which follows from \cite[Theorem 4.1 and Remark 4.2]{engelke2025graphical}. Below for a nonempty subset $I\subset V$ and exponent measure $\Lambda$, we use $\Lambda^0_{I}(\cdot)$ to denote the restriction of $\Lambda( \{\mbf{y}\in \bb{E}_V:\ \mbf{y}_I \in \cdot\, , \, \mbf{y}_{V\setminus I}=\mbf{0}_{V\setminus I} \})$   to $\bb{E}_I$. 
\begin{proposition}\label{Pro:alt e cond indep}
Following the notation in Definition \ref{Def:e cond indep}, let $A$, $B$ and $C$ be disjoint nonempty subsets of $V$ such that $V=A\cup B\cup C$.
The extremal conditional independence relation $A\perp B\mid C\, [\Lambda]$ is equivalent to the following two statements:  i) 
 The probabilistic conditional independence $\mbf{Y}_A^{(v)}\perp\mbf{Y}_{B}^{(v)}\mid \mbf{Y}_C^{(v)}$ holds for all $v\in C$;
 ii) $A\perp B \left[\Lambda^0_{A\cup B}\right]$ (understood as always true if $\Lambda^0_{A\cup B}$ is a zero measure). 
\end{proposition}
We note that although the proposition only concerns the case where all index sets $A$, $B$ and $C$ are nonempty, but when this is not the case, the understanding described in Definition \ref{Def:e cond indep} still applies.

\begin{proof}[Proof of Theorem \ref{Thm:e causal Markov}]
As mentioned before the comments of Theorem \ref{Thm:e causal Markov},  it suffices to prove the local  directed Markov property  \eqref{eq:local markov}. We  fix   a node $v\in V$ from now on. In view of Remark \ref{Rem:e cond indep}, we can assume  $\{v\} \cup  \pp{\nd(v)\setminus  \pa(v)}  \cup \pa(v)=V$, or equivalently,   $\de(v)=\emptyset$.
Under this assumption,  \eqref{eq:local markov} becomes
\begin{equation*} 
\{v\}\perp V\setminus \pp{ \{v\}\cup \pa(v)     } \mid \pa(v) [\Lambda],\quad v\in V,
\end{equation*}
which is what we aim to show.

\noindent $\bullet$ The case   $V=\{v\}$ is trivial.

\noindent $\bullet$ The case $V\neq\{v\}$ and $\pa(v)=\emptyset$.

In this case, one needs to show $\{v\}\perp V\setminus  \pc{v}    [\Lambda]$. In view of Remark \ref{Rem:e cond indep}, it suffices to show $\mu(Y_v>0,  Y_u>0 )=0$ for any $u\in V\setminus\pc{v} $. Fix such a pair $(u,v)$ in the following. 
Note that since $v$ is a root node, in view of \eqref{eq:general eSCM}, one has only $Y_v=a_v\eta_v$. So $Y_v>0$ implies $\eta_v>0$, and hence $\eta_w=0$ for all $w\neq v$ due to the single-activation nature of $\sbf{\eta}$. Since $\de(v)=\emptyset$ by assumption, we have $v\notin\An(u)$, and hence $Y_v>0$ implies  $Y_u=F_{\cl{A}(u)}\pp{ \mbf{0}_{\mrm{An}(u)}, \sbf{\theta}_{\mrm{An}(u)} }=0$ (see \eqref{eq:Y=F(eta,theta)}).  Therefore $\mu(Y_v>0,  Y_u>0 )=0$.  

\noindent $\bullet$ The case  $V= \{v\}\cup \pa(v) $ and  $\pa(v)\neq \emptyset$  is trivial.

\noindent $\bullet$ The case $V\neq  \{v\}\cup \pa(v) $ and  $\pa(v)\neq \emptyset$.

In this case, we apply Proposition \ref{Pro:alt e cond indep} with $A=\{v\}$, $B=V\setminus \pp{ \{v\}\cup \pa(v)   }$ and $C=\pa(v)$.

\noindent \emph{Verification of condition i) in Proposition \ref{Pro:alt e cond indep}.} 

For this purpose,  fix $u\in \pa(v)$.    Assume now without loss of generality that the underlying measure space $(\Omega,\cl{F},\mu)$ is the canonical space:  $\Omega=\bb{E}_V\times [0,1]^V$, $\cl{F}$ is the Borel-$\sigma$-field, and $\mu=\Lambda^{\perp}\otimes \mathrm{Leb}^d$, where $\mathrm{Leb}$ denotes the Lebesgue measure on $[0,1]$.  Define $\Omega_u=\pc{Y_u \ge  1}\subset \Omega$, and introduce a probability measure $\mu_u(\cdot)$ on $\Omega_u$ as the restriction of  $\mu(\cdot \cap \Omega_u)/\mu(\Omega_u)$ to $\Omega_u$.  
Now we define $\mbf{Y}^{(u)}=\mbf{F}_{\cl{G}}(\sbf{\eta},\sbf{\theta})$, with $\mbf{F}_{\cl{G}}$ as in  \eqref{eq:Y=F(eta,theta)}, on the probability space $(\Omega_u,\cl{F}_u,\mu_u )$, where $\cl{F}_u$ is the restriction of $\cl{F}$ to $\Omega_u$. 
Then the probabilistic law   of $\mbf{Y}^{(u)}=\pp{Y_v^{(u)}}_{v\in V}$ aligns with the random vector described in  Definition \ref{Def:e cond indep}. 

Next, recall one may express $Y_u$ be its ancestors as $Y_u=F_{\cl{A}(u)}\pp{ \sbf{\eta}_{\mrm{An}(u)}, \sbf{\theta}_{\mrm{An}(u)} }$, $ F_{\cl{A}(u)}$ is as in \eqref{eq:Y=F(eta,theta)}.  Therefore, $\Omega_u$ can be expressed as
\begin{equation}\label{eq:Omega_u}
\Omega_u=\pc{ \pp{\sbf{\eta},  \sbf{\theta}}\in  \bb{E}_V\times [0,1]^V: \, \pp{\sbf{\eta}_{\mrm{An}(u)}, \sbf{\theta}_{\mrm{An}(u)}} \in  F_{\cl{A}(u)}^{-1}[1,\infty)  }.
\end{equation}
Furthermore,    $Y_u\ge 1$
  implies $\eta_w>0$ for precisely one $w\in \An(u)$. In particular, we must have $\eta_v=0$, since as a child node of $u$, the node $v\notin \An(u)$. Hence on $\Omega_u$, 
\begin{equation}\label{eq:Y_v^{(u)}}
Y_v^{(u)}=f_v\pp{ \mbf{Y}^{(u)}_{\pa(v)}, 0, \theta_v}=h_v\pp{\mbf{Y}^{(u)}_{\pa(v)},   \theta_v}.
\end{equation}   
In view of the fact   $v\notin \An(u)$, \eqref{eq:Omega_u} and the definition of $\mu$,  we can also see that  
under $(\Omega_u,\mu_u )$, the random variable $\theta_v$   is independent of the random vector $(\sbf{\eta}_{V\setminus \{v\}}, \sbf{\theta}_{V\setminus \{v\}})$. Combining this with \eqref{eq:Y_v^{(u)}},  we conclude that under $(\Omega_u,\mu_u )$, conditioning on $\mbf{Y}^{(u)}_{\pa(v)}$,   we have the independence   between $Y_v$ and $\mbf{Y}^{(u)}_{ V\setminus \pp{\{v\}\cup \pa(v) } }$,   the latter being a measurable function of $(\sbf{\eta}_{V\setminus \{v\}}, \sbf{\theta}_{V\setminus \{v\}})$.

\noindent \emph{Verification of condition ii) in Proposition \ref{Pro:alt e cond indep}.}

It suffices to show that 
$\mu\pp{\mbf{Y}_{\pa(v)}=0 ,\,  Y_v>0,\, Y_u>0    }=0$
for any   $u\in V\setminus \pp{ \{v\}\cup \pa(v)}$.  Indeed, under  $\mbf{Y}_{\pa(v)}=0$, the stipulation  $Y_v=f_v\pp{ \mbf{0}_{\pa(v)}, \eta_v, \theta_v}=a_v \eta_v>0$ implies that $\eta_v>0$, and hence $\eta_w=0$ for all $w\neq v$. Since also  $\de(v)=\emptyset$ by assumption, and $u\neq v$, we know $v\notin \An(u)$, which further implies $Y_u=F_{\cl{A}(u)}\pp{ \sbf{0}_{\mrm{An}(u)}, \sbf{\theta}_{\mrm{An}(u)} }=0$.  The conclusion then follows.

\end{proof}

\subsection{Proof of Theorem \ref{Thm:Lambda expr by eSCM}}

Assume for simplicity that the scaling constant $s=1$ in \eqref{eq:Lambda_perp}.
We prove the theorem by induction on the node size. To start the induction, note that when we only have a single node $1$ in \eqref{eq:general eSCM},  one can simply set  $Y_1=f_1(\eta_1,\theta_1)=s_1^{1/\alpha}\eta_1$ to achieve the desirable exponent measure. 

Now suppose that the conclusion holds for node size $d\in \bb{Z}_+$, and we want to prove it when the node size becomes $d+1$. 
  We  use $\cl{G}_+=(V_+,E_+)$ to denote the DAG with node set $V_+=\{1,\ldots,d+1\}$ and edge set $E_+$. Suppose $\Lambda_{V_+}$ is an  exponent measure on $\bb{E}_{V_+}$ obeying the extremal causal Markov property with respect to $\cl{G}_+$.  Since $\cl{G}_+$ is a DAG, there exists at least one leaf (i.e., childless) node.
Without loss of generality, suppose $d+1$ is such a leaf node.  Set $V=V_+\setminus\{d+1\}=\{1,\ldots,d\}$, and let $\cl{G}$ be the sub-DAG of $\cl{G}_+$ with node set $V$. 

Let $\Lambda_V(\cdot)$ denote the restriction of
$\Lambda_{V_+}\pp{\pc{\mathbf{y}_V\in\cdot\,,\ \mathbf{y}_V\neq \mathbf{0}_V}}$ to
$\bb{E}_V=[0,\infty)^V\setminus\pc{\mathbf{0}_V}$. It is an exponent measure on
$\bb{E}_V$: the homogeneity \eqref{eq:Lambda homo} is inherited from
$\Lambda_{V_+}$, and $\Lambda_V(y_v>1)=\Lambda_{V_+}(y_v>1)=s_v\in(0,\infty)$ for
$v\in V$, since $\pc{y_v>1}\subset\pc{\mathbf{y}_V\neq\mathbf{0}_V}$.
Moreover, since $d+1$ is a leaf node of $\cl{G}_+$, we have
$\pa_{\cl{G}}(v)=\pa_{\cl{G}_+}(v)$ and $\nd_{\cl{G}}(v)\subset\nd_{\cl{G}_+}(v)$
for every $v\in V$, where the subscripts on $\pa$ and $\nd$ indicate the graph
with respect to which these sets are taken. Hence the local Markov property
\eqref{eq:local markov} of $\Lambda_{V_+}$, combined with the decomposition
property of the semi-graphoid axiom (see, e.g., \cite[Section 3.3]{engelke2025graphical}) and the marginalization equivalence
mentioned in Remark~\ref{Rem:e cond indep}, yields
\[
  \pc{v}\perp \pp{\nd_{\cl{G}}(v)\setminus\pa_{\cl{G}}(v)}
  \mid \pa_{\cl{G}}(v)\,[\Lambda_V], \qquad v\in V;
\]
that is, $\Lambda_V$ satisfies the extremal causal Markov property with respect
to $\cl{G}$.

By the induction assumption,   there exist functions $f_v$, $v\in V$, as described in Definition \ref{def:eSCM}, such that with the extreme variables $\mbf{Y}_V=\pp{Y_v}_{v\in V}$ given by the recursive equations \eqref{eq:general eSCM}, one has 
\begin{equation}\label{eq:induction assump} 
\cl{L}(\mbf{Y}_V)=\Lambda_V,
\end{equation}
where   $\Lambda_V(\cdot)$ is an exponent measure on $\bb{E}_V= [0,\infty)^V\setminus\{\mbf{0}_V\}$ obtained by the restriction of $\Lambda_{V_+}( \{\mbf{y}_V \in \cdot\ ,   \mbf{y}_V\neq \mbf{0}_V\})$  to $\bb{E}_V$, and  $\cl{L}\pp{\mbf{Y}_V}$ denotes  the restriction of $\mu(\mbf{Y}_V\in \cdot )$  to $\bb{E}_V$.

Now we enlarge the measure space $\Omega$ by adjoining a new pair of variables $(\eta_{d+1},\theta_{d+1})$. In particular, we set $\Omega_+= \bb{E}_{V_+}\times [0,1]^{V_+}=\pc{ \pp{ (\eta_v)_{v\in V_+}, (\theta_v)_{v\in V_+}} }$, and consider the measure $\mu_+=\Lambda^{\perp}_+ \otimes\mathrm{Leb}^{d+1}$, where $\Lambda^{\perp}_+$  is a measure on $\bb{E}_{V_+}$  defined in the same way as  $\Lambda^{\perp}$ in \eqref{eq:Lambda_perp} but with dimensionality $d+1$.
The variables $\mbf{Y}_V=\pp{Y_v}_{v\in V}$ constructed by the recursive equations \eqref{eq:general eSCM} continue to make sense in the enlarged measurement space, once we additionally require $\mbf{Y}_V$ not to depend on $\theta_{d+1}$ on $\{\eta_{d+1}=0\}$ and set $\mbf{Y}_V=\mbf{0}_V$ on $\{\eta_1=\ldots=\eta_d=0,\eta_{d+1}>0\}$ (note that the relation $\eta_1=\ldots=\eta_d=0$  is not admissible in the original $\Omega$ space).

With  the construction above, we claim that the following marginalization relation holds:   for any Borel $U\subset \bb{E}_V$, one has
\begin{equation}\label{eq:mu_+ marg}
    \mu_+\pp{\mbf{Y}_{V}\in U}=  \mu\pp{\mbf{Y}_{V}\in U}, 
\end{equation}
 where we slightly abuse the notation to use $\mbf{Y}_{V}$ to denote both  the $V$-marginal variable of $\mbf{Y}_{V_+}$ on the left-hand side, as well as the full variable $\mbf{Y}_{V}$ taking value in $\bb{E}_V$ on the right-hand side. 
 To see \eqref{eq:mu_+ marg}, recall that one can write $\mbf{Y}_V=\mbf{F}_{\cl{G}}\pp{\boldsymbol{\eta}_V,\sbf{\theta}_V }$ for some $\mbf{F}_{\cl{G}}:\Omega=\bb{E}_V\times [0,1]^V\mapsto [0,\infty)^V$   as in \eqref{eq:Y=F(eta,theta)}. Here the node $d+1$ is not involved in expressing $\mbf{Y}_V$ since it is a leaf node. Observe also that     $\mbf{Y}_V\neq \mbf{0}_V $ implies  $\eta_v>0$ for some $v\in V$ and thus $\eta_{d+1}=0$. Hence with $U \subset \bb{E}_V$ (thus $\mbf{0}_V\notin U$), one has 
 \begin{align*}
     \mu_+\pp{\mbf{Y}_V\in  U}=\mu_+\pp{ \pp{\mbf{F}_{\cl{G}}^{-1}U} \times \{0\}^{\{d+1\}} \times [0,1]^{\{d+1\}} }.
 \end{align*}
 We claim that the last expression is equal to $\mu\pp{\mbf{F}_{\cl{G}}^{-1} U}$. Indeed, since $\mbf{F}_{\cl{G}}\pp{ \mbf{0}_V ,\sbf{\theta}_V }=\mbf{0}_V$ for any $\sbf{\theta}_V\in [0,1]^V$,  we have
 $
 \mbf{F}_{\cl{G}}^{-1} U \subset  \bb{E}_V \times [0,1]^V.
 $
 So by a  measure-determining argument, it suffices to show  
 \[
 \mu_+\pp{ \pp{K\times L} \times \pp{\{0\}^{\{d+1\}} \times [0,1]^{\{d+1\}}} }=\mu\pp{K\times L},
 \]
 where $K \subset \bb{E}_{V}$ and $L\subset [0,1]^V$ are Borel subsets. 
To do so, observe that by the definitions of $\mu$ and $\mu_+$, we have 
\begin{align*}
    &\mu_+\pp{ \pp{K\times L} \times \pp{\{0\}^{\{d+1\}}\times [0,1]^{\{d+1\}}}} \\= & \Lambda_+^{\perp} \pp{  \sbf{\eta}_V\in K,\, \eta_{d+1}=0  } \times \mathrm{Leb}^{d} \pp{L}\times \mathrm{Leb}([0,1]^{\pc{d+1}})\\
    = & \Lambda^{\perp}(K) \times \mathrm{Leb}^{d} \pp{L} =\mu(K\times L)
\end{align*}
So the proof of \eqref{eq:mu_+ marg} is finished.

Next, to complete the induction argument, we need to construct a measurable function $f_{d+1}: [0,\infty)^{|\pa(d+1)|}\times [0,\infty) \times [0,1]\mapsto [0,\infty)$ in the form of \eqref{eq:general eSCM}, such that with $Y_{d+1}=f_{d+1}\left(\mbf{Y}_{\pa(d+1)},\eta_{d+1},\theta_{d+1}\right)$, we have $\cl{L}(\mbf{Y}_{V_+})=\Lambda$ with $\mbf{Y}_{V_+}:=\pp{Y_v}_{v\in V_+}$.

First,  recall by the extremal causal Markov property,  we have
\begin{equation}\label{eq:d+1 cond indep}
\{d+1\}\perp V\setminus \pa(d+1)  \mid \pa(d+1) [\Lambda_{V_+}].
\end{equation}
We divide the construction of $f_{d+1}$ into several cases.

\noindent $\bullet$ The case  $\pa(d+1)=\emptyset$.

In this case,  we simply let $$Y_{d+1}=f_{d+1}(\eta_{d+1},\theta_{d+1}):=s_{d+1}^{1/\alpha} \eta_{d+1},$$ where $s_{d+1}=\Lambda_{V_+}(y_{d+1}\ge 1)\in (0,\infty)$. 
Then one has for $(x_1,\ldots,x_{d+1})\in \bb{E}_{V_+}$ that
\begin{align*}
&\mu_+(Y_1\ge x_1,\ldots, Y_{d+1}\ge x_{d+1} )   \\ 
=&\begin{cases}
    0     & \text{ if } (x_1,\ldots,x_{d})\neq \mbf{0}_V  \text{ and } x_{d+1}>0,\\
        s_{d+1} x_{d+1}^{-\alpha}     & \text{ if } (x_1,\ldots,x_{d})=\mbf{0}_V  \text{ and } x_{d+1}>0 , \\
     \Lambda_{V}(y_1\ge x_1,\ldots,y_{d}\ge x_d)            & \text{ if }(x_1,\ldots,x_{d})\neq \mbf{0}_V \text{ and } x_{d+1}=0.
\end{cases}
\end{align*}
Here, the first case holds since if $Y_v>0$ for some $v\in V$, then $\eta_w>0$ for some $w\in \An(v)\subset V=\{1,\ldots,d\}$ in view of \eqref{eq:Y=F(eta,theta)}, which implies $\eta_{d+1}=0$ since $d+1\notin\An(v)$ as a leaf node.  The second case holds by the definition of $s_{d+1}$ and the homogeneity  property: $ \Lambda_{V_+}(y_{d+1}>x_{d+1})=x_{d+1}^{-\alpha} \Lambda_{V_+}(y_{d+1}\ge 1)$. The third case  holds due to \eqref{eq:induction assump} and \eqref{eq:mu_+ marg}.

On the other hand, recall in the case $\pa(d+1)=\emptyset$, the relation \eqref{eq:d+1 cond indep} means extremal independence, i.e., $ \Lambda_{V_+}( \mbf{y}_V \neq \mbf{0}_V, y_{d+1}>0)=0$. Based on this and again the homogeneity property of $\Lambda_{V_+}$, one can derive the same expression for $ \Lambda_{V_+}(y_1\ge x_1,\ldots, y_{d+1}\ge x_{d+1} )$  as the one
displayed above. The conclusion $\cl{L}(\mbf{Y}_{V_+})=\Lambda_{V_+}$ then follows from a usual measure-determining argument (e.g., one based on Dynkin's $\pi$-$\lambda$ Theorem and $\sigma$-finiteness).   

\medskip
\noindent $\bullet$ The case $\pa(d+1)\neq \emptyset$ and $\pa(d+1)\neq V$.

Recall $\|\cdot\|_\infty$ is the $\ell^\infty$ norm.  We shall construct the function 
$f_{d+1}$ as 
\begin{align}\label{eq:f d+1 def gen}
   &f_{d+1}(\mbf{Y}_{\pa(d+1)}, \eta_{d+1},\theta_{d+1})=r_{d+1}^{1/\alpha} \eta_{d+1} + \notag\\& \mbf{1}_{\{\mbf{Y}_{\pa(d+1)}\neq  \mbf{0}_{\pa(d+1)}\}}    \|\mbf{Y}_{\pa(d+1)}\|_\infty   g\left( 
 \frac{\mbf{Y}_{\pa(d+1)}}{\|\mbf{Y}_{\pa(d+1)}\|_\infty}, \theta_{d+1}  \right)
\end{align}
for a suitable measurable mapping $g: \bb{S}_{\pa(d+1)}\times [0,1] \mapsto [0,\infty)$ that will be described below,  where 
$$\bb{S}_{\pa(d+1)}:=\left\{\mbf{y}_{\pa(d+1)}\in [0,\infty)^{\pa(d+1)}: \ \|\mbf{y}_{\pa(d+1)}\|_\infty=1 \right\},
$$
and   $$r_{d+1}:= \Lambda_{V_+}(y_{d+1}>1, \mbf{y}_{\pa(d+1)}=\mbf{0}_{\pa(d+1)})= \Lambda_{V_+}(y_{d+1}>1, \mbf{y}_{V}=\mbf{0}_{V}).$$
Here, the second equality holds due to the   Markov property \eqref{eq:d+1 cond indep} and   case ii) of Proposition \ref{Pro:alt e cond indep}. 
Note that the proper structural function  extracted from \eqref{eq:f d+1 def gen}
$$
h_{d+1}(\mbf{y}_{\pa(d+1)},\theta_{d+1}):=\mbf{1}_{\{  \mbf{y}_{\pa(d+1)}\neq  \mbf{0}_{\pa(d+1)}\}}    \|\mbf{y}_{\pa(d+1)}\|_\infty   g\pp{ \frac{\mbf{y}_{\pa(d+1)}}{\|\mbf{y}_{\pa(d+1)}\|_\infty}, \theta_{d+1}}
$$ 
satisfies the homogeneity requirement: $h_{d+1}(c\mbf{y}_{\pa(d+1)})=ch_{d+1}(\mbf{y}_{\pa(d+1)})$, for any constant $c\ge 0$. 
Here, the fraction  $\frac{\mbf{y}_{\pa(d+1)}}{\|\mbf{y}_{\pa(d+1)}\|_\infty}$ inside $g$ can be understood as an arbitrary fixed point on $\bb{S}_{\pa(d+1)}$ when $\|\mbf{y}_{\pa(d+1)}\|_\infty=0$.
This in turn results in the homogeneity of $f_{d+1}$ in \eqref{eq:f d+1 def gen},
which combined with the induction assumption also ensures the anticipated homogeneity property for $\mu_+$, that is, 
\begin{equation}\label{eq:mu_+ homo}
\mu_+\pp{ \mbf{Y}_{V_+}\in c  B  }=c^{-\alpha} \mu_+\pp{ \mbf{Y}_{V_+}\in   B  } 
\end{equation}
for any Borel $B\subset \bb{E}_{V_+}$ and $c>0$; see the Proof of Proposition \ref{Pro:basic} in Section \ref{sec:pf Pro:basic}.

Now we describe the construction of $g$.  Below, we use the conditioning notation even for infinite measures whenever appropriate, e.g., we use $ \Lambda_{V_+}( \ \cdot  \mid R ) $ to denote $  \Lambda_{V_+}( \, \cdot  \cap R )/  \Lambda_{V_+}(R)$ for any Borel $R\subset \bb{E}_{V_+}$ with $ \Lambda_{V_+}(R)\in (0,\infty)$.
Let $\sigma$ be the probability measure on $\bb{S}_{\pa(d+1)}\times [0,\infty)^{\{d+1\}}$  defined by 
\[
\sigma( U )= \Lambda_{V_+}\pp{ \pp{  \frac{\mbf{y}_{\pa(d+1)}} {\|\mbf{y}_{\pa(d+1)}\|_\infty}, \frac{ y_{d+1}}{\|\mbf{y}_{\pa(d+1)}\|_\infty} }  \in U  \ \bigg| \
 \|\mbf{y}_{\pa(d+1)}\|_\infty>1   }
\]
for Borel $U$ on $\bb{S}_{\pa(d+1)}\times [0,\infty)^{\{d+1\}}$.
If  $(\mbf{S}, Z )$ is a random vector following the distribution $\sigma$ above, by the noise outsourcing lemma (e.g., \cite[Proposition 8.20]{kallenberg:2021:foundations}), there exists a measurable function $g: \bb{S}_{\pa(d+1)}\times [0,1] \mapsto [0,\infty)$, such that 
\begin{equation}\label{eq:outsource}
   (\mbf{S}, Z )\EqD \pp{\mbf{S}, g(\mbf{S},\theta) },
\end{equation} 
where $\theta$ is a Uniform(0,1) random variable independent of $\mbf{S}$.

We now proceed to check $\cl{L}(\mbf{Y}_{V_+})=\Lambda_{V_+}$. Decompose 
\begin{align}
 \Lambda_{V_+}(\cdot)&=\Lambda_{V_+}\pp{\mbf{y}_{V_+}\in \cdot \ , \ \mbf{y}_{\pa(d+1)}=\mbf{0}_{\pa(d+1)} } + \Lambda_{V_+}\pp{\mbf{y}_{V_+}\in \cdot \ , \ \mbf{y}_{\pa(d+1)}\neq \mbf{0}_{\pa(d+1)} } \notag \\
&=:  \Lambda^{(1)}_{V_+}\pp{\cdot} + \Lambda^{(2)}_{V_+}\pp{\cdot},\label{eq:Lambda decomp 1 2}
\end{align} 
and  $\mu_+=\mu_+^{(1)}+\mu_+^{(2)}$  with the two measures $\mu_+^{(1)}$ and $\mu_+^{(2)}$  defined in an analogous fashion as $\Lambda^{(1)}_{{V_+}}$ and $\Lambda^{(2)}_{{V_+}}$, respectively.  The rest of the proof aims to  show  $\mu_+^{(i)}\left(\mbf{Y}_{V_+}\in  B\right)=\Lambda^{(i)}_{V_+}(B)$, $i=1,2$, for any Borel $B\subset \bb{E}_{V_+}$,  which finishes the proof. 

Note that Proposition \ref{Pro:alt e cond indep} implies that  $\Lambda^{(1)}_{V_+}( y_{d+1}>0,\  \mbf{y}_{V_0} \neq \mbf{0}_{V_0} )=0$, where $$V_0:=V\setminus \pa(d+1).$$
 Using argument similar to that for the case $\pa(d+1)=\emptyset$ above, it can be verified that for any $B(\mbf{x})\subset \bb{E}_{V+}$ of the form $B(\mbf{x})=\{\mbf{y}_{V_+}\in \bb{E}_{V+}:\ y_v\ge x_v,\ v\in V_+ \}$, $\mbf{x}=(x_v)_{v\in V_+}\in \bb{E}_{V_+}$, one has for $\mbf{x}_{\pa(d+1)}=\mbf{0}_{\pa(d+1)}$ that
\begin{align*}
    &\Lambda^{(1)}_{V_+}(B(\mbf{x}))=\mu^{(1)}_+(\mbf{Y}_{V_+}\in B(\mbf{x}) ) \\=&\begin{cases}
    0     & \text{ if } \mbf{x}_{V_0}\neq \mbf{0}_{V_0}  \text{ and } x_{d+1}>0,\\
        r_{d+1} x_{d+1}^{-\alpha}     & \text{ if } \mbf{x}_{V_0}=\mbf{0}_{V_0}  \text{ and } x_{d+1}>0 , \\
    \Lambda_V(\mbf{y}_{w} \ge \mbf{x}_{w},\, w\in V_0,\ \mbf{y}_{\pa(d+1)}=\mbf{0}_{\pa(d+1)})            & \text{ if }\mbf{x}_{V_0}\neq \mbf{0}_{V_0} \text{ and } x_{d+1}=0,
\end{cases}
\end{align*}
 and both are $0$ for $\mbf{x}_{\pa(d+1)}\neq \mbf{0}_{\pa(d+1)}$. Then by a measure-determining argument, we infer that the same relation continues to hold if $B(\mbf{x})$ above is replaced by a general Borel subset of $\bb{E}_{V_+}.$

It remains to show that 
\begin{equation}\label{eq:Lambda^(2)=mu^(2)}
\Lambda^{(2)}_{V_+} \pp{B(\mbf{x})}= \mu^{(2)}_+\pp{ \mbf{Y}_{V_+}\in B(\mbf{x}) }
\end{equation}
for any  $B(\mbf{x})$ as above, $\mbf{x}\in  \bb{E}_{V_+}$. By the homogeneity property of $\Lambda^{(2)}_{V_+}$ and $\mu^{(2)}_+(\mbf{Y}_{V_+}\in \cdot)$ (restricted to $\bb{E}_{V_+}$),  it suffices to show for every $u\in \pa(d+1)$, the relation \eqref{eq:Lambda^(2)=mu^(2)} holds with $\mbf{x} \in  \bb{E}_{V_+}$ such that $x_u=1$.  From now on, fix such an $u\in \pa(d+1)$ and $\mbf{x}=(x_1,\ldots,x_{d+1})\in \bb{E}_{V_+}$ with $x_u=1$.
Furthermore,     we have  
\begin{equation}\label{eq:Lambda^2 mu_+}
\Lambda^{(2)}_{V_+}(y_u\ge 1)= \Lambda_{V_+}(y_u\ge 1)=\mu(Y_u\ge 1)=\mu_+(Y_u\ge 1)=\mu_+^{(2)}(Y_u\ge 1),
\end{equation}
where the first equality is due to \eqref{eq:Lambda decomp 1 2}, the second    due to \eqref{eq:induction assump},   the third  due to \eqref{eq:mu_+ marg}, and the last one follows from the definition of $\mu_+^{(2)}$.  
So taking into account \eqref{eq:Lambda^2 mu_+}, in order to show \eqref{eq:Lambda^(2)=mu^(2)} under  the restriction $x_u=1$,  it suffices to show 
\begin{equation}\label{eq:y Y eqd}
 \pp{ \mbf{y}^{(u)}_{V_0},\mbf{y}^{(u)}_{\pa(d+1)},y^{(u)}_{d+1} } \EqD   \pp{ \mbf{Y}^{(u)}_{V_0},\mbf{Y}^{(u)}_{\pa(d+1)},Y^{(u)}_{d+1} },
\end{equation}
where  $ \mbf{y}^{(u)}_{V_+}:=\pp{ \mbf{y}^{(u)}_{V_0},\mbf{y}^{(u)}_{\pa(d+1)},y^{(u)}_{d+1} }$ is a random vector following the distribution   $\Lambda^{(2)}_{V_+} (\ \cdot  \mid y_u\ge 1 )= \Lambda_{V_+}( \ \cdot  \mid y_u\ge 1 )$, 
and $ \mbf{Y}^{(u)}_{V_+}:=\pp{ \mbf{Y}^{(u)}_{V_0},\mbf{Y}^{(u)}_{\pa(d+1)},Y^{(u)}_{d+1} }$ is a random vector following the distribution   $\mu_+^{(2)} (\ \cdot  \mid Y_u\ge 1 )=\mu_+(\ \cdot  \mid Y_u\ge 1 )$.
 
 Next,
in view of the conditional independence relation \eqref{eq:d+1 cond indep} and Proposition \ref{Pro:alt e cond indep}, we have the  conditional independence  relation
\begin{equation}\label{eq:cond indep y}
    y^{(u)}_{d+1} \perp \mbf{y}^{(u)}_{V_0} \mid \mbf{y}^{(u)}_{\pa(d+1)}.
\end{equation}
On the other hand,     $Y_u\ge 1$ implies $\eta_v>0$ for some $v\in \An(u)$, and hence $\eta_{d+1}=0$. So from \eqref{eq:f d+1 def gen},  on $\{Y_u\ge 1\}$
we have
\begin{equation}\label{eq:Y_d+1 def}
Y_{d+1}= \|\mbf{Y}_{\pa(d+1)}\|_\infty   g( \mbf{Y}_{\pa(d+1)}/\|\mbf{Y}_{\pa(d+1)}\|_\infty, \theta_{d+1}   ).
\end{equation}
Since by construction, under $\mu_+(\ \cdot \mid Y_u \ge 1)$, the random variable $\theta_{d+1}$ is independent of $\pp{\mbf{Y}_{V_0}^{(u)}, \mbf{Y}_{\pa(d+1)}^{(u)}}$ as a function of $\pp{\sbf{\eta}_V, \sbf{\theta}_V}$, we also have the conditional independence relation 
\begin{equation}\label{eq:cond indep Y}
Y_{d+1}^{(u)}  \perp \mbf{Y}_{V_0}^{(u)}  \mid \mbf{Y}_{\pa(d+1)}^{(u)}. 
\end{equation}

In addition, it can be inferred  from  the induction assumption \eqref{eq:induction assump} and relation \eqref{eq:mu_+ marg} that
\begin{equation}\label{eq:y_V Y_V}
 \pp{ \mbf{y}^{(u)}_{V_0},\mbf{y}^{(u)}_{\pa(d+1)}} \EqD   \pp{ \mbf{Y}^{(u)}_{V_0},\mbf{Y}^{(u)}_{\pa(d+1)}}. 
\end{equation}
So combining \eqref{eq:cond indep y}, \eqref{eq:cond indep Y} and \eqref{eq:y_V Y_V}, in order to show \eqref{eq:y Y eqd}, it suffices to show  $\pp{ \mbf{y}^{(u)}_{\pa(d+1)},y^{(u)}_{d+1} } \EqD   \pp{ \mbf{Y}^{(u)}_{\pa(d+1)},Y^{(u)}_{d+1} }$, that is,  
 \begin{align}\label{eq:mu+ Lambda final}
  \Lambda_{V_+} \pp{ \pp{y_{d+1},\mbf{y}_{\pa(d+1)}} \in \cdot  \mid y_u\ge 1 } = \mu_+ \pp{ \pp{Y_{d+1}, \mbf{Y}_{\pa(d+1)}}\in  \cdot \mid  Y_u\ge 1 }. 
\end{align}
To do so, we first make the following claim:
\begin{align}  \label{eq:key equal d}
 \quad & \Lambda_{V_+}\pp{\pp{   \| \mbf{y}_{\pa(d+1)}\|_\infty, \frac{\mbf{y}_{\pa(d+1)}}{\| \mbf{y}_{\pa(d+1)}\|_\infty}, \frac{y_{d+1}}{\|\mbf{y}_{\pa(d+1)}\|_\infty
  }}\in \cdot   \ \biggr |   \|\mbf{y}_{\pa(d+1)}\|_\infty \ge 1 } & \notag \\
 \quad  =&  \mu_+ \pp{\pp{   \| \mbf{Y}_{\pa(d+1)}\|_\infty, \frac{\mbf{Y}_{\pa(d+1)}}{\| \mbf{Y}_{\pa(d+1)}\|_\infty}, \frac{Y_{d+1}}{\|\mbf{Y}_{\pa(d+1)}\|_\infty
  }}\in \cdot \ \biggr|  \|\mbf{Y}_{\pa(d+1)}\|_\infty \ge 1 }.
\end{align}
Indeed,   we point out that under the probability measure $\mu_+\pp{\ \cdot \mid  \| \mbf{Y}_{\pa(d+1)}\|_\infty \ge 1  }$, the random variable $ \| \mbf{Y}_{\pa(d+1)}\|_\infty$ is independent of $\mbf{Y}_{\pa(d+1)}/ \| \mbf{Y}_{\pa(d+1)}\|_\infty$ and $Y_{d+1}/\|\mbf{Y}_{\pa(d+1)}\|_\infty$. This  follows from the homogeneity of $\mu_+\pp{ \mbf{Y}_{V_+} \in \cdot }$ as mentioned in \eqref{eq:mu_+ homo}; 
see, e.g., the proof of \cite[Theorem B.2.5]{kulik2020heavy}.  A similar independence conclusion also holds for the $\mbf{y}$-random variables under $ \Lambda_{V_+}(\ \cdot  \mid \|\mbf{y}_{\pa(d+1)}\|_\infty\ge 1 )$ in \eqref{eq:key equal d}. Then   \eqref{eq:key equal d} follows from these independence relations,  \eqref{eq:outsource} and \eqref{eq:Y_d+1 def}.

{To conclude \eqref{eq:mu+ Lambda final} from \eqref{eq:key equal d}, let
$B\subset [0,\infty)^{\pa(d+1)\cup \{d+1\}}$ be Borel. Then
\begin{align}
\label{eq:pf_markov_to_escm_0}
 &\Lambda_{V_+} \pp{ \pp{y_{d+1},\mbf{y}_{\pa(d+1)}} \in B  \mid y_u\ge 1 } \notag\\
 =&   \Lambda_{V_+} \pp{ \pp{y_{d+1},\mbf{y}_{\pa(d+1)}} \in B,\, y_u\ge 1
 \mid \|\mbf{y}_{\pa(d+1)}\|_\infty \ge 1  }
 \frac{\Lambda_{V_+} \pp{\|\mbf{y}_{\pa(d+1)}\|_\infty \ge 1}}
 {\Lambda_{V_+}(y_u\ge 1)},
\end{align}
where we used $\{y_u\ge 1\}\subset \{\|\mbf{y}_{\pa(d+1)}\|_\infty \ge 1\}$, since
$u\in \pa(d+1)$. By \eqref{eq:key equal d},
\begin{align}
\label{eq:pf_markov_to_escm_1}
&\Lambda_{V_+} \pp{ \pp{y_{d+1},\mbf{y}_{\pa(d+1)}} \in B,\, y_u\ge 1
 \mid \|\mbf{y}_{\pa(d+1)}\|_\infty \ge 1  }\notag\\
=&\mu_{+} \pp{ \pp{Y_{d+1},\mbf{Y}_{\pa(d+1)}} \in B,\, Y_u\ge 1
 \mid \|\mbf{Y}_{\pa(d+1)}\|_\infty \ge 1  } .
\end{align}
Moreover, for every Borel set $U\subset \bb{E}_{\pa(d+1)}$, \eqref{eq:induction assump} and \eqref{eq:mu_+ marg} imply
\[
\Lambda_{V_+}\pp{ \mbf{y}_{\pa(d+1)}\in U  }
=
\mu_+\pp{ \mbf{Y}_{\pa(d+1)} \in U } .
\]
Hence
\begin{equation}
\label{eq:pf_markov_to_escm_2}
 \frac{\Lambda_{V_+} \pp{\|\mbf{y}_{\pa(d+1)}\|_\infty \ge 1}}
 {\Lambda_{V_+}(y_u\ge 1)}
 =
 \frac{\mu_{+} \pp{\|\mbf{Y}_{\pa(d+1)}\|_\infty \ge 1}}
 {\mu_{+}(Y_u\ge 1)} .
 \end{equation}
Combining \eqref{eq:pf_markov_to_escm_0}-\eqref{eq:pf_markov_to_escm_2} above, and using
$\{Y_u\ge 1\}\subset \{\|\mbf{Y}_{\pa(d+1)}\|_\infty \ge 1\}$, yields
\eqref{eq:mu+ Lambda final}.}

\medskip

\noindent $\bullet$ The case $\pa(d+1)= V$ is  similar to the previous case once obvious simplifications due  to $V_0=\emptyset$    are applied.     We omit the details.

\subsection{Proof of Proposition \ref{Pro:causal asym}}

For the first claim, recall first by the nature of the activation variables,  if $\eta_u>0$, then we have $\eta_w=0$ for all $w\neq u$.   Recall also $Y_v=F_{\cl{A}(v)}\pp{ \sbf{\eta}_{\mrm{An}(v)}, \sbf{\theta}_{\mrm{An}(v)} }$, where $F_{\cl{A}(v)}\pp{ \sbf{0}_{\mrm{An}(v)}, \sbf{\theta}_{\mrm{An}(v)} }=0$. Since also $a_u>0$ by   Assumption \ref{ass:eta act}, we have
\begin{align*}
\mu\left(Y_u>0,Y_v=0\right)\ge  \mu\pp{  a_u \eta_u>0,  \sbf{\eta}_{\mrm{An}(v)}= \mbf{0}_{\mrm{An}(v)}}= \mu\pp{\eta_u>0}>0.
\end{align*}

To show the second claim, 
suppose a directed path from $u$ to $v$ is given by $(u_0:=u,u_1,\ldots,u_s:=v)$, $s\in \bb{Z}_+$.  Since $u_i\in \pa\left(u_{i+1}\right)$, by Assumption \ref{ass:weak asym} and \eqref{eq:general eSCM},    $\mu(Y_{u_{i}}>0, Y_{u_{i+1}}=0)=0$, $i\in\{0,\ldots,s-1\}$.  Since $Y_u>0, Y_v=0$ implies   $Y_{u_{i}}>0, Y_{u_{i+1}}=0$ for some $i\in \{0,\ldots,s-1\}$, applying  the union bound, one has
\[
\mu(Y_u>0, Y_v=0)    \le \sum_{i=0}^{s-1}\mu(Y_{u_{i}}>0, Y_{u_{i+1}}=0)=0.
\]


\subsection{Proof of Proposition \ref{pro:char ass}}


  We have by assumption that $h_v(\mbf{Y}_{\pa(v)},\theta_v)\ge d_v\|\mbf{Y}_{\pa(v)}\| $ $\mu$-a.e.\ for some  constant $d_v>0$, $v\in V$. Since  the norm $\|\cdot\|$ is equivalent to $\|\cdot\|_1$, we have for each $v\in V$, there exists a positive constant $c_v>0$, such that 
\begin{equation*}
 Y_v = a_v \eta_v +h_v\pp{\mbf{Y}_{\pa(v)}, \theta_v} \ge a_v\eta_v + c_v \sum_{w'\in \pa(v)} Y_{w'}\ge c_v Y_w \quad \text{for all } w\in \pa(v),\quad \text{$\mu$-a.e.}.
\end{equation*}
Suppose now $v\in V$ and $u\in \an(v)$. Then there exists a directed path $u_0=u\rightarrow u_1\rightarrow \cdots \rightarrow u_s=v$, $s\in \bb{Z}_+$. Iterating the preceding displayed relation along this path gives
\[
Y_v\ge  \pp{\prod_{i=1}^{s} c_{u_i}} Y_u  \quad \text{$\mu$-a.e.}.
\]
Setting $c_{uv}:=\prod_{i=1}^{s} c_{u_i}>0$, we obtain
\[
\Lambda_{\{u,v\}}\pp{y_v<c_{uv}y_u}=\mu\pp{Y_v<c_{uv}Y_u,\ Y_u>0}=0,
\]
which is Assumption \ref{ass:asym}.

\subsection{Estimate of angular support interval}
\label{sec:angular_support_interval}

To make use of AAC $\tau(u,v)$ as described in Section \ref{Sec:stats caus dir} for inferring causal direction, one needs to estimate the angular support interval $[a,b]$. For such a purpose, we need to step back from the limit eSCM $\mbf{Y}$  to the distributional property of the pre-limit data $\mbf{X}$. In particular, one needs a second-order condition (with respect to the first order limit $\cl{L}(\mbf{Y})$) which, roughly speaking, describes a contrast between the radial tail within the angular support interval $[a,b]$ and the one outside  $[a,b]$. 

\begin{definition}[Second-Order Condition  $\cl{SO}(\rho)$.]\label{Def:second order}
  Let $(X_1,X_2)$ be a MRV random vector taking value in $\bb{E}_2$ satisfying \eqref{eq:pareto marginal} and \eqref{eq:Lambda}, which has an angular support interval $[a,b]\subset  [0,1]$. We say $(X_1,X_2)$ satisfies      $\cl{SO}(\rho)$, with $\rho>0$, if the following holds:     For any   Borel $B\subset [0,1]\setminus [a,b] $  whose closure $\overline{B}\cap [a,b]=\emptyset$,  we have  
  \begin{equation}\label{eq:second order}
  \Prt\pp{ W \in B  \mid R>t }=O(\Prt(R>t)^{\rho})
  \end{equation}
  as $t\rightarrow\infty$, where $(W,R):=(X_1/(X_1+X_2), X_1+X_2)$.
\end{definition}
By monotonicity of the conditional probability in \eqref{eq:second order},  it suffices to consider $B$ of the form $B=[0, a-\epsilon )\cup (b+\epsilon ,1]$, $\epsilon>0$, where an interval $[s,t)$ or $(s,t]$ is understood as empty if $s>t$.  Here, the constant hidden behind the $O(\cdot)$ notation may depend on $B$ chosen. 

The condition $\cl{SO}(\rho)$ can be related to the hidden regular variation condition  on the cone $ [0,\infty)^2\setminus \bb{C}_{a,b} $, where  $\bb{C}_{a,b}:=\{ (x_1,x_2)\in  [0,\infty)^2  : \  a(x_1+x_2) \le x_1 \le b (x_1+x_2)   \}$ is the forbidden zone  \cite{resnick2024art}.  
 Recall under MRV of $(X_1,X_2)$ on $\bb{E}_2$ as described in Definition \ref{Def:second order}, we have the vague convergence $
 \Prt\left(   W  \in  \cdot   \mid  R>t  \right) \overset{v}{\rightarrow}  \Lambda_{\{1,2\}}( y_1/(y_1+y_2)\in \cdot  \mid  y_1+y_2>1 )
$ as $t\rightarrow\infty$, where $\Lambda_{\{1,2\}}$ is the exponent measure of $(X_1,X_2)$.
On the other hand, the condition $\cl{SO}(\rho)$ can be related to the hidden regular variation condition  on the cone outside the angle range $[a,b]$; see  e.g., \cite{resnick2024art} for more details. In particular, consider the case where the law $(X_1,X_2)$ is MRV on $\bb{E}_2\setminus \bb{C}_{a,b} $ in the sense of the following:  There exists a measure $\Lambda_0$ on the Borel $\sigma$-field of $[0,\infty)^2 \setminus \bb{C}_{a,b} $ that is finite on any Borel subset of $[0,\infty)^2\setminus \bb{C}_{a,b}$  and separated from  $\bb{C}_{a,b}$,  such that  $\lim_{t\rightarrow\infty}t\Prt(  (X_1,X_2) \in d_0(t) A)=\Lambda_{0}(A)$ for  any Borel $A\subset [0,\infty)^2\setminus \bb{C}_{a,b}$  with  $\Lambda_0(\partial A)=0$, and the measurable function  $d_0: (0,\infty)\mapsto(0,\infty)$  is regularly varying with index $1/[(1+\wt{\rho})\alpha]$, $\wt{\rho}>0$, as $t\rightarrow\infty$.  Note that $\lim_{t\rightarrow\infty} t^{1/\alpha} /d_0(t) =\infty$, where $t^{1/\alpha}$ corresponds to the normalization in the MRV condition \eqref{eq:Lambda} on the full space $\bb{E}_2$.
Then the $\cl{SO}(\rho)$ condition is satisfied with any $\rho\in (0,\wt{\rho})$ in view of the Potter's bound (e.g., \cite[Theorem 1.5.6]{bingham:1989:regular}), or one may take $\rho=\wt{\rho}$ if $d_0(t)\sim c t^{1/[\alpha(1+\wt{\rho})]}$ readily for some constant $c>0$. On the other hand, the $\cl{SO}(\rho)$ condition also covers the situations beyond hidden regular variation such as  $\Prt((X_1,X_2)\notin \bb{C}_{a,b})=0$, for which one may take a $\rho>0$ arbitrarily large.

Now we formulate an estimator of the angular support interval $[a,b]$, which covers the one employed in Section \ref{Sec:stats caus dir} as a special case.  Let $\Delta=\{(s,t)\in [0,1]^2,\  s\le t\}$.  Consider a measurable function $d: [0,1]\times \Delta \mapsto [0,1] $ which serves as a   distance from the point $w\in [0,1]$ to the interval $[s,t]$, $0\le s\le t\le 1$. We assume that $d(w,s,t)$ is continuous in $w\in [0,1]$ for each $(s,t)\in \Delta$ fixed,  and it is also continuous in $(s,t)\in \Delta$  for each $w\in [0,1]$ fixed. Furthermore, suppose that $d(w,s,t)>0$ if and only if  $w\notin [s,t]$, and that it satisfies the monotonicity property $d(w,s,t)\ge  d(w,s',t')$ if $s'\le s$ and $t'\ge t$.  
Consider also a  continuous function $L:[1,\infty)\mapsto (0,\infty)$ which will play the role of weighting the observations according to their  radial locations.  
Let $(X_{i,1},X_{i,2})_{i=1,\ldots,n}$ be i.i.d.\ observations of $(X_1,X_2)$ in Definition \ref{Def:second order}. Order them as random vectors $(X_{(1),1},X_{(1),2}),\ldots,(X_{(n),1},X_{(n),2})$, so that $R_{(1)}\ge  \ldots \ge  R_{(n)}$, $R_{(i)}:= X_{(i),1}+X_{(i),2}$. Set $W_{(i)}=X_{(i),1}/R_{(i)}$. Here and below, we often suppress a notation's dependence on sample size $n$ for simplicity.
Define for $1\le k\le n$ that
$$
D_k(s,t)=\frac{1}{k} \sum_{i=1}^k  d(W_{(i)},s,t) L(R_{(i)}/ R_{(k)}),
$$
and set the objective function 
\begin{equation}\label{eq:g_n form gamma}
 g_n(s,t)=   t-s   +  \lambda k^\gamma   D_k(s,t),
\end{equation}
where $\lambda\in (0,\infty)$ and $\gamma\in (0,\infty)$ are fixed parameters.  Note that $g_n(s,t)$ is a continuous function on $\Delta$.    
The asymptotic theory below is formulated for general choices of $d$, $L$, $\lambda$, $\gamma$, while empirically we found that the specific choices described in Section \ref{Sec:stats caus dir} seem to work reasonably well.

The estimator of $a$ and $b$ is formulated as follows
\begin{equation}\label{eq:a,b est main}
\pp{\wh{a}_n,\wh{b}_n
}=\argmin_{ (s,t)\in \Delta} g_n(s, t),
\end{equation}
where the operation $\argmin$ is understood as selecting a measurable representative of the minimizer if the latter is not unique. 


Now we present a  consistency result below.
We shall work with an intermediate  sequence $k=k_n\in \bb{Z}_+$ that tends to $\infty$ with $k_n=o(n)$, for which we suppress its dependence on sample size $n$ for simplicity.

\begin{theorem}\label{Thm:consistency of a,b}
Consider the setup of  Definition \ref{Def:second order}, including  the  second order condition $\cl{SO}(\rho)$, $\rho>0$,  as well as the assumptions described above for $d(w,s,t)$ and $L(r)$.  Assume in addition that for some constants $\delta\in (0,\alpha)$ and $C>0$, we have $L(r)\le C r^{\delta}$, $r\ge 1$.  Then the estimator in \eqref{eq:a,b est main} is consistent: $\wh{a}_n\ConvP a$ and $\wh{b}_n\ConvP b$ as $n\rightarrow\infty$, when $k=k_n\rightarrow\infty$ and $k=o(n^{\rho/(\gamma+\rho)})$ as $n\rightarrow\infty$, where $\gamma$ is as in \eqref{eq:g_n form gamma}.   
\end{theorem}
We point out that it is possible to relax the assumption $L(r)\le C r^{\delta}$, with $\delta<\alpha$, to allow, e.g., $L(r)=r^{\delta}$ with $\delta>\alpha$. This requires a more involved analysis which  we do not pursue here. 

The proof of Theorem \ref{Thm:consistency of a,b} follows a similar strategy as the proof of \cite[Theorem 5]{WangResnick2024}. We first prepare a lemma about the $D_k(s,t)$ term in the objective function $g_n(s,t)$.  
\begin{lemma}\label{Lem:D behavior}
    Under the assumptions of  Theorem \ref{Thm:consistency of a,b}, except that here $k$ is only required to satisfy $k\rightarrow\infty$ and $k=o(n)$,  we have the following asymptotic behaviors of $D_k(s,t)$.
    For general $0\le s\le t\le 1$, we have 
    \begin{equation}\label{eq:D_k(s,t) limit}
    D_k(s,t)\ConvP  \int_{[0,1]} d(w,s,t) S(dw) \int_1^\infty  L(r) \bar{\nu}_{\alpha}(dr) 
    \end{equation} 
    as $n\rightarrow\infty$, where $S$ is the angular measure and $\bar{\nu}_{\alpha}(dr)=\alpha r^{-\alpha-1}dr$.
    If, in addition, $B:=[0,1]\setminus[s,t]$ satisfies $\overline{B}\cap[a,b]=\emptyset$
(which in particular implies $[a,b]\subset[s,t]$), then
\begin{equation}\label{eq:D_k(s,t) 2nd limit}
  D_k(s,t)=O_p\pp{(k/n)^\rho}
\end{equation}
as $n\rightarrow\infty$.
\end{lemma}
\begin{proof}[Proof of Lemma \ref{Lem:D behavior}]
Suppose $d(t)>0$ satisfies $\lim_{t\rightarrow\infty} t\Prt\pp{ R>d(t)}=1$; in fact $d(t)\sim t^{1/\alpha} \Lambda_{\{1,2\}}\left(y_1+y_2\ge 1\right)^{1/\alpha}$ as $t\rightarrow\infty$ under the assumption.
First, recall a
well-known approximation 
\begin{equation}\label{eq:order lln}
 \frac{R_{(k)}}{d(n/k)}\ConvP 1   
\end{equation}
as $n\rightarrow\infty$; see, e.g., \cite[Eq.\ (4.17)]{resnick2007heavy}. Leveraging \eqref{eq:order lln},
it follows from an argument similar to that for \cite[Eq.\ (9.37)]{resnick2007heavy}
that 
\begin{equation}\label{eq:lln rm}
 \frac{1}{k}\sum_{i=1}^n \delta_{\pp{W_{(i)}, R_{(i)}/R_{(k)}}} \ConvD  S\times \bar{\nu}_{\alpha}, 
\end{equation}
where $\ConvD$ is understood as weak convergence of random measures on $[0,1]\times (0,\infty)$ under the vague topology (here, subsets of $(0,\infty)$ separated from the origin is considered bounded); see, e.g., \cite[Chapter 9]{kulik2020heavy}). Assume for now that $L$ is bounded. Note also that $\nu_\alpha$ is atomless. So one can apply \cite[Lemma 23.17]{kallenberg:2021:foundations} by integrating the function $d(w,s,t)L(r)\mbf{1}_{\{r\ge 1\}}$, whose discontinuity set is of zero $S\times \nu_{\alpha}$-measure, with respect to the left-hand side measure in \eqref{eq:lln rm} to reach the first conclusion. If $L$ is unbounded,  introduce the truncation $L(r)=L(r) \mbf{1}_{\{r\le M\}} +  L(r)\mbf{1}_{\{r> M\}}$, $M>0$. The desirable conclusion is obtained by the same argument  applied to the first term with letting $n\rightarrow\infty$ first, and then  $M\rightarrow\infty$, given that one can show 
\begin{equation}\label{eq:remainder tri approx}
\lim_{M\rightarrow\infty}\limsup_{n\rightarrow\infty} \Prt  \pp{\frac{1}{k} \sum_{i=1}^n (R_{(i)}/ R_{(k)})^{\delta} \mbf{1}_{\{ R_{(i)}/ R_{(k)} >M\} }>\epsilon}=0
\end{equation}
for any $\epsilon>0$, where we have   applied the assumption $L(r)\le C r^{\delta}$, $\delta\in (0,\alpha)$, and the fact that $d(w,s,t)\le 1$. To do so,  first by \eqref{eq:order lln}, on an event $\Omega_n$ whose probability tends to $1$ as $n\rightarrow\infty$,
one has $R_{(k)}\ge d(n/k)/2$, and thus by monotonicity we have
\begin{equation}\label{eq:D_k^*}
\frac{1}{k} \sum_{i=1}^n (R_{(i)}/ R_{(k)})^{\delta} \mbf{1}_{\{ R_{(i)}/ R_{(k)} >M\} }\le\frac{1}{k} \sum_{i=1}^n \pp{\frac{ R_{i}}{ d(n/k)/2}}^{\delta} \mbf{1}_{\{ R_{i}/ (d(n/k)/2) >M\} }=:D_k^*. 
\end{equation}
on $\Omega_n$.  Let $$\pp{R,W}\EqD \pp{R_i=(X_{i,1}+X_{i,2}),W_i=X_{i,1}/(X_{i,1}+X_{i,2})}.$$
Applying \cite[Proposition 1.4.6]{kulik2020heavy}, one has
\begin{align*}
\E D_k^* = & 2^{\delta} \frac{n}{k} d(n/k)^{-\delta} \E\pb{R^\delta  \mbf{1}_{\{ R  >M d(n/k)/2\}} } 
\\\le & C \frac{n}{k} d(n/k)^{-\delta}  (M d(n/k)/2)^{\delta} \Prt\pp{ R  >M d(n/k)/2} \le C M^{\delta-\alpha},  
\end{align*}
where we have used the fact that $(n/k)\Prt(R> M d(n/k)/2 )\le C (M/2)^{-\alpha}$, and the value of the constant $C>0$ may change from one expression to another, although it does not depend on $n$ or $M$.   Therefore, we have $\lim_M\limsup_{n}\E D_k^* = 0$, which together with $\lim_{n} \Prt\pp{\Omega_n}=1$ implies \eqref{eq:remainder tri approx}.
We have thus finished the proof of the first claim. 

For the second claim,  first based on the $\cl{SO}(\rho)$ condition,  we infer that  
\begin{equation}\label{eq:SO imp}
  \Prt\pp{ R>r, W\in [s,t]^c }\le C r^{-(1+\rho)\alpha},  \quad r>0,  
\end{equation}
where the constant $C>0$   does not depend on $r$. Next, using a similar argument as that around \eqref{eq:D_k^*} as well as the fact that $d(w,s,t)\le \mbf{1}_{\pc{w\in [s,t]^c}}$, it suffices to show
\begin{equation}\label{eq:D_k^*(s,t)}
  D_k^*(s,t):= \frac{1}{k} \sum_{i=1}^n \pp{\frac{ R_{i}}{ d(n/k)/2}}^{\delta} \mbf{1}_{\{ R_{i}  >d(n/k)/2,\  W_i\in [s,t]^c\} } =O_p\pp{ \pp{\frac{k}{n}}^{\rho}}.
\end{equation}
Indeed, by Fubini,  \eqref{eq:SO imp} and $\delta\in (0,\alpha)$, one has
\begin{align*}
\E D_k^*(s,t) &\le \frac{C n}{k d(n/k)^{\delta}} \E\pb{  \int_{0}^R r^{\delta -1} dr  \mbf{1}_{\{ R >  d(n/k)/2 ,\  W\in [s,t]^c\} }   }\\
&= \frac{C n}{k d(n/k)^{\delta}} \int_0^\infty r^{\delta -1} dr \Prt\pp{     { R > r \vee \pp{  d(n/k)/2}  ,\  W\in [s,t]^c }   }\\
&\le \frac{C n}{k d(n/k)^{\delta}} \pp{\int_0^{d(n/k)/2} r^{\delta -1}  d(n/k)^{-(1+\rho)\alpha}   dr   + \int_{d(n/k)/2}^\infty r^{\delta -1-(1+\rho)\alpha} dr}\\
&\le   \frac{C n}{k d(n/k)^{\delta}}   \cdot  d(n/k)^{\delta-(1+\rho)\alpha}\le  C \pp{ \frac{k}{n}}^{\rho},
\end{align*}
where in the last step we have used $d(n/k) \sim C (n/k)^{1/\alpha}$ as $n\rightarrow\infty$. Therefore, the relation \eqref{eq:D_k^*(s,t)} follows, and so does the second claim.  
\end{proof}

\begin{proof}[Proof of Theorem \ref{Thm:consistency of a,b}]
Note that under the assumption of the exponent measure $\Lambda_{\{1,2\}}$ of
$(X_1,X_2)$ having non-vanishing marginals, necessarily $a<1$ and $b>0$, while
it is possible for $a=0$ or $b=1$.

Throughout, fix $\epsilon\in(0,1)$ with $\epsilon< \min (4(1-a), 4b)$. We claim that it is enough to show
\begin{equation}\label{eq:cons goal}
 \lim_{n\rightarrow\infty} \Prt\pp{ \inf_{(s,t) \in \Delta_\epsilon } g_n(s,t)
 > \inf_{(s,t)\in\Delta_\epsilon^c } g_n(s,t)+\epsilon/4 } =1,
\end{equation}
where $$\Delta_\epsilon=\pc{(s,t)\in \Delta: \ |s-a|> \epsilon \text{ or }
|t-b|>\epsilon }$$ and $\Delta_\epsilon^c$ is its complement in
$\Delta=\pc{(s,t):\ 0\le s\le t\le 1}$. Indeed, $g_n$ is continuous on the
compact set $\Delta$, so its minimum is attained; on the event in
\eqref{eq:cons goal} every minimizer lies in $\Delta_\epsilon^c$, whence
$\pc{|\wh{a}_n-a|\le \epsilon} \cap \pc{|\wh{b}_n-b|\le \epsilon}$.

 Set
\[
  s_0=(a-\epsilon/8)\vee 0, \qquad t_0=(b+\epsilon/8)\wedge 1 .
\]
Then $(s_0,t_0)\in\Delta_\epsilon^c$  and  
\begin{equation}\label{eq:ref length}
  t_0-s_0\le (b-a)+\epsilon/4 .
\end{equation}
Note that the set $[0,1]\setminus[s_0,t_0]$ has closure disjoint from $[a,b]$. Hence
\eqref{eq:D_k(s,t) 2nd limit} applies to $(s_0,t_0)$, and together with the
assumption $k^\gamma (k/n)^\rho\rightarrow 0$,
\begin{equation}\label{eq:ref vanish}
  \lambda k^\gamma D_k(s_0,t_0)\ConvP 0 \qquad \text{as } n\rightarrow\infty.
\end{equation}

  Recall from the monotonicity assumption on $d$ that
$D_k(s,t)\ge D_k(s',t')$ whenever $s'\le s$ and $t'\ge t$. We divide into the following two cases.

\emph{Case 1: $s\ge a+\epsilon/4$ or $t\le b-\epsilon/4$.} In the first
situation $[s,t]\subset[a+\epsilon/4,1]$, so $D_k(s,t)\ge D_k(a+\epsilon/4,1)$;
in the second $[s,t]\subset[0,b-\epsilon/4]$, so
$D_k(s,t)\ge D_k(0,b-\epsilon/4)$.  By
\eqref{eq:D_k(s,t) limit}, as $n\rightarrow\infty$,
\[
 D_k(a+\epsilon/4,1)\ConvP A_\epsilon>0, \qquad
 D_k(0,b-\epsilon/4)\ConvP B_\epsilon>0,
\]
where $$A_\epsilon=\int_{[0,1]} d(w,a+\epsilon/4,1)\,S(dw)
\int_1^\infty L(r)\,\overline{\nu}_\alpha(dr)$$ and
$$B_\epsilon=\int_{[0,1]} d(w,0,b-\epsilon/4)\,S(dw)
\int_1^\infty L(r)\,\overline{\nu}_\alpha(dr).$$     By these  relations, $t-s\ge 0$, $t_0-s_0\le 1$,   and \eqref{eq:ref vanish}, we have
\begin{align}\label{eq:case1}
 &g_n(s,t)-g_n(s_0,t_0)\notag \ge \\& -1 + \lambda k^\gamma
 \pb{ D_k(a+\epsilon/4,1)\wedge D_k(0,b-\epsilon/4)-D_k(s_0,t_0) }
 \ConvP \infty.
\end{align}

\emph{Case 2: $s< a+\epsilon/4$ and $t> b-\epsilon/4$.} Since
$(s,t)\in\Delta_\epsilon$, either $|s-a|>\epsilon$, which forces $s<a-\epsilon$,
or $|t-b|>\epsilon$, which forces $t>b+\epsilon$. In the first case
$t-s>(b-\epsilon/4)-(a-\epsilon)=(b-a)+3\epsilon/4$, and in the second
$t-s>(b+\epsilon)-(a+\epsilon/4)=(b-a)+3\epsilon/4$. Combining these with
\eqref{eq:ref length}, $D_k(s,t)\ge 0$ and \eqref{eq:ref vanish},
\begin{equation}\label{eq:case2}
 g_n(s,t)-g_n(s_0,t_0)\ \ge\ \frac{\epsilon}{2}
 - \lambda k^\gamma D_k(s_0,t_0) \ConvP \frac{\epsilon}{2}.
\end{equation}

By \eqref{eq:case1} and \eqref{eq:case2}, with probability tending to one
$\inf_{(s,t)\in\Delta_\epsilon} g_n(s,t) > g_n(s_0,t_0)+\epsilon/4$, and since
$(s_0,t_0)\in\Delta_\epsilon^c$, this gives \eqref{eq:cons goal}. The whole
proof is then finished.                                   
\end{proof}


\subsection{Proof of Proposition \ref{Pro:EASE empirical}}

{When the DAG $\mathcal{G}$ is totally disconnected, i.e., contains no directed edges, the EASE algorithm in Algorithm \ref{Alg:EASE} is trivially valid, since every ordering is a causal ordering.
We next state a result adapted from \cite[Proposition 2]{gnecco2021causal}. Proposition \ref{Pro:EASE empirical} then follows directly whenever the DAG $\mathcal{G}$ contains at least one directed edge.}

\begin{lemma}
Let $\cl{G}=(V,E)$ be a DAG with $ V=\{1,\ldots,d\}$ and let $\pp{\tau(u,v)}_{u,v\in V, u\neq v}$ be real coefficients satisfying 
   $u\in \an(v)$ if and only if $\tau(u,v)>0$.
Suppose $\pp{\wh{\tau}(u,v)}_{u,v\in V, u\neq v}$ are estimators of $\pp{\tau(u,v)}_{u,v\in V, u\neq v}$. Let $\wh{\pi}:V\mapsto V$ be a causal order returned by the EASE algorithm in Algorithm \ref{Alg:EASE} when  $\pp{\wh{\tau}(u,v)}_{u,v\in V, u\neq v}$ is supplied as the input. Let $\Pi=\{\pi\}$ be the collection of correct causal orders associated with $\cl{G}$. Then
\begin{align*}
 \Prt\pp{\wh{\pi}\notin \Pi} \le d^2  \bigvee_{(u,v)\in V^2,\ u\neq v} \Prt\pp{ |\wh{\tau}(u,v)-\tau(u,v)|  >  m_\tau  /2},  
\end{align*} 
where $m_\tau=\min\{\tau(u,v):\ u\in \an(v),\ u,v\in V \}$.
\end{lemma}
\begin{proof}
 The proof follows exactly that of \cite[Proposition 2]{gnecco2021causal} in the Supplementary Material of that paper, once at the first displayed formula below (S.21),   the role of ``$1$'' there is replaced  by $m_\tau$, and the role of ``$\eta$'' there is replaced  by $0$.  
\end{proof}

\end{appendix}

\begin{acks}[Acknowledgments]
Shuyang Bai and Fei Fang contributed equally to this work and are co-first authors. Tiandong Wang is the corresponding author. The authors also thank Sebastian Engelke and Johan Segers for helpful discussions. Generative AI tools were used for language editing and proofreading of mathematical arguments.

\end{acks}

\begin{funding}
T. Wang gratefully acknowledges the National Key R\&D Program of China (No. 2025YFA1016503) and the National Natural Science Foundation of China Grant 12301660. 
\end{funding}

\end{document}